\documentclass[oneside, reqno, 11pt, a4paper]{amsart}
\usepackage[australian]{babel}

\AtBeginDocument{
  \DeclareSymbolFont{AMSb}{U}{msb}{m}{n}
  \DeclareSymbolFontAlphabet{\mathbb}{AMSb}
}
\DeclareFontFamily{U}{mathx}{\hyphenchar\font45}
\DeclareFontShape{U}{mathx}{m}{n}{<-> mathx10}{}
\DeclareSymbolFont{mathx}{U}{mathx}{m}{n}
\DeclareMathAccent{\widebar}{0}{mathx}{"73}

\usepackage[dvipsnames]{xcolor}
\usepackage{graphicx}
\usepackage{svg}
\usepackage{tikz}
\usepackage{caption}
\usepackage{subcaption}

\usetikzlibrary{shapes.misc,matrix,fit,positioning,arrows.meta,decorations.markings,decorations.shapes,decorations.pathreplacing,calc,shapes.geometric,arrows,shapes,patterns,patterns.meta,math}
\tikzset{Matrix/.style={matrix of nodes, font=\footnotesize,text height=1pt, text depth=0.5pt, text width=8.5pt, align=center, column sep=0pt, row sep=0pt, nodes in empty cells}}
\usepackage{pgf}
\usepgflibrary{fpu}

\graphicspath{{./}}

\usepackage[a4paper, pdftex, left=2cm, top=2cm, right=2cm, bottom=2cm]{geometry}
\usepackage[final]{pdfpages}
\usepackage{changepage}

\usepackage[foot]{amsaddr}
\numberwithin{equation}{section}

\usepackage{url}
\usepackage{hyperref}
\hypersetup{
    breaklinks=true,
    bookmarksopen=true,
    pdftitle={GMG solvers for HHO}, 
    pdfauthor={Santiago Badia, Jordi Manyer}, 
    pdfsubject={}, 
    colorlinks=true,
    linkcolor=blue,
    citecolor=blue,
    filecolor=blue,
    urlcolor=blue
}

\newcommand{\sect}[1]{Sect.~\ref{#1}}

\usepackage{csquotes}
\usepackage[backend=biber, defernumbers=true, maxbibnames=5, style=numeric-comp, isbn=false, bibencoding=utf8, safeinputenc, url=false, doi=true, giveninits=true]{biblatex}
\usepackage{mathtools}
\usepackage{mathabx}
\usepackage{stmaryrd}
\usepackage{bbm}

\usepackage{float}
\usepackage{framed}
\usepackage{verbatim}
\usepackage{fancyvrb}
\usepackage{booktabs}
\usepackage{longtable}
\usepackage{colortbl}
\usepackage{multirow}
\usepackage{algorithm}
\usepackage{algpseudocode}
\usepackage{cancel}
\usepackage{accents}
\usepackage{mleftright}
\usepackage{thm-restate}

\makeatletter
\algdef{S}[IF]{IfNoThen}[1]{\algorithmicif\ #1}
\makeatother
\usepackage[inline]{enumitem}
\usepackage{siunitx}

\newtheorem{theorem}{Theorem}
\newtheorem{lemma}{Lemma}
\newtheorem{corollary}{Corollary}

\newtheorem{assumption}{Assumption}
\newtheorem{definition}{Definition}
\theoremstyle{remark}
\newtheorem{remark}{Remark}
\theoremstyle{plain}
\usepackage[breakable]{tcolorbox}
\tcbuselibrary{theorems}
\tcbuselibrary{skins}
\tcbset{
	commonstyle/.style={
		theorem style=plain,
		enhanced jigsaw,
		fonttitle=\bfseries,
		fontupper=\itshape,
		halign=justify,
		separator sign=:,
		description delimiters none,
		description font=\bfseries, 
		terminator sign={.\hspace{0.25em}},
		arc=0mm,outer arc=0mm,
		boxrule=0pt,toprule=0pt,bottomrule=0pt,leftrule=0pt,rightrule=0pt,
		titlerule=0pt,toptitle=0pt,bottomtitle=0pt,top=0pt,
		colback=white,coltitle=black,
		boxsep=0pt, bottom=0pt, left=0pt, 
	}
}
\newtcbtheorem[]{myproblem}{Problem}%
{center, commonstyle, fonttitle=\bfseries}{pb}
\newtcbtheorem[]{mydefinition}{Definition}
{center, commonstyle, fonttitle=\bfseries}{pb}
\newtcbtheorem[]{myassumption}{Assumption}
{center, commonstyle, fonttitle=\bfseries}{pb}

\usepackage{amsmath, amsfonts, amssymb, amscd, bm, mathtools}
\mathtoolsset{showonlyrefs}

\newcommand{\identity}{\mathbbm{1}}

\newcommand{\ul}[1]{\underline{#1}}
\newcommand{\ol}[1]{\overline{#1}}
\newcommand{\norm}[2]{\|#2\|_{#1}}
\newcommand{\seminorm}[2]{|#2|_{#1}}
\newcommand{\jump}[1]{\llbracket #1 \rrbracket}
\newcommand{\braket}[1]{\langle #1 \rangle}
\newcommand{\triplevert}[1]{{\left\vert\kern-0.25ex\left\vert\kern-0.25ex\left\vert #1 \right\vert\kern-0.25ex\right\vert\kern-0.25ex\right\vert}}

\newcommand{\Pk}[1]{\mathbb{P}^{#1}}
\newcommand{\IPk}[1]{\mathcal{P}^{#1}}
\newcommand{\lproj}[2]{\pi_{#1}^{#2}}
\newcommand{\rec}[2]{\mathcal{R}_{#1}}

\newcommand{\dom}{\Omega}
\newcommand{\T}{\mathcal{T}}
\newcommand{\F}{\mathcal{F}}

\newcommand{\interior}{{\rm in}}
\newcommand{\boundary}{{\rm bd}}

\newcommand{\Tl}[1]{\T_{#1}}
\newcommand{\Fl}[1]{\F_{#1}}

\newcommand{\Th}{\Tl{\ell}}
\renewcommand{\TH}{\Tl{\ell-1}}
\newcommand{\Fh}{\Fl{\ell}}
\newcommand{\FH}{\Fl{\ell-1}}

\newcommand{\Fhi}{\Fh^{\interior}}
\newcommand{\Fhb}{\Fh^{\boundary}}

\newcommand{\h}{h_{\ell}}
\renewcommand{\H}{h_{\ell-1}}
\renewcommand{\k}{k}
\newcommand{\K}{k}

\newcommand{\Ul}[1]{\ul{U}_{#1}}
\newcommand{\Uli}[1]{U_{#1}}
\newcommand{\Ulb}[1]{M_{#1}}

\newcommand{\Uh}{\Ul{\ell}}
\newcommand{\Uhi}{\Uli{\ell}}
\newcommand{\Uhb}{\Ulb{\ell}}
\newcommand{\UH}{\Ul{\ell-1}}

\newcommand{\UHb}{\Ulb{\ell-1}}

\newcommand{\uh}{\ul{u}_{\ell}}
\newcommand{\vh}{\ul{v}_{\ell}}

\newcommand{\tr}{\gamma}

\newcommand{\trH}{\tr_{\ell-1}}

\newcommand{\isc}[1]{\mathcal{U}_{#1}}
\newcommand{\dualisc}[1]{\mathcal{V}_{#1}}
\newcommand{\isch}{\isc{\ell}}
\newcommand{\iscH}{\isc{\ell-1}}
\newcommand{\dualisch}{\dualisc{\ell}}

\newcommand{\Al}[1]{A_{#1}}
\newcommand{\Ah}{\Al{\ell}}
\newcommand{\eigenmax}{\lambda_{\mathrm{max}}(\Ah)}

\newcommand{\al}[1]{a_{#1}}
\newcommand{\ah}{\al{\ell}}
\newcommand{\aH}{\al{\ell-1}}
\newcommand{\ual}[1]{\ul{a}_{#1}}
\newcommand{\uah}{\ual{\ell}}
\newcommand{\uaH}{\ual{\ell-1}}

\newcommand{\Bl}[1]{B_{#1}}
\newcommand{\Bh}{\Bl{\ell}}
\newcommand{\BH}{\Bl{\ell-1}}

\newcommand{\Il}[1]{I_{#1}}
\newcommand{\IHh}{\Il{\ell}}
\newcommand{\IhH}{\Il{\ell}'}

\newcommand{\IU}[1]{I^{U}_{#1}}
\newcommand{\IR}[1]{I^{R}_{#1}}
\newcommand{\IUh}{\IU{\ell}}
\newcommand{\IRh}{\IR{\ell}}

\newcommand{\Pav}[1]{\Pi^{\mathrm{av}}_{#1}}
\newcommand{\Pavh}{\Pav{\ell}}
\newcommand{\Sl}[1]{R_{#1}}
\newcommand{\Sh}{\Sl{\ell}}
\newcommand{\Kl}[1]{K_{#1}}
\newcommand{\Kh}{\Kl{\ell}}
\newcommand{\Dl}[1]{D_{#1}}
\newcommand{\Dh}{\Dl{\ell}}

\newcommand{\El}[1]{\mathcal{E}_{#1}}
\newcommand{\Eh}{\El{\ell}}
\newcommand{\EH}{\El{\ell-1}}

\newcommand{\Ph}{P_{\ell-1}}

\newcommand{\Atag}[1]{\textup{A\ref{#1}}}

\newcommand{\eqi}[2]{\eqref{#1}\ensuremath{_{#2}}}

\newcommand{\just}[2]{\overset{\makebox[1.8em]{\ensuremath{\scriptscriptstyle #2}}}{#1}}

\newcommand{\ulisch}{\ul{\mathcal{U}}_\ell}
\newcommand{\ulisc}[1]{\ul{\mathcal{U}}_{#1}}
\newcommand{\Gh}{\mathcal{G}_\ell}
\newcommand{\GH}{\mathcal{G}_{\ell-1}}
\newcommand{\ulGh}{\ul{\mathcal{G}}_\ell}

\newcommand{\Aref}[1]{Assumption~\ref{#1}}
\newcommand{\NO}{N_{\mathrm{c}}}

\newcommand{\assnum}{%
  \renewcommand{\theequation}{A.\arabic{assumption}}%
  \renewcommand{\theHequation}{A.\arabic{assumption}}%
  \addtocounter{equation}{-1}}
\newcommand{\asssubnum}{%
  \renewcommand{\theparentequation}{A.\arabic{assumption}}%
  \renewcommand{\theHparentequation}{A.\arabic{assumption}}}

\noeqref{eq:ass-bounded,eq:ass-smoother}

\usepackage{acronym}
\usepackage{enumitem}
\acrodef{pde}[PDE]{partial differential equation}
\acrodef{fe}[FE]{finite element}
\acrodef{fem}[FEM]{finite element method}
\acrodefplural{fem}[FEMs]{finite element methods}
\acrodef{dof}[DOF]{degree of freedom}
\acrodefplural{dof}[DOFs]{degrees of freedom}
\acrodef{lhs}[LHS]{left hand side}
\acrodef{rhs}[RHS]{right hand side}

\acrodef{hho}[HHO]{Hybrid High-Order}
\acrodef{hdg}[HDG]{Hybridizable Discontinuous Galerkin}
\acrodef{dg}[DG]{Discontinuous Galerkin}
\acrodef{vem}[VEM]{Virtual Element Method}
\acrodef{wg}[WG]{Weak Galerkin}
\acrodef{dpg}[DPG]{Discontinuous Petrov-Galerkin}

\acrodef{gmg}[GMG]{geometric multigrid}

\graphicspath{{./}}

\title[]{
Geometric Multigrid solvers for hybrid high-order methods on polytopal meshes
}
\date{\today}

\address{$^\dagger$School of Mathematics\\Monash University\\Clayton\\Victoria 3800\\Australia}
\author[S. Badia]{Santiago Badia$^{\dagger}$}
\email{santiago.badia@monash.edu}
\author[J. Manyer]{Jordi Manyer$^{\dagger}$}
\email{jordi.manyer@monash.edu}

\begin{document}

\begin{abstract}
We propose the first optimal geometric multigrid solver for hybrid high-order discretisations that can handle arbitrary polytopal agglomeration hierarchies in both two and three dimensions. The key ingredient is the use of modified skeleton spaces, which naturally accommodate non-planar interfaces arising during coarsening while reducing the number of degrees of freedom. We prove robust convergence with respect to the mesh size and the number of levels, and we validate our results numerically on a range of agglomeration-based mesh hierarchies. The approach extends naturally to other hybrid discretisations such as hybridizable discontinuous Galerkin and Weak Galerkin methods.
\end{abstract}

\maketitle

\section{Introduction}
\label{sec:introduction}

\ac{hho} methods \cite{DiPietro-Droniou-2020-HHO-Book,Cicuttin-Ern-Pignet-2021-HHO-Book,DiPietro-Ern-Lemaire-2014-HHOstablisation} are a class of numerical methods for the approximation of partial differential equations (PDEs) on general polytopal meshes. They belong to the broader family of non-conforming hybrid methods, characterized by the use of both cell-based and face-based unknowns, which also include \ac{hdg} methods \cite{Cockburn_2009} and Weak Galerkin methods \cite{Wang2014}. Compared to \ac{dg} methods \cite{diPietro2012,Cangiani2017hpVersion}, hybrid methods offer reduced global problem sizes through the elimination of cell unknowns via static condensation, while maintaining high-order accuracy and flexibility in mesh design.

Multigrid methods for nonconforming hybrid discretisations have been studied in the literature, mainly for \ac{hdg} \cite{Cockburn2013,SCHUTZ2017500,MURALIKRISHNAN2020109240,Lu2021,Lu2022,wildey2018} and \ac{hho} \cite{DiPietro2021,DiPietro2023,DiPietro2024,DiPietro2021b}, but also for \ac{wg} \cite{CHEN2015330}, \ac{dpg} \cite{ROBERTS20172018,PETRIDES202112}, or hybridised Raviart--Thomas / Douglas--Marini methods \cite{Gopalakrishnan2009}. While earlier approaches \cite{Cockburn2013,Gopalakrishnan2009,Kronbichler2018} recast skeletal functions as bulk functions to enable the use of traditional \ac{fe} multigrid solvers, more recent works have focused on the design of \ac{gmg} solvers that operate directly on skeletal functions. However, the latter are usually restricted to nested multigrid mesh hierarchies with planar faces at all levels, drastically reducing their applicability \cite{Lu2021,Lu2022,wildey2018,DiPietro2021,DiPietro2023,DiPietro2024}. Furthermore, the analysis of these methods is often restricted to conforming, nested, simplicial mesh hierarchies and relies on conforming piecewise linear spaces \cite{DiPietro2024,Lu2021,Lu2022}.

The work in \cite{DiPietro2021,DiPietro2021b} tries to circumvent this issue with the use of non-nested hierarchies, where the coarser levels are defined on geometrically coarsened meshes with planar faces. However, these algorithms rely on geometrical rediscretisations of the domain to generate a hierarchy of meshes with flat interfaces that readily permit coarsening of face unknowns. Due to the complexity of the geometrical algorithms required to generate such hierarchies, they have only been developed in 2D and never extended to 3D.

Polytopal methods are ideally suited for multilevel solvers based on element agglomeration, as they naturally avoid the complexities of non-conforming meshes and hanging nodes that arise with traditional \ac{fe} methods \cite{plewa2005adaptive}. The power of polytopal methods in enabling flexible agglomeration strategies has recently been demonstrated in the context of \ac{dg} \cite{PAN2022110775,Antonietti2020} methods and the \ac{vem} \cite{Antonietti2023,antonietti2026reducedbasismultigridscheme}, but the design of multilevel solvers that exploit agglomerated mesh hierarchies is still open. This is the objective of this work.

We design and analyse \ac{gmg} solvers for \ac{hho} methods on arbitrary polytopal meshes and on hierarchies built by agglomeration, in both 2D and 3D. To our knowledge, this is the first \ac{gmg} method for hybrid discretisations that handles such hierarchies. Agglomeration keeps every vertex produced by the coarsening, so the faces of a coarse cell are curved in the sense that they are unions of fine faces. Interfaces are therefore never rediscretised and the hierarchy is built by a purely combinatorial algorithm. Face unknowns of two consecutive levels are unrelated, so the coarse bilinear form is not the fine one restricted through the prolongation.

Two ingredients make the analysis possible. The first is the choice of interface space, enriched as in \cite{Yemm2022ANA} with the normal traces of the reconstruction space, which is what a curved face requires and which at the same time reduces the number of \acp{dof}. A similar construction is used in \ac{vem} for pixel-based agglomerations \cite{Bertoluzza2024}. The second is the inner product in which residuals are measured, which pairs the condensed cell components and adds the face residual of the condensed pair, weighted by the cell size. Its two terms are controlled by different mechanisms, the cell one by duality and the face one by the stabilisation. Since the coarse form is not inherited, we work in the framework of \cite{Duan2007AGB}, which extends \cite{Bramble1991} to that situation.

The main result is Theorem~\ref{th:convergence}: the V-cycle contracts uniformly and the variable V-cycle is a uniform preconditioner, with bounds independent of the mesh size and of the number of levels. It covers the two prolongation operators of \sect{sec:prolongation} and the patch smoothers of \sect{sec:smoother-def}. As a by-product of the analysis we establish in Theorem~\ref{th:error} an $L^2$ error estimate of order $\h^2$ for the agglomerated \ac{hho} discretisation itself, which appears to be the first such result for enriched interface spaces on agglomerated polytopal meshes. While we focus on \ac{hho} methods, the techniques developed here extend to other hybrid methods such as \ac{hdg} and Weak Galerkin methods, due to their established equivalence \cite{Cockburn.Di-Pietro}.

The paper is organised as follows. \sect{sec:setting} fixes the model problem, the mesh hierarchy, the interface spaces and the \ac{hho} discretisation on each level, and introduces the skeletal inner product. \sect{sec:hho-analysis} collects everything that concerns a single level, ending with the error estimates. \sect{sec:method} defines the V-cycle, the prolongations and the smoothers, and \sect{sec:framework} states the three hypotheses of \cite{Duan2007AGB} and reduces them to properties of these two ingredients. Those properties are proved in \sect{sec:intergrid} for the prolongations and in \sect{sec:smoother} for the smoothers. \sect{sec:main-result} assembles them into Theorem~\ref{th:convergence}, and \sect{sec:results} reports numerical experiments on agglomerated hierarchies in 2D and 3D.

\section{Setting}\label{sec:setting}

We write $a \lesssim b$ when $a \leq C b$ with $C$ never depending on the level, the mesh size or the data, and $a \simeq b$ when $a \lesssim b \lesssim a$; the mesh-regularity parameter and the polynomial degree fixed below are the only quantities $C$ may depend on. 

\subsection{Model problem}\label{sec:model}

Let $\dom \subset \mathbb{R}^d$, $d \in \{2,3\}$, be a bounded polytopal domain. Given $f \in L^2(\dom)$, the model problem is
\begin{equation}
  \label{eq:poisson}
  -\Delta u = f \quad \text{in } \dom, \qquad u = 0 \quad \text{on } \partial \dom,
\end{equation}
with weak form: find $u \in H^1_0(\dom)$ such that
\begin{equation}
  \label{eq:weak-poisson}
  (\nabla u, \nabla v)_{\dom} = (f, v)_{\dom} \quad \forall v \in H^1_0(\dom).
\end{equation}
Non-homogeneous Dirichlet conditions require only notational changes.

\begin{assumption}[Elliptic regularity]\label{ass:regularity}
  For every $f \in L^2(\dom)$ the solution $u \in H^1_0(\dom)$ of \eqref{eq:weak-poisson} with datum $f$ belongs to $H^2(\dom)$ and satisfies $\norm{H^2(\dom)}{u} \lesssim \norm{L^2(\dom)}{f}$, with a constant independent of $f$.
\end{assumption}

\noindent Assumption~\ref{ass:regularity} holds for convex $\dom$. It is used only through duality arguments.

\subsection{Mesh hierarchy}\label{sec:meshes}

Given a polytopal mesh $\Tl{L}$ of $\dom$, we build a hierarchy $\Tl{0}, \ldots, \Tl{L}$ by agglomeration: each cell $T \in \TH$ is the union of cells of $\Th$, so that $\ol{T} = \bigcup_{t \in \Th(T)} \ol{t}$ with $\Th(T) \doteq \{ t \in \Th : t \subset T \}$. Cells are open and disjoint, and $h_t \doteq \operatorname{diam}(t)$.

Agglomeration produces interfaces that are piecewise linear for $d = 2$ and piecewise planar for $d = 3$; we call them faces. Interior and boundary faces are
\begin{equation} \label{eq:faces}
  \Fhi \doteq \{ \ol{t} \cap \ol{t'} : t, t' \in \Th, \ t \neq t' \},
  \qquad
  \Fhb \doteq \{ \ol{t} \cap \partial\dom : t \in \Th \},
\end{equation}
retaining only those of positive $(d-1)$-dimensional measure, and $\Fh \doteq \Fhi \cup \Fhb$, so that $\partial t$ is covered up to a null set by the faces $\Fh(t) \doteq \{ F \in \Fh : F \subset \partial t \}$ of a cell. Applying the same construction to $\Fh$ gives the edges of the skeleton, and to these its vertices.

Every face is oriented. For $F = \ol{t} \cap \ol{t'} \in \Fhi$ we fix the unit normal $\bm{n}_F$ pointing from $t$ to $t'$, either labelling of the pair serving, and for $F \in \Fhb$ we take $\bm{n}_F$ outward of $\dom$. The jump of a function $v$, cell-wise smooth, is
\begin{equation} \label{eq:jump}
  \jump{v}|_F \doteq v|_t - v|_{t'} \quad F \in \Fhi, \qquad
  \jump{v}|_F \doteq v|_t \quad F \in \Fhb,
\end{equation}
so that $\jump{z} = 0$ on every face for $z \in H^1_0(\dom)$. Given two consecutive levels, $\Fh \cap \FH$ denotes the fine faces lying on a coarse face and $\Fh \setminus \FH$ those interior to a coarse cell.

\begin{assumption}[Mesh hierarchy]\label{ass:geometry}
  There is $\varrho > 0$, independent of $\ell$, such that for every level and every $t \in \Th$:
  \begin{subequations} \asssubnum \label{eq:geometry}
  \begin{align}
    &t \text{ is star-shaped with respect to a ball of radius } \varrho h_t, \label{eq:geometry-a} \\
    &\#\Fh(t) \lesssim 1, \qquad h_t \simeq \h \doteq \max_{t' \in \Th} h_{t'}, \qquad \H \simeq \h . \label{eq:geometry-b}
  \end{align}
  \end{subequations}
\end{assumption}

\eqi{eq:geometry-b}{1} bounds the number of faces of a cell without limiting the number of planar pieces composing a single face. The analysis rests on four inequalities on a single cell, all consequences of star-shapedness, which is what agglomeration preserves.\footnote{The corresponding results of \cite[Chapter 1]{DiPietro-Droniou-2020-HHO-Book} assume an auxiliary simplicial refinement of the mesh whose simplices are shape regular and as large as the cell containing them. Agglomerated cells admit none uniformly in $\ell$, the simplices meeting an interface being confined to its planar pieces, which are arbitrarily small relative to the cell.}

\begin{lemma}[Polytopal inequalities]\label{lem:polytopal}
  Let Assumption~\ref{ass:geometry} hold. For all $t \in \Th$, $m \geq 0$, $v \in \Pk{m}(t)$ and $w \in H^1(t)$,
  \begin{equation} \label{eq:polytopal-inequalities}
  \begin{gathered}
    \norm{\partial t}{v} \lesssim h_t^{-1/2} \norm{t}{v},
    \qquad
    \norm{t}{w} \lesssim h_t \seminorm{1,t}{w} + h_t^{1/2} \norm{\partial t}{w}, \\
    \norm{\partial t}{w}^2 \lesssim h_t^{-1} \norm{t}{w}^2 + h_t \seminorm{1,t}{w}^2,
    \qquad
    \norm{t}{\nabla v} \lesssim h_t^{-1} \norm{t}{v},
  \end{gathered}
  \end{equation}
  the first and the last with a constant depending also on $m$.
\end{lemma}
\begin{proof}
  Let $\bm{x}_0$ be the centre of the ball of \eqref{eq:geometry-a}, so that $B(\bm{x}_0, \varrho h_t) \subset t \subset B(\bm{x}_0, h_t)$ and, $t$ being star-shaped, $(\bm{y} - \bm{x}_0) \cdot \bm{n} \geq \varrho h_t$ for a.e.\ $\bm{y} \in \partial t$. The divergence theorem applied to $w^2 (\bm{x} - \bm{x}_0)$ gives \eqi{eq:polytopal-inequalities}{3}. The equivalence of norms on $\Pk{m}(B(\bm{0},1))$, scaled to $B(\bm{x}_0, h_t)$ and combined with $\norm{B(\bm{x}_0,h_t)}{v} \lesssim \norm{B(\bm{x}_0,\varrho h_t)}{v} \leq \norm{t}{v}$, gives \eqi{eq:polytopal-inequalities}{4}, and the two together give \eqi{eq:polytopal-inequalities}{1}. Finally, \eqi{eq:polytopal-inequalities}{2} is the generalised Poincar\'e--Friedrichs inequality with scaling.
\end{proof}

\subsection{Polynomial and face spaces}\label{sec:spaces}

For a geometric entity $X$, $\Pk{k}(X)$ is the space of polynomials of degree at most $k$ on $X$, $\lproj{X}{k}$ the $L^2$-orthogonal projector onto it and $\lproj{X}{1,k}$ the elliptic projector \cite[Definition 1.39]{DiPietro-Droniou-2020-HHO-Book}. For a collection $\mathcal{X}_\ell$ we set $\Pk{k}(\mathcal{X}_\ell) \doteq \bigtimes_{X \in \mathcal{X}_\ell} \Pk{k}(X)$ and let the projectors act element-wise. We write $(\cdot,\cdot)_X$ and $\norm{X}{\cdot}$ for the $L^2(X)$ inner product and norm, extended element-wise to collections, and $\seminorm{1,X}{\cdot}$, $\seminorm{1,\mathcal{X}_\ell}{\cdot}$ for the $H^1$ and broken $H^1$ seminorms.

Face unknowns live in spaces $\IPk{\k}(F)$ of polynomial traces on the faces $F \in \Fh$. The analysis uses one property of that choice.

\begin{assumption}[Face spaces]\label{ass:face-spaces}
  For every $t \in \Th$ and every $F \in \Fh(t)$,
  {\assnum
  \begin{equation} \label{eq:interface-space-design}
    \Pk{0}(t)|_F \subset \IPk{\k}(F),
    \qquad
    \nabla \Pk{\k+1}(t) \cdot \bm{n}_F|_F \subset \IPk{\k}(F) .
  \end{equation}}
\end{assumption}

The standard choice $\IPk{\k}(F) = \Pk{\k}(F)$ satisfies Assumption~\ref{ass:face-spaces} but is ill-suited to agglomeration. The number of planar pieces of a face grows under coarsening, so the number of face unknowns is reduced suboptimally, and the extra unknowns carry high-frequency modes that face-block smoothers fail to damp \cite{DiPietro2021,DiPietro2023,DiPietro2024}. Coarsening the face unknowns under the assumption that coarse faces are planar \cite{DiPietro2021,DiPietro2023,DiPietro2024} is inpractical because restricts the admissible hierarchies, and the non-nested alternatives of \cite{DiPietro2021,DiPietro2021b} use prolongations that are expensive, require geometric rediscretisations and are two-dimensional.

We take instead the spaces of \cite{Yemm2022ANA},
\begin{equation}
  \label{eq:interface-space}
  \IPk{\k}(F) \doteq \Pk{0}(\dom)|_F + \nabla \Pk{\k+1}(\dom)|_F \cdot \bm{n}_F .
\end{equation}
Their dimension does not depend on the number of planar pieces of $F$, so agglomeration reduces the number of face unknowns, and for planar $F$ they reduce to $\Pk{\k}(F)$. A basis is obtained by local orthogonalisation on each face \cite{Yemm2022ANA}, which can be carried out hierarchically, the level-$\ell$ basis feeding the construction of the level-$(\ell-1)$ one.

\begin{lemma}[The spaces \eqref{eq:interface-space} are admissible]\label{lem:design-property}
  The spaces \eqref{eq:interface-space} satisfy Assumption~\ref{ass:face-spaces}.
\end{lemma}
\begin{proof}
  The first inclusion is immediate from $\Pk{0}(t)|_F = \Pk{0}(\dom)|_F$. For the second, let $w \in \Pk{\k+1}(t)$ and let $\tilde{w} \in \Pk{\k+1}(\dom)$ be a polynomial with $\nabla w = \nabla \tilde{w}$ on $\ol{t}$. Since $\bm{n}_F$ is a single vector field on $F$, $\nabla w \cdot \bm{n}_F|_F = \nabla \tilde{w} \cdot \bm{n}_F|_F$ belongs to \eqref{eq:interface-space}.
\end{proof}

\subsection{The HHO discretisation on each level}\label{sec:hho}

We use a mixed-order formulation, the cell unknowns carrying one degree more than the face unknowns \cite{Cicuttin-Ern-Pignet-2021-HHO-Book}. Writing $\IPk{\k}(\Fh) \doteq \bigtimes_{F \in \Fh} \IPk{\k}(F)$ for the face space and
\begin{equation} \label{eq:HHO-space}
  \Uhb \doteq \{ \mu \in \IPk{\k}(\Fh) : \mu|_F = 0 \ \ \forall F \in \Fhb \}
\end{equation}
for its subspace with vanishing boundary trace, the hybrid space is $\Uh \doteq \Uhi \times \Uhb$ with $\Uhi \doteq \Pk{\k+1}(\Th)$. Interpolation of functions that do not vanish on $\partial\dom$ takes values in $\Pk{\k+1}(\Th) \times \IPk{\k}(\Fh)$, and we use $\Uhb$ only where the boundary condition is imposed. We write $\lproj{F}{\k}$ for the $L^2$-projector onto $\IPk{\k}(F)$ and $\lproj{\partial t}{\k}$, $\lproj{\Fh}{\k}$ for its face-wise counterparts on $\partial t$ and on the skeleton. The polynomial degree is taken independent of $\ell$; since the constants are not tracked in $\k$, the analysis applies unchanged to level-dependent degrees.

For $t \in \Th$ we write $\ul{u}_t = (u_t, u_{\partial t})$ for the restriction of $\uh \in \Uh$, and define the \emph{face residual}
\begin{equation} \label{def:face-residual}
  r_t(\ul{u}_t) \doteq u_{\partial t} - \lproj{\partial t}{\k} u_t \in \IPk{\k}(\Fh(t)),
\end{equation}
the mismatch between the two components on $\partial t$. The \ac{hho} norm on $\Uh$ is
\begin{equation} \label{def:hho-norm}
  \norm{\ul{1,\ell}}{\uh}^2 \doteq \sum_{t \in \Th} \norm{\ul{1,t}}{\ul{u}_t}^2, \qquad
  \norm{\ul{1,t}}{\ul{u}_t}^2 \doteq \seminorm{1,t}{u_t}^2 + h_t^{-1} \norm{\partial t}{u_{\partial t} - u_t}^2 .
\end{equation}
The local reconstruction $\rec{\ell,t}{\k+1} \ul{u}_t \in \Pk{\k+1}(t)$ is defined by
\begin{subequations} \label{eq:reconstruction}
\begin{align}
  (\nabla \rec{\ell,t}{\k+1} \ul{u}_t, \nabla w)_t &= (\nabla u_t, \nabla w)_t + (u_{\partial t} - u_t, \nabla w \cdot \bm{n})_{\partial t} \quad \forall w \in \Pk{\k+1}(t), \label{eq:reconstruction-a} \\
  (\rec{\ell,t}{\k+1} \ul{u}_t, 1)_t &= (u_t, 1)_t , \label{eq:reconstruction-b}
\end{align}
\end{subequations}
and the local bilinear form by
\begin{equation} \label{eq:local-bilinear-form}
  \ual{\ell,t}(\ul{u}_t, \ul{v}_t) \doteq (\nabla \rec{\ell,t}{\k+1} \ul{u}_t, \nabla \rec{\ell,t}{\k+1} \ul{v}_t)_t + h_t^{-1} \big( r_t(\ul{u}_t), r_t(\ul{v}_t) \big)_{\partial t},
\end{equation}
the second term being the usual stabilisation, written through \eqref{def:face-residual}. Summing over cells gives $\uah$, and the discrete problem reads: find $\uh \in \Uh$ with
\begin{equation}
  \label{eq:discrete-poisson}
  \uah(\uh, \vh) = (f, v_{\Th})_{\Th} \quad \forall \vh \in \Uh .
\end{equation}

Cell unknowns are eliminated by static condensation. For $t \in \Th$, $\lambda \in \Uhb$ and $g \in L^2(\dom)$, the operators $\isc{\ell,t}$ and $\dualisc{\ell,t}$ into $\Pk{\k+1}(t)$ are defined by
\begin{subequations} \label{eq:inverse-static-condensation}
\begin{align}
  \ual{\ell,t}((\isc{\ell,t} \lambda, 0), (w, 0)) &= - \ual{\ell,t}((0, \lambda), (w, 0)) \quad \forall w \in \Pk{\k+1}(t), \label{eq:inverse-static-condensation-a} \\
  \ual{\ell,t}((\dualisc{\ell,t} g, 0), (w, 0)) &= (g,w)_t \quad \forall w \in \Pk{\k+1}(t), \label{eq:inverse-static-condensation-b}
\end{align}
\end{subequations}
with global counterparts $\isch$, $\dualisch$ obtained cell-wise; both problems are well posed by Lemma \ref{lem:local-stability}, whose proof uses only the present subsection. We write $\ulisch \lambda \doteq (\isch \lambda, \lambda) \in \Uh$ for the hybrid extension of a face function and set
\begin{equation} \label{eq:condensed-form}
  \ah(\lambda, \mu) \doteq \uah(\ulisch \lambda, \ulisch \mu),
\end{equation}
so that the condensed problem reads: find $\lambda \in \Uhb$ with
\begin{equation} \label{eq:discrete-poisson-skeleton}
  \ah(\lambda, \mu) = (f, \isch \mu)_{\Th} \quad \forall \mu \in \Uhb .
\end{equation}
We denote by $\norm{\uah}{\cdot}$ and $\norm{\ah}{\cdot}$ the energy norms of $\uah$ and $\ah$. For these and for \eqref{def:hho-norm}, a subset of the mesh appended to the subscript denotes the restriction to it.

Since the analysis compares discrete solutions with different data and on different levels, we denote by $\Gh : L^2(\dom) \to \Uhb$ the operator mapping $g$ to the solution of \eqref{eq:discrete-poisson-skeleton} with datum $g$, and write
\begin{equation} \label{eq:solution-operator}
  \ulGh g \doteq (\isch \Gh g + \dualisch g, \Gh g) \in \Uh
\end{equation}
for the corresponding element of $\Uh$. It solves \eqref{eq:discrete-poisson} with datum $g$: every $\vh \in \Uh$ splits as a bubble $(v_{\Th} - \isch v_{\Fh}, 0)$ plus $\ulisch v_{\Fh}$, and testing with the two types gives \eqref{eq:inverse-static-condensation-b} and \eqref{eq:discrete-poisson-skeleton} respectively.

The \ac{hho} interpolator and the consistency error are
\begin{equation} \label{eq:interpolator-consistency}
  \ul{\mathcal{J}}_\ell z \doteq (\lproj{\Th}{\k+1} z, \lproj{\Fh}{\k} z), \qquad
  \Eh(z; \vh) \doteq (-\Delta z, v_{\Th})_{\Th} - \uah(\ul{\mathcal{J}}_\ell z, \vh),
\end{equation}
for $z \in H^1_0(\dom) \cap H^2(\dom)$ and $\vh = (v_{\Th}, v_{\Fh}) \in \Uh$; note $\ul{\mathcal{J}}_\ell z \in \Uh$ for such $z$, since $\lproj{F}{\k} z = 0$ on boundary faces.

\subsection{The skeletal inner product}\label{sec:inner-product}

The multigrid analysis needs an inner product on $\Uhb$. Pairing only the condensed cell components, $(\isch \lambda, \isch \mu)_{\Th}$, is not definite, since $\isch$ has a nontrivial kernel on agglomerated levels. We add the face residual \eqref{def:face-residual} of the condensed pair, which is what the cell component misses, and set
\begin{equation} \label{def:skeleton-inner-product}
  \braket{\lambda,\mu}_{\ell} \doteq (\isch \lambda, \isch \mu)_{\Th} + \sum_{t \in \Th} h_t \big( r_t(\ulisch \lambda), r_t(\ulisch \mu) \big)_{\partial t},
  \qquad
  \norm{\ell}{\lambda}^2 \doteq \braket{\lambda,\lambda}_{\ell} .
\end{equation}
This is the $L^2$ product of the hybrid pairs $\ulisch \lambda$ and $\ulisch \mu$, the weight $h_t$ giving the face term the scaling of a bulk integral. Alternative choices for this inner product and their effect on the analysis will be discussed in Remark~\ref{rem:face-based}. The level operator $\Ah : \Uhb \to \Uhb$ is defined by
\begin{equation} \label{eq:level-operator}
  \ah(\lambda,\mu) = \braket{\Ah \lambda, \mu}_\ell \quad \forall \lambda, \mu \in \Uhb ,
\end{equation}
which determines $\Ah$ once \eqref{def:skeleton-inner-product} is known to be definite, as Lemma~\ref{lem:inner-product} shows.

\subsection{Notation}\label{sec:notation}

Table~\ref{tab:notation} collects the symbols used throughout, together with the equation that defines each.

\begin{table}[H]
  \centering\small
  \begin{tabular}{@{}llll@{}}
    \toprule
    Symbol & Meaning & Symbol & Meaning \\
    \midrule
    $\dom$, $d$ & domain, dimension $2,3$ & $\braket{\cdot,\cdot}_\ell$, $\norm{\ell}{\cdot}$ & skeletal inner product \eqref{def:skeleton-inner-product} \\
    $\Th$, $\ell$, $L$ & mesh on level $\ell$, levels & $\Ah$, $\eigenmax$ & level operator \eqref{eq:level-operator}, eigenvalue \eqref{eq:max-eigenvalue} \\
    $t$, $T$ & cell of $\Th$, of $\TH$ & $\Bh$, $m_\ell$ & V-cycle \eqref{eq:vcycle}, smoothing steps \\
    $\Fh$, $\Fhi$, $\Fhb$ & faces, interior, boundary \eqref{eq:faces} & $\norm{\ul{1,\ell}}{\cdot}$, $\norm{\ah}{\cdot}$ & HHO \eqref{def:hho-norm} and energy norms \\
    $\bm{n}_F$, $\jump{\cdot}$ & face normal, jump \eqref{eq:jump} & $\IHh$, $\IUh$, $\IRh$ & prolongations \eqref{eq:prolongations} \\
    $h_t$, $\h$, $\H$ & cell, level, coarse level size & $\IhH$, $\ul{I}_\ell$ & restriction \eqref{eq:restriction}, extension \eqref{eq:prolongations} \\
    $\Pk{k}(X)$, $\lproj{X}{k}$ & polynomials, $L^2$-projector & $\Pavh$, $\ul{\mathcal{W}}_\ell$ & averaging \eqref{def:Pi-av}, cell extension \eqref{def:Wh} \\
    $\lproj{X}{1,k}$ & elliptic projector & $\Ph$ & coarse projection \eqref{eq:coarse-projection} \\
    $\IPk{\k}(F)$ & face space \eqref{eq:interface-space} & $\Sh$, $\Kh$ & smoother \eqref{def:patch-smoother}, symmetrised \eqref{eq:symmetrised-smoother} \\
    $\Uh$, $\Uhb$ & hybrid, face spaces \eqref{eq:HHO-space} & $\mathcal{D}$, $N_r$, $\NO$ & patch entities \eqref{def:patch-space}, colours \\
    $r_t$ & face residual \eqref{def:face-residual} & $\omega$ & damping factor \\
    $\rec{\ell}{\k+1}$ & reconstruction \eqref{eq:reconstruction} & $f$, $u$ & datum, solution \eqref{eq:weak-poisson} \\
    $\uah$, $\ah$ & hybrid \eqref{eq:local-bilinear-form}, condensed \eqref{eq:condensed-form} forms & $\Gh$, $\ulGh$ & discrete solution operator \eqref{eq:solution-operator} \\
    $\isch$, $\dualisch$, $\ulisch$ & static condensation \eqref{eq:inverse-static-condensation} & $\ul{\mathcal{J}}_\ell$, $\Eh$ & interpolator, consistency error \eqref{eq:interpolator-consistency} \\
    \bottomrule
  \end{tabular}
  \caption{Summary of notation.}
  \label{tab:notation}
\end{table}

\section{Analysis of the HHO discretisation}\label{sec:hho-analysis}

Every result the multigrid analysis needs on a single level is collected here. Assumptions~\ref{ass:geometry} and \ref{ass:face-spaces} are assumed in the numerical analysis, and are not repeated in the statements below; the hypotheses listed there are those beyond them.

\subsection{Local stability and the skeletal norm}\label{sec:local-stability}

\begin{lemma}[Local solver stability]\label{lem:local-stability}
  For all $t \in \Th$, $v \in \Pk{\k+1}(t)$, $\lambda \in \IPk{\k}(\Fh(t))$ and $\ul{u}_t \in \Pk{\k+1}(t) \times \IPk{\k}(\Fh(t))$,
  \begin{subequations}\label{eq:local-stability}
  \begin{align}
    \seminorm{1,t}{v}^2 + h_t^{-1} \norm{\partial t}{\lambda - \lproj{\partial t}{\k} v}^2
      &\lesssim \ual{\ell,t}((v,\lambda),(v,\lambda)), \label{eq:local-coercivity} \\
    \ual{\ell,t}(\ulisc{\ell,t} \lambda, \ulisc{\ell,t} \lambda) + h_t^{-2} \norm{t}{\isc{\ell,t}\lambda}^2
      &\lesssim h_t^{-1} \norm{\partial t}{\lambda}^2, \label{eq:local-energy-bound} \\
    \ual{\ell,t}(\ul{u}_t, \ul{u}_t) &\simeq \norm{\ul{1,t}}{\ul{u}_t}^2 . \label{eq:local-continuity}
  \end{align}
  \end{subequations}
\end{lemma}
\begin{proof}
  By \eqref{eq:reconstruction-a}, $(\nabla v, \nabla w)_t = (\nabla \rec{\ell,t}{\k+1}(v,\lambda), \nabla w)_t - (\lambda - v, \nabla w \cdot \bm{n})_{\partial t}$ for all $w \in \Pk{\k+1}(t)$. Since $\nabla w \cdot \bm{n}|_F \in \IPk{\k}(F)$ by \eqref{eq:interface-space-design} and $\lambda|_F \in \IPk{\k}(F)$, the boundary term equals $(\lambda - \lproj{\partial t}{\k} v, \nabla w \cdot \bm{n})_{\partial t}$. Taking $w = v$ and a Cauchy--Schwarz inequality,
  \[
    \seminorm{1,t}{v}^2 \just{\lesssim}{\eqi{eq:polytopal-inequalities}{1}} \big( \norm{t}{\nabla \rec{\ell,t}{\k+1}(v,\lambda)} + h_t^{-1/2} \norm{\partial t}{\lambda - \lproj{\partial t}{\k} v} \big) \seminorm{1,t}{v},
  \]
  and \eqref{eq:local-coercivity} follows from the definition of $\ual{\ell,t}$ in \eqref{eq:local-bilinear-form}.

  For \eqref{eq:local-energy-bound}, \eqref{eq:inverse-static-condensation-a} states that $\ulisc{\ell,t}\lambda$ minimises $v \mapsto \ual{\ell,t}((v,\lambda),(v,\lambda))$, so testing with $v = 0$,
  \[
    \ual{\ell,t}(\ulisc{\ell,t}\lambda, \ulisc{\ell,t}\lambda) \leq \norm{t}{\nabla \rec{\ell,t}{\k+1}(0,\lambda)}^2 + h_t^{-1} \norm{\partial t}{\lambda}^2
    \lesssim h_t^{-1} \norm{\partial t}{\lambda}^2,
  \]
  since $(\nabla \rec{\ell,t}{\k+1}(0,\lambda), \nabla w)_t = (\lambda, \nabla w \cdot \bm{n})_{\partial t} \lesssim h_t^{-1/2} \norm{\partial t}{\lambda} \norm{t}{\nabla w}$ by \eqref{eq:reconstruction-a} and \eqi{eq:polytopal-inequalities}{1}. For the $L^2$ term set $v \doteq \isc{\ell,t}\lambda$. Using the two previous bounds, we get $\seminorm{1,t}{v} + h_t^{-1/2} \norm{\partial t}{\lambda - \lproj{\partial t}{\k} v} \lesssim \norm{\uah,t}{\ulisc{\ell,t}\lambda} \lesssim h_t^{-1/2} \norm{\partial t}{\lambda}$. From \eqref{eq:interface-space-design} we have that constants belong to $\IPk{\k}(F)$ and thus $(1 - \lproj{\partial t}{\k}) v = (1 - \lproj{\partial t}{\k})(v - \ol{v})$ with $\ol{v}$ the mean of $v$ on $t$. Whence, 
  \begin{equation} \label{eq:mean-chain}
    \norm{\partial t}{(1 - \lproj{\partial t}{\k}) v} \leq \norm{\partial t}{v - \ol{v}}
    \just{\lesssim}{\eqi{eq:polytopal-inequalities}{1}} h_t^{-1/2} \norm{t}{v - \ol{v}} \lesssim h_t^{1/2} \seminorm{1,t}{v}
  \end{equation}
  by Poincar\'e--Wirtinger. Therefore $\norm{\partial t}{v} \lesssim \norm{\partial t}{\lambda}$, and \eqi{eq:polytopal-inequalities}{2} gives $\norm{t}{v} \lesssim h_t \seminorm{1,t}{v} + h_t^{1/2} \norm{\partial t}{v} \lesssim h_t^{1/2} \norm{\partial t}{\lambda}$.

  For \eqref{eq:local-continuity}, \eqref{eq:reconstruction-a} and \eqi{eq:polytopal-inequalities}{1} give $\norm{t}{\nabla \rec{\ell,t}{\k+1} \ul{u}_t} \lesssim \seminorm{1,t}{u_t} + h_t^{-1/2} \norm{\partial t}{u_{\partial t} - u_t}$, and the stabilisation is bounded by $h_t^{-1} \norm{\partial t}{u_{\partial t} - u_t}^2$, since $r_t(\ul{u}_t) = \lproj{\partial t}{\k}(u_{\partial t} - u_t)$ by \eqref{def:face-residual}. The reverse bound is \eqref{eq:local-coercivity} combined with \eqref{eq:mean-chain}.
\end{proof}

\begin{lemma}[Skeletal inner product]\label{lem:inner-product}
  \eqref{def:skeleton-inner-product} is an inner product on $\Uhb$, and $\Ah$ is well defined, self-adjoint and positive definite with respect to it.
\end{lemma}
\begin{proof}
  Symmetry and bilinearity are clear, and $\braket{\lambda,\lambda}_\ell \geq 0$. If $\braket{\lambda,\lambda}_\ell = 0$ then $\isc{\ell,t}\lambda = 0$ and $r_t(\ulisch \lambda) = 0$ on every cell, and the first identity turns the second into $\lambda|_{\partial t} = 0$. Every face belongs to some cell boundary, so $\lambda = 0$.

  For $\Ah$, symmetry is that of $\ah$. If $\ah(\lambda,\lambda) = 0$ then $\ual{\ell,t}(\ulisch\lambda,\ulisch\lambda) = 0$ on every cell, so \eqref{eq:local-coercivity} makes $\isc{\ell,t}\lambda$ constant on $t$ and equal to $\lambda$ on $\partial t$. Two cells sharing an interface carry the same constant, and the constant vanishes on any cell with a boundary face. Since $\dom$ is connected and the interfaces of \eqref{eq:faces} have positive measure, every cell is joined to such a cell by a chain of interfaces, so $\lambda = 0$.
\end{proof}

\begin{lemma}[Skeletal norm]\label{lem:skeletal-norm}
  For all $\lambda \in \Uhb$,
  \begin{subequations}\label{eq:skeletal-norm}
  \begin{alignat}{2}
    \norm{\Th}{\isch \lambda} &\leq \norm{\ell}{\lambda},
      &\qquad \norm{\ell}{\lambda} &\leq \norm{\Th}{\isch \lambda} + \h \norm{\ah}{\lambda}, \label{eq:l2-representative} \\
    \Big( \sum_{t \in \Th} h_t \norm{\partial t}{r_t(\ulisch \lambda)}^2 \Big)^{1/2} &\leq \h \norm{\ah}{\lambda},
      &\qquad \norm{\ell}{\lambda}^2 &\simeq \sum_{t \in \Th} h_t \norm{\partial t}{\lambda}^2 . \label{eq:face-bounds}
  \end{alignat}
  \end{subequations}
\end{lemma}
\begin{proof}
  Evaluating \eqref{eq:local-bilinear-form} at $\ulisch \lambda$, we have that
  \[
    h_t \norm{\partial t}{r_t(\ulisch \lambda)}^2 = h_t^2 \, h_t^{-1} \norm{\partial t}{r_t(\ulisch \lambda)}^2 \leq h_t^2 \, \ual{\ell,t}(\ulisch \lambda, \ulisch \lambda) = h_t^2 \norm{\ah,t}{\lambda}^2,
  \]
  and \eqi{eq:face-bounds}{1} follows on summing over $t$ and using $h_t \leq \h$. The bound in \eqi{eq:l2-representative}{1} is immediate from the definition in \eqref{def:skeleton-inner-product}. Then \eqi{eq:l2-representative}{2} follows from \eqi{eq:face-bounds}{1} applied to \eqref{def:skeleton-inner-product}.

  For \eqi{eq:face-bounds}{2}, take $v \doteq \isc{\ell,t}\lambda$ for which \eqref{eq:local-energy-bound} gives $\norm{t}{v}^2 \lesssim h_t \norm{\partial t}{\lambda}^2$. Then, using the $L^2(F)$-contractivity of $\lproj{F}{\k}$, we have
  \[
    h_t \norm{\partial t}{r_t(\ulisch \lambda)}^2 \lesssim h_t \norm{\partial t}{\lambda}^2 + h_t \norm{\partial t}{v}^2
    \just{\lesssim}{\eqi{eq:polytopal-inequalities}{1}} h_t \norm{\partial t}{\lambda}^2 + \norm{t}{v}^2
    \just{\lesssim}{\eqref{eq:local-energy-bound}} h_t \norm{\partial t}{\lambda}^2.
  \]
  The upper bound in \eqi{eq:face-bounds}{2} then follows from \eqref{def:skeleton-inner-product}. The lower bound follows from $h_t \norm{\partial t}{\lambda}^2 \lesssim h_t \norm{\partial t}{r_t(\ulisch \lambda)}^2 + \norm{t}{v}^2$, obtained in the same way, on summing over $t \in \Th$.
\end{proof}

\begin{lemma}[Spectral bound]\label{lem:spectral-bound}
  The largest eigenvalue of $\Ah$ satisfies
  \begin{equation} \label{eq:max-eigenvalue}
    \eigenmax \lesssim \h^{-2} .
  \end{equation}
\end{lemma}
\begin{proof}
  By \eqref{eq:condensed-form}, for $\lambda \in \Uhb$,
  \[
    \ah(\lambda,\lambda) = \sum_{t \in \Th} \ual{\ell,t}(\ulisc{\ell,t}\lambda, \ulisc{\ell,t}\lambda)
    \just{\lesssim}{\eqref{eq:local-energy-bound}} \sum_{t \in \Th} h_t^{-1} \norm{\partial t}{\lambda}^2
    \just{\lesssim}{\eqi{eq:geometry-b}{2}} \h^{-2} \sum_{t \in \Th} h_t \norm{\partial t}{\lambda}^2
    \just{\simeq}{\eqi{eq:face-bounds}{2}} \h^{-2} \norm{\ell}{\lambda}^2 . \qedhere
  \]
\end{proof}

\begin{remark}[Smallest eigenvalue]\label{rem:min-eigenvalue}
  A discrete Poincar\'e--Friedrichs inequality gives $\lambda_{\mathrm{min}}(\Ah) \gtrsim 1$, so that $\Ah$ has condition number $\lesssim \h^{-2}$.
\end{remark}

\subsection{Approximation, interpolation and consistency}\label{sec:approximation}

\begin{lemma}[Polynomial approximation]\label{lem:approximation}
  Let $t \in \Th$, $l \geq 0$, $1 \leq s \leq l+1$, $0 \leq m \leq s$ and $w \in H^s(t)$, and let $\pi_t$ be either $\lproj{t}{l}$ or $\lproj{t}{1,l}$. Then
  \begin{subequations}\label{eq:approximation}
  \begin{gather}
    \seminorm{H^m(t)}{w - \pi_t w} \lesssim h_t^{s-m} \seminorm{H^s(t)}{w}, \qquad
    h_t^{1/2} \norm{\partial t}{w - \pi_t w} \lesssim h_t^{s} \seminorm{H^s(t)}{w}, \label{eq:approximation-bulk} \\
    h_t^{3/2} \norm{\partial t}{\nabla (w - \pi_t w)} \lesssim h_t^{s} \seminorm{H^s(t)}{w} . \label{eq:approximation-face}
  \end{gather}
  \end{subequations}
  \eqref{eq:approximation-face} holds for $s \geq 2$.
\end{lemma}

The bounds \eqref{eq:approximation-bulk} are \cite[Theorems 1.45 and 1.48]{DiPietro-Droniou-2020-HHO-Book}, for the $L^2$-orthogonal and for the elliptic projector respectively. For \eqref{eq:approximation-face} we apply \eqi{eq:polytopal-inequalities}{3} to $\nabla(w - \pi_t w)$, which makes use of \eqi{eq:approximation-bulk}{1} with $m = 2$, whence the restriction to $s \geq 2$. They are stated in \cite{DiPietro-Droniou-2020-HHO-Book} on regular mesh sequences, but the underlying averaged Taylor construction of \cite{DupontScott1980} requires only star-shapedness with respect to a ball, so \eqref{eq:approximation} holds under \eqref{eq:geometry-a}.

\begin{lemma}[HHO interpolation]\label{lem:interpolation}
  For $T \in \TH$ and $v \in H^1(T)$, and for $z \in H^1_0(\dom) \cap H^2(\dom)$,
  \begin{subequations}\label{eq:interpolation}
  \begin{align}
    \norm{\ul{1,\ell},T}{\ul{\mathcal{J}}_\ell v} &\lesssim \seminorm{1,T}{v}, \label{eq:interpolation-bounded} \\
    \Big( \sum_{T \in \TH} \norm{\ul{1,\ell},T}{\ul{\mathcal{J}}_\ell (z - \lproj{\TH}{1,\K+1} z)}^2 \Big)^{1/2}
      + \H^{-1/2} \norm{\FH}{\jump{\lproj{\TH}{1,\K+1} z}} &\lesssim \H \seminorm{H^2(\dom)}{z} . \label{eq:interpolation-jumps}
  \end{align}
  \end{subequations}
\end{lemma}
\begin{proof}
  Fix $t \in \Th(T)$ and let $\ol{v}$ be the mean of $v$ on $t$. Since $\lproj{t}{\k+1} \ol{v} = \ol{v}$, the inverse inequality, the contractivity of $\lproj{t}{\k+1}$ and Poincar\'e--Wirtinger give $\seminorm{1,t}{\lproj{t}{\k+1} v} \lesssim \seminorm{1,t}{v}$. For the face term, $(1 - \lproj{\partial t}{\k}) v = (1 - \lproj{\partial t}{\k})(v - \ol{v})$ by \eqref{eq:interface-space-design}, so
  \[
    \norm{\partial t}{\lproj{\partial t}{\k} v - \lproj{t}{\k+1} v}
    \leq \norm{\partial t}{v - \ol{v}} + \norm{\partial t}{v - \lproj{t}{\k+1} v}
    \lesssim h_t^{1/2} \seminorm{1,t}{v},
  \]
  by \eqi{eq:polytopal-inequalities}{3} with Poincar\'e--Wirtinger for the first term, and \eqi{eq:approximation-bulk}{2} for the second. Summing over $t \in \Th(T)$ gives \eqref{eq:interpolation-bounded}.

  For \eqref{eq:interpolation-jumps} set $z_{\ell-1} \doteq \lproj{\TH}{1,\K+1} z$. The first term is bounded by \eqref{eq:interpolation-bounded} on each coarse cell followed by \eqi{eq:approximation-bulk}{1} with $s = 2$, $m = 1$. For the second, $z$ is continuous across interior faces and vanishes on $\partial\dom$, so $\jump{z_{\ell-1}} = \jump{z_{\ell-1} - z}$ on every face by \eqref{eq:jump}, and since each interior coarse face is shared by two cells,
  \[
    \norm{\FH}{\jump{z_{\ell-1}}}^2 \leq 2 \sum_{T \in \TH} \norm{\partial T}{z - z_{\ell-1}}^2
    \just{\lesssim}{\eqi{eq:approximation-bulk}{2}} \H^3 \seminorm{H^2(\dom)}{z}^2 .
  \]
\end{proof}

The reconstruction is exact on interpolants: for $t \in \Th$ and $w \in H^2(t)$,
\begin{equation} \label{eq:elliptic-identity}
  \rec{\ell,t}{\k+1} \ul{\mathcal{J}}_{\ell,t} w = \lproj{t}{1,\k+1} w .
\end{equation}
Indeed, integrating the bulk term of \eqref{eq:reconstruction-a} by parts and using $\Delta q \in \Pk{\k+1}(t)$ and $\nabla q \cdot \bm{n}|_F \in \IPk{\k}(F)$ from \eqi{eq:interface-space-design}{2} gives $(\nabla \rec{\ell,t}{\k+1} \ul{\mathcal{J}}_{\ell,t} w, \nabla q)_t = (\nabla w, \nabla q)_t$ for all $q \in \Pk{\k+1}(t)$, and \eqref{eq:reconstruction-b} matches the means. Both properties define $\lproj{t}{1,\k+1} w$. This is the first of the two places where Assumption~\ref{ass:face-spaces} is essential.

\begin{lemma}[Consistency]\label{lem:consistency}
  For all $z \in H^1_0(\dom) \cap H^2(\dom)$ and all $\vh \in \Uh$,
  \begin{equation} \label{eq:consistency-bound}
    |\Eh(z; \vh)| \lesssim \h \seminorm{H^2(\dom)}{z} \norm{\uah}{\vh} .
  \end{equation}
\end{lemma}
\begin{proof}
  Starting from \eqref{eq:interpolator-consistency}, we proceed similarly to \cite[Lemma 2.18]{DiPietro-Droniou-2020-HHO-Book}. Integrating by parts cell by cell and using \eqref{eq:elliptic-identity}, \eqref{eq:reconstruction-a} and \eqref{eq:local-bilinear-form},
  \begin{equation} \label{eq:consistency-identity}
    \Eh(z; \vh) = \sum_{t \in \Th} \big( \nabla (z - \lproj{t}{1,\k+1} z) \cdot \bm{n}, v_{\partial t} - v_t \big)_{\partial t} - \sum_{t \in \Th} h_t^{-1} \big( r_t(\ul{\mathcal{J}}_{\ell,t} z), r_t(\ul{v}_t) \big)_{\partial t} .
  \end{equation}
  By \eqref{eq:local-coercivity} and \eqref{eq:mean-chain}, $h_t^{-1} \norm{\partial t}{v_{\partial t} - v_t}^2 \lesssim \ual{\ell,t}(\ul{v}_t, \ul{v}_t)$, so Cauchy--Schwarz and \eqref{eq:approximation-face} bound the first sum by $\h \seminorm{H^2(\dom)}{z} \norm{\uah}{\vh}$. The second is bounded in the same way, the stabilisation being positive semidefinite and $h_t^{-1/2} \norm{\partial t}{r_t(\ul{\mathcal{J}}_{\ell,t} z)} \lesssim h_t \seminorm{H^2(t)}{z}$ by \eqi{eq:approximation-bulk}{2}.
\end{proof}

\begin{lemma}[Bubble bounds]\label{lem:bubbles}
  Let $\ul{\delta}_\ell = (\delta_\ell, 0) \in \Uh$ and $z \in H^1_0(\dom) \cap H^2(\dom)$. Then
  \begin{subequations}\label{eq:bubbles}
  \begin{align}
    \norm{\Th}{\delta_\ell} &\lesssim \h \norm{\uah}{\ul{\delta}_\ell}, \label{eq:bubble-l2} \\
    \norm{\uah}{\ul{\mathcal{J}}_\ell z - \ulisch \lproj{\Fh}{\k} z} &\lesssim \h \seminorm{H^2(\dom)}{z} . \label{eq:bubble-interp}
  \end{align}
  \end{subequations}
\end{lemma}
\begin{proof}
  On each cell, $\ul{\delta}_\ell$ has zero face component, so $\norm{\ul{1,t}}{\ul{\delta}_t}^2 = \seminorm{1,t}{\delta_t}^2 + h_t^{-1} \norm{\partial t}{\delta_t}^2$ and \eqi{eq:polytopal-inequalities}{2} gives $\norm{t}{\delta_t} \lesssim h_t \norm{\ul{1,t}}{\ul{\delta}_t}$. Squaring, summing over $t$ and using \eqref{eq:local-continuity} yields \eqref{eq:bubble-l2}.

  For \eqref{eq:bubble-interp}, the difference $\ul{\delta}_\ell \doteq \ul{\mathcal{J}}_\ell z - \ulisch \lproj{\Fh}{\k} z$ has zero face component, and $\ulisch \lproj{\Fh}{\k} z$ is discrete harmonic, so $\uah(\ulisch \lproj{\Fh}{\k} z, \ul{\delta}_\ell) = 0$ by \eqref{eq:inverse-static-condensation-a}. Hence, by \eqref{eq:interpolator-consistency},
  \[
    \norm{\uah}{\ul{\delta}_\ell}^2 = \uah(\ul{\mathcal{J}}_\ell z, \ul{\delta}_\ell) = (-\Delta z, \delta_\ell)_{\Th} - \Eh(z; \ul{\delta}_\ell)
    \just{\lesssim}{\eqref{eq:bubble-l2},\,\eqref{eq:consistency-bound}} \h \seminorm{H^2(\dom)}{z} \norm{\uah}{\ul{\delta}_\ell} . \qedhere
  \]
\end{proof}

\subsection{Error estimates}\label{sec:error}

\begin{theorem}[Error estimates]\label{th:error}
  Let Assumption~\ref{ass:regularity} hold. For $g \in L^2(\dom)$ let $u_g$ solve \eqref{eq:weak-poisson} with datum $g$. Then
  \begin{equation} \label{eq:error}
    \norm{\ah}{\Gh g - \lproj{\Fh}{\k} u_g} + \h^{-1} \norm{L^2(\dom)}{\isch \Gh g - u_g} \lesssim \h \norm{L^2(\dom)}{g} .
  \end{equation}
\end{theorem}
\begin{proof}
  Write $\lambda \doteq \Gh g$ and $\ulGh g = (u_{\Th}, \lambda)$ with $u_{\Th} = \isch \lambda + \dualisch g$.

  \emph{Step 1: energy error.} By \eqref{eq:solution-operator} and \eqref{eq:interpolator-consistency}, $\uah(\ulGh g - \ul{\mathcal{J}}_\ell u_g, \vh) = \Eh(u_g; \vh)$ for all $\vh \in \Uh$. Taking $\vh = \ulGh g - \ul{\mathcal{J}}_\ell u_g$,
  \begin{equation} \label{eq:energy-error}
    \norm{\uah}{\ulGh g - \ul{\mathcal{J}}_\ell u_g} \just{\lesssim}{\eqref{eq:consistency-bound}} \h \seminorm{H^2(\dom)}{u_g} \just{\lesssim}{\Atag{ass:regularity}} \h \norm{L^2(\dom)}{g},
  \end{equation}
  and the energy bound in \eqref{eq:error} follows, the skeletal component of $\ulGh g - \ul{\mathcal{J}}_\ell u_g$ being $\Gh g - \lproj{\Fh}{\k} u_g$ and static condensation being energy minimal.

  \emph{Step 2: the cell lifting is $O(\h^2)$.} Summing \eqref{eq:inverse-static-condensation-b} over the cells with $w = \dualisch g$ gives $\norm{\uah}{(\dualisch g, 0)}^2 = (g, \dualisch g)_{\Th} \leq \norm{L^2(\dom)}{g} \norm{\Th}{\dualisch g}$, and \eqref{eq:bubble-l2} bounds the last factor by $\h \norm{\uah}{(\dualisch g, 0)}$, so $\norm{\Th}{\dualisch g} \lesssim \h^2 \norm{L^2(\dom)}{g}$. By the triangle inequality it therefore suffices to prove $\norm{L^2(\dom)}{u_{\Th} - u_g} \lesssim \h^2 \norm{L^2(\dom)}{g}$.

  \emph{Step 3: superconvergence.} Let $z, w \in H^1_0(\dom) \cap H^2(\dom)$ and evaluate \eqref{eq:consistency-identity} at $\vh = \ul{\mathcal{J}}_\ell w$. Splitting the test difference on $\partial t$ as $\lproj{\partial t}{\k} w - \lproj{t}{\k+1} w = (\lproj{\partial t}{\k} w - w) + (w - \lproj{t}{\k+1} w)$, the first sum of \eqref{eq:consistency-identity} becomes $\mathfrak{T}_1 + \mathfrak{T}_2$, with
  \[
    \mathfrak{T}_i \doteq \sum_{t \in \Th} \big( \nabla (z - \lproj{t}{1,\k+1} z) \cdot \bm{n}, \, \phi_i \big)_{\partial t},
    \qquad \phi_1 \doteq \lproj{\partial t}{\k} w - w, \quad \phi_2 \doteq w - \lproj{t}{\k+1} w .
  \]
  The term $\mathfrak{T}_1$ vanishes. On a boundary face $\phi_1$ itself vanishes, $w \in H^1_0(\dom)$ giving $w|_F = \lproj{F}{\k} w = 0$. On an interior face the contribution of $\nabla z$ cancels between the two cells, $\lproj{\Fh}{\k} w - w$ being single-valued and $z \in H^2(\dom)$, and that of $\nabla \lproj{t}{1,\k+1} z$ vanishes because $\nabla \lproj{t}{1,\k+1} z \cdot \bm{n}|_F \in \IPk{\k}(F)$ by \eqi{eq:interface-space-design}{2} while $w - \lproj{F}{\k} w \perp \IPk{\k}(F)$.\footnote{This is the second place where Assumption~\ref{ass:face-spaces} is essential: for a face space violating it, $\mathfrak{T}_1$ is only $O(\h)$ and the duality gain is lost.} For $\mathfrak{T}_2$, Cauchy--Schwarz and \eqref{eq:approximation} with $s = 2$ give
  \[
    |\mathfrak{T}_2| \leq \sum_{t \in \Th} \norm{\partial t}{\nabla (z - \lproj{t}{1,\k+1} z)} \norm{\partial t}{w - \lproj{t}{\k+1} w}
    \lesssim \sum_{t \in \Th} h_t^{1/2} \seminorm{H^2(t)}{z} \cdot h_t^{3/2} \seminorm{H^2(t)}{w} .
  \]
  Bounding the stabilisation as in Lemma~\ref{lem:consistency}, but with \eqi{eq:approximation-bulk}{2} applied to both arguments,
  \begin{equation} \label{eq:superconvergence}
    |\Eh(z; \ul{\mathcal{J}}_\ell w)| \lesssim \h^2 \seminorm{H^2(\dom)}{z} \seminorm{H^2(\dom)}{w} .
  \end{equation}

  \emph{Step 4: duality.} Let $\phi \doteq u_{\Th} - u_g$ and let $u_\phi$ solve \eqref{eq:weak-poisson} with datum $\phi$. By \Aref{ass:regularity}, $u_\phi \in H^2(\dom)$, so $-\Delta u_\phi = \phi$ in $L^2(\dom)$. Split $\norm{L^2(\dom)}{\phi}^2 = (\phi, u_{\Th})_{\Th} - (\phi, u_g)_{\dom}$. Testing \eqref{eq:weak-poisson} for $u_g$ with $u_\phi$ and for $u_\phi$ with $u_g$ gives  $(\phi, u_g)_{\dom} = (\nabla u_g, \nabla u_\phi)_{\dom} = (g, u_\phi)_{\dom}$, while \eqref{eq:interpolator-consistency} evaluated at $\vh = \ulGh g$ gives
  \[
    (\phi, u_{\Th})_{\Th} = \Eh(u_\phi; \ulGh g) + \uah(\ul{\mathcal{J}}_\ell u_\phi, \ulGh g)
    = \Eh(u_\phi; \ulGh g) + (g, \lproj{\Th}{\k+1} u_\phi)_{\Th} ,
  \]
  the last equality by the symmetry of $\uah$ and \eqref{eq:discrete-poisson} tested with $\ul{\mathcal{J}}_\ell u_\phi$. Combining these results and inserting $\pm \ul{\mathcal{J}}_\ell u_g$ in the second argument of $\Eh(u_\phi; \cdot)$, we obtain 
  \[
    \norm{L^2(\dom)}{\phi}^2 = \Eh(u_\phi; \ulGh g - \ul{\mathcal{J}}_\ell u_g) + \Eh(u_\phi; \ul{\mathcal{J}}_\ell u_g) + (g, \lproj{\Th}{\k+1}u_\phi - u_\phi)_{\Th} .
  \]
  Each is bounded by $\h^2 \seminorm{H^2(\dom)}{u_\phi} \norm{L^2(\dom)}{g}$, in turn by \eqref{eq:consistency-bound} with \eqref{eq:energy-error}, by \eqref{eq:superconvergence}, and by \eqi{eq:approximation-bulk}{1}. Assumption~\ref{ass:regularity} gives $\seminorm{H^2(\dom)}{u_\phi} \lesssim \norm{L^2(\dom)}{\phi}$, and dividing concludes.
\end{proof}

\begin{corollary}[Residual representative]\label{cor:residual-representative}
  Let $\lambda \in \Uhb$ and put $f_\lambda \doteq \isch(\Ah \lambda) \in L^2(\dom)$. Then
  \begin{equation} \label{eq:residual-representative}
    \norm{L^2(\dom)}{f_\lambda} \leq \norm{\ell}{\Ah \lambda}, \qquad
    \norm{\ah}{\lambda - \Gh f_\lambda} \leq \h \norm{\ell}{\Ah \lambda} .
  \end{equation}
\end{corollary}
\begin{proof}
  The first bound is \eqi{eq:l2-representative}{1} applied to $\Ah \lambda$. For the second, \eqref{eq:level-operator} and \eqref{def:skeleton-inner-product} give, for all $\mu \in \Uhb$,
  \[
    \ah(\lambda, \mu) = \braket{\Ah\lambda, \mu}_\ell = (f_\lambda, \isch \mu)_{\Th} + \sum_{t \in \Th} h_t \big( r_t(\ulisch \Ah\lambda), r_t(\ulisch \mu) \big)_{\partial t},
  \]
  so that $\ah(\lambda - \Gh f_\lambda, \mu)$ equals the last sum. By Cauchy--Schwarz, \eqref{def:skeleton-inner-product} and \eqi{eq:face-bounds}{1},
  \[
    \Big| \sum_{t \in \Th} h_t \big( r_t(\ulisch \Ah\lambda), r_t(\ulisch \mu) \big)_{\partial t} \Big|
    \leq \norm{\ell}{\Ah\lambda} \Big( \sum_{t \in \Th} h_t \norm{\partial t}{r_t(\ulisch \mu)}^2 \Big)^{1/2}
    \leq \h \norm{\ell}{\Ah\lambda} \norm{\ah}{\mu},
  \]
  and taking $\mu = \lambda - \Gh f_\lambda$ concludes.
\end{proof}

\begin{remark}[The residual as a datum]\label{rem:face-based}
  The multigrid analysis needs to feed the residual $\Ah\lambda$ to the solution operator, but $\Ah\lambda \in \Uhb$ lives on the skeleton and is not a function on $\dom$. Since \eqref{def:skeleton-inner-product} pairs the hybrid extensions $\ulisch \lambda$ and $\ulisch \mu$, its bulk term supplies the datum directly; $f_\lambda = \isch(\Ah\lambda)$ is a broken polynomial, and \eqref{eq:residual-representative} follows from the definition of $\Ah$ with the exact constants of \eqi{eq:l2-representative}{1} and \eqi{eq:face-bounds}{1}. With a purely face-based inner product, as used in \cite{DiPietro2021,DiPietro2024,Lu2021,DiPietroHossjer2026}, with weights $|t|/|\partial t|$ in \cite{DiPietro2021,DiPietro2024}, no such datum is available. The analysis of \cite{DiPietro2021,DiPietro2024} constructs one through a lifting into a conforming subspace, which is what restricts it to simplicial hierarchies; the $p$-multigrid analysis of \cite{DiPietroHossjer2026} likewise relies on a Raviart--Thomas lifting on a simplicial submesh. The two are uniformly equivalent by \eqi{eq:face-bounds}{2} and the fact that $|t|/|\partial t| \simeq h_t$ on star-shaped cells, so the convergence theory is insensitive to the choice. What the bulk term changes is the route to the duality argument.
\end{remark}

\section{The multigrid method}\label{sec:method}

The method is a symmetric V-cycle on the hierarchy of \sect{sec:meshes}, built from two ingredients: an intergrid transfer operator and a smoother. Both are defined here; their analysis occupies \sect{sec:intergrid} and \sect{sec:smoother}.

\subsection{The V-cycle}\label{sec:vcycle}

Let $\IHh : \UHb \to \Uhb$ be a prolongation, chosen in \sect{sec:prolongation}, and let $\Sh : \Uhb \to \Uhb$ be a smoother, chosen in \sect{sec:smoother-def}. The restriction is the adjoint of the prolongation for the inner products \eqref{def:skeleton-inner-product}, which are definite by Lemma~\ref{lem:inner-product},
\begin{equation} \label{eq:restriction}
  \braket{\IhH v, w}_{\ell-1} = \braket{v, \IHh w}_\ell \qquad \forall v \in \Uhb, \ w \in \UHb .
\end{equation}
In the implemented algorithm residuals are vectors in the dual degree-of-freedom basis and this is the transpose of the prolongation matrix, a realisation independent of $\braket{\cdot,\cdot}_\ell$, so \eqref{eq:restriction} costs nothing to realise.

The multigrid operator $\Bh : \Uhb \to \Uhb$ is defined recursively. On the coarsest level $\Bl{0} \doteq \Al{0}^{-1}$. For $\ell \geq 1$ and $b \in \Uhb$, $\Bh b \doteq x^{(2m_\ell)}$, where $x^{(0)} \doteq 0$ and
\begin{subequations} \label{eq:vcycle}
\begin{align}
  x^{(i+1)} &= x^{(i)} + \Sh(b - \Ah x^{(i)}), && i = 0, \ldots, m_\ell - 1, \label{eq:vcycle-pre} \\
  x^{(m_\ell)} &\leftarrow x^{(m_\ell)} + \IHh \BH \IhH (b - \Ah x^{(m_\ell)}), \label{eq:vcycle-coarse} \\
  x^{(i+1)} &= x^{(i)} + \Sh'(b - \Ah x^{(i)}), && i = m_\ell, \ldots, 2m_\ell - 1, \label{eq:vcycle-post}
\end{align}
\end{subequations}
with $m_\ell \geq 1$ the number of smoothing steps and $\Sh'$ the $\braket{\cdot,\cdot}_\ell$-adjoint of $\Sh$. The post-smoothing uses $\Sh'$ so that $\Bh$ is self-adjoint for every choice of smoother; the smoothers of \sect{sec:smoother-def} satisfy $\Sh' = \Sh$. In terms of the error, \eqref{eq:vcycle} reads
\begin{equation} \label{eq:vcycle-error}
  \identity - \Bh \Ah = (\identity - \Sh' \Ah)^{m_\ell} \, (\identity - \IHh \BH \IhH \Ah) \, (\identity - \Sh \Ah)^{m_\ell} .
\end{equation}
Taking $m_\ell = m$ on every level gives the standard V-cycle. The variable V-cycle instead lets the number of smoothing steps grow towards the coarse levels,
\begin{equation}\label{eq:smoothing-steps-condition}
  \rho_1 m_\ell \leq m_{\ell-1} \leq \rho_2 m_\ell \qquad \forall \ell \leq L,
\end{equation}
for constants $1 < \rho_1 \leq \rho_2$ independent of $\ell$. In the implementation the right-hand sides $b$ are not formed level by level; one carries the correction and the residual $b - \Ah x$ simultaneously.

\subsection{Intergrid transfer}\label{sec:prolongation}

Both prolongations reconstruct a polynomial on the coarse cells and project it onto the fine faces. The projection is the weighted $L^2$-projector $\Pavh : H^1(\TH) \to \Uhb$,
\begin{equation}
  \label{def:Pi-av}
  \Pavh(v)|_F \doteq
  \begin{cases}
    \lproj{F}{\k} (v|_T)|_F & F \in \Fhi \setminus \FH, \\
    \alpha_{FT} \lproj{F}{\k} (v|_T)|_F + \alpha_{FT'} \lproj{F}{\k}(v|_{T'})|_F & F \in \Fhi \cap \FH, \\
    0 & F \in \Fhb,
  \end{cases}
\end{equation}
where $T, T' \in \TH$ are the coarse cells sharing $F$ and the weights form a positive partition of unity, $\alpha_{FT}, \alpha_{FT'} > 0$ and $\alpha_{FT} + \alpha_{FT'} = 1$; we take $\alpha_{FT} = |T|/(|T| + |T'|)$. A fine face interior to a coarse cell sees a single coarse polynomial, one lying on a coarse face sees the two adjacent ones, and boundary faces carry no unknown.

\begin{definition}[Prolongations]\label{def:reconstruction-prolongation}
  With $\trH : \UH \to \UHb$ the trace $\trH(v_{\TH}, \mu) \doteq \mu$,
  \begin{equation} \label{eq:prolongations}
    \IUh \doteq \Pavh \, \iscH, \qquad
    \IRh \doteq \Pavh \, \rec{\ell-1}{\K+1} \, \ulisc{\ell-1}, \qquad
    \ul{I}_\ell \doteq \ulisch \, \IHh \, \trH ,
  \end{equation}
  the last being the extension of either prolongation to the hybrid space.
\end{definition}

The two differ in what is projected: $\IUh$ transfers the condensed cell polynomial $\iscH \mu$, $\IRh$ its reconstruction $\rec{\ell-1}{\K+1} \ulisc{\ell-1} \mu$. Both were proposed for standard face spaces in \cite{DiPietro2021,Lu2021,Lu2022}; here they act into the spaces \eqref{eq:interface-space}.

\subsection{Smoothers}\label{sec:smoother-def}

The smoothers are additive Schwarz iterations over patches of face unknowns. The patches are indexed by a set $\mathcal{D}$ of entities of the skeleton, the choice of $\mathcal{D}$ being what distinguishes one smoother from another. For $r \in \mathcal{D}$ let $\mathcal{F}_r \subset \Fhi$ be the interfaces meeting $r$, let
\begin{equation} \label{def:patch-space}
  N_r \doteq \{ \mu \in \Uhb : \mu|_F = 0 \ \ \forall F \in \Fhi \setminus \mathcal{F}_r \},
\end{equation}
and let $\omega_r \subset \dom$ be the interior of the union of the cells meeting $\mathcal{F}_r$. Let $I_r : N_r \to \Uhb$ be the canonical embedding, $I_r'$ its $\braket{\cdot,\cdot}_\ell$-adjoint, and $A_r : N_r \to N_r$ the local operator, $\braket{A_r \mu, \eta}_\ell = \ah(\mu,\eta)$ for $\mu, \eta \in N_r$, with energy norm $\norm{A_r}{\cdot}$. The smoother is
\begin{equation} \label{def:patch-smoother}
  \Sh \doteq \omega \sum_{r \in \mathcal{D}} I_r A_r^{-1} I_r', \qquad \omega > 0 ,
\end{equation}
and each $A_r$ is symmetric positive definite, so $\Sh' = \Sh$. Such patch smoothers are standard in multigrid for $H(\mathrm{div})$ and $H(\mathrm{curl})$ Riesz maps \cite{Arnold1997,Arnold2000}, for elasticity \cite{Schoberl1999} and for Navier--Stokes \cite{Farrell2019}, and are subspace correction methods in the sense of \cite{Xu1992}.

The \emph{face-star} patches take $\mathcal{D} = \Fhi$ and $\mathcal{F}_F = \{F\}$; the local problem is posed on the interior of $\ol{t} \cup \ol{t'}$ with homogeneous conditions on its boundary. The \emph{vertex-star} patches take $\mathcal{D}$ to be the vertices of the skeleton, $\mathcal{F}_r$ collecting the $F \in \Fhi$ that contain a given vertex; in three dimensions the \emph{edge-star} patches take $\mathcal{D}$ to be its edges. Figure~\ref{fig:smoothers} shows the two families. The condensed extension of a function of $N_r$ vanishes outside $\omega_r$, the local problem \eqref{eq:inverse-static-condensation-a} there having zero data, so $\ah(\mu,\eta) = 0$ whenever $\mu \in N_r$, $\eta \in N_{r'}$ and $\omega_r \cap \omega_{r'} = \emptyset$.

\begin{assumption}[Patch families]\label{ass:patches}
  Every interior face belongs to at least one patch, $\bigcup_{r \in \mathcal{D}} \mathcal{F}_r = \Fhi$, and the supports $\{\omega_r\}_{r \in \mathcal{D}}$ can be coloured with at most $\NO$ colours, independently of $\ell$, so that patches of the same colour have disjoint supports.
\end{assumption}

The first condition is automatic for the face-star family. For the vertex- and edge-star families it can fail on degenerate agglomerations, where a cell enclosed by a single neighbour produces a face carrying no vertex, and it is needed for $\Sh$ to be invertible, since otherwise $\sum_r N_r$ is a proper subspace of $\Uhb$. The colouring condition holds for all three families. Two supports meet only if their generating entities lie in a common cell, so the patches conflicting with a given one are those generated by an entity of a cell meeting $\mathcal{F}_r$. Each cell contains a ball of radius $\varrho h_t$ by \eqref{eq:geometry-a} and has diameter $\simeq \h$ by \eqi{eq:geometry-b}{2}, so a bounded number of cells meet any given point, and each carries a bounded number of faces by \eqi{eq:geometry-b}{1}, hence of generating entities. The conflict graph therefore has bounded degree and a greedy colouring uses boundedly many colours.

\begin{figure}
  \centering
  \begin{tikzpicture}[scale=4]

\coordinate (m1) at ( 0.0, 0.125 ) {};
\coordinate (m2) at ( 0.0, 0.0 ) {};
\coordinate (m3) at ( 0.125, 0.0 ) {};
\coordinate (m4) at ( 0.25, 0.0 ) {};
\coordinate (m5) at ( 0.375, 0.0 ) {};
\coordinate (m6) at ( 0.5, 0.0 ) {};
\coordinate (m7) at ( 0.625, 0.0 ) {};
\coordinate (m8) at ( 0.5833333333333334, 0.08333333333333333 ) {};
\coordinate (m9) at ( 0.6666666666666666, 0.16666666666666666 ) {};
\coordinate (m10) at ( 0.5833333333333334, 0.3333333333333333 ) {};
\coordinate (m11) at ( 0.6666666666666666, 0.4166666666666667 ) {};
\coordinate (m12) at ( 0.5833333333333334, 0.5833333333333334 ) {};
\coordinate (m13) at ( 0.4166666666666667, 0.6666666666666666 ) {};
\coordinate (m14) at ( 0.3333333333333333, 0.5833333333333334 ) {};
\coordinate (m15) at ( 0.16666666666666666, 0.6666666666666666 ) {};
\coordinate (m16) at ( 0.08333333333333333, 0.5833333333333334 ) {};
\coordinate (m17) at ( 0.0, 0.625 ) {};
\coordinate (m18) at ( 0.0, 0.5 ) {};
\coordinate (m19) at ( 0.0, 0.375 ) {};
\coordinate (m20) at ( 0.0, 0.25 ) {};
\coordinate (m21) at ( 0.75, 0.0 ) {};
\coordinate (m22) at ( 0.875, 0.0 ) {};
\coordinate (m23) at ( 1.0, 0.0 ) {};
\coordinate (m24) at ( 1.0, 0.125 ) {};
\coordinate (m25) at ( 1.0, 0.25 ) {};
\coordinate (m26) at ( 1.0, 0.375 ) {};
\coordinate (m27) at ( 1.0, 0.5 ) {};
\coordinate (m28) at ( 1.0, 0.625 ) {};
\coordinate (m29) at ( 0.9166666666666666, 0.6666666666666666 ) {};
\coordinate (m30) at ( 0.8333333333333334, 0.5833333333333334 ) {};
\coordinate (m31) at ( 0.6666666666666666, 0.6666666666666666 ) {};
\coordinate (m32) at ( 0.5833333333333334, 0.8333333333333334 ) {};
\coordinate (m33) at ( 0.6666666666666666, 0.9166666666666666 ) {};
\coordinate (m34) at ( 0.625, 1.0 ) {};
\coordinate (m35) at ( 0.5, 1.0 ) {};
\coordinate (m36) at ( 0.375, 1.0 ) {};
\coordinate (m37) at ( 0.25, 1.0 ) {};
\coordinate (m38) at ( 0.125, 1.0 ) {};
\coordinate (m39) at ( 0.0, 1.0 ) {};
\coordinate (m40) at ( 0.0, 0.875 ) {};
\coordinate (m41) at ( 0.0, 0.75 ) {};
\coordinate (m42) at ( 1.0, 0.75 ) {};
\coordinate (m43) at ( 1.0, 0.875 ) {};
\coordinate (m44) at ( 1.0, 1.0 ) {};
\coordinate (m45) at ( 0.875, 1.0 ) {};
\coordinate (m46) at ( 0.75, 1.0 ) {};

\draw[thick] (m1) -- (m2);
\draw[thick] (m2) -- (m3);
\draw[thick] (m3) -- (m4);
\draw[thick] (m4) -- (m5);
\draw[thick] (m5) -- (m6);
\draw[thick] (m6) -- (m7);
\draw[thick] (m7) -- (m8);
\draw[thick] (m8) -- (m9);
\draw[thick] (m9) -- (m10);
\draw[thick] (m10) -- (m11);
\draw[thick] (m11) -- (m12);
\draw[thick] (m12) -- (m13);
\draw[thick] (m13) -- (m14);
\draw[thick] (m14) -- (m15);
\draw[thick] (m15) -- (m16);
\draw[thick] (m16) -- (m17);
\draw[thick] (m17) -- (m18);
\draw[thick] (m18) -- (m19);
\draw[thick] (m19) -- (m20);
\draw[thick] (m20) -- (m1);
\draw[thick] (m7) -- (m21);
\draw[thick] (m21) -- (m22);
\draw[thick] (m22) -- (m23);
\draw[thick] (m23) -- (m24);
\draw[thick] (m24) -- (m25);
\draw[thick] (m25) -- (m26);
\draw[thick] (m26) -- (m27);
\draw[thick] (m27) -- (m28);
\draw[thick] (m28) -- (m29);
\draw[thick] (m29) -- (m30);
\draw[thick] (m30) -- (m31);
\draw[thick] (m31) -- (m12);
\draw[thick] (m31) -- (m32);
\draw[thick] (m32) -- (m33);
\draw[thick] (m33) -- (m34);
\draw[thick] (m34) -- (m35);
\draw[thick] (m35) -- (m36);
\draw[thick] (m36) -- (m37);
\draw[thick] (m37) -- (m38);
\draw[thick] (m38) -- (m39);
\draw[thick] (m39) -- (m40);
\draw[thick] (m40) -- (m41);
\draw[thick] (m41) -- (m17);
\draw[thick] (m28) -- (m42);
\draw[thick] (m42) -- (m43);
\draw[thick] (m43) -- (m44);
\draw[thick] (m44) -- (m45);
\draw[thick] (m45) -- (m46);
\draw[thick] (m46) -- (m34);

\coordinate (p1_1) at ( 0.02, 0.125 ) {};
\coordinate (p1_2) at ( 0.014142135623730949, 0.014142135623730949 ) {};
\coordinate (p1_3) at ( 0.125, 0.02 ) {};
\coordinate (p1_4) at ( 0.25, 0.02 ) {};
\coordinate (p1_5) at ( 0.375, 0.02 ) {};
\coordinate (p1_6) at ( 0.5, 0.02 ) {};
\coordinate (p1_7) at ( 0.625, 0.02 ) {};
\coordinate (p1_8) at ( 0.75, 0.02 ) {};
\coordinate (p1_9) at ( 0.875, 0.02 ) {};
\coordinate (p1_10) at ( 0.9858578643762691, 0.014142135623730949 ) {};
\coordinate (p1_11) at ( 0.98, 0.125 ) {};
\coordinate (p1_12) at ( 0.98, 0.25 ) {};
\coordinate (p1_13) at ( 0.98, 0.375 ) {};
\coordinate (p1_14) at ( 0.98, 0.5 ) {};
\coordinate (p1_15) at ( 0.9829869838329592, 0.6144853777576174 ) {};
\coordinate (p1_16) at ( 0.919870311526806, 0.6469249175139167 ) {};
\coordinate (p1_17) at ( 0.8365369781934727, 0.5635915841805834 ) {};
\coordinate (p1_18) at ( 0.669870311526806, 0.6469249175139167 ) {};
\coordinate (p1_19) at ( 0.5865369781934727, 0.5635915841805834 ) {};
\coordinate (p1_20) at ( 0.41987031152680604, 0.6469249175139167 ) {};
\coordinate (p1_21) at ( 0.33653697819347267, 0.5635915841805834 ) {};
\coordinate (p1_22) at ( 0.169870311526806, 0.6469249175139167 ) {};
\coordinate (p1_23) at ( 0.08653697819347267, 0.5635915841805834 ) {};
\coordinate (p1_24) at ( 0.010514622242382676, 0.6079869838329592 ) {};
\coordinate (p1_25) at ( 0.02, 0.5 ) {};
\coordinate (p1_26) at ( 0.02, 0.375 ) {};
\coordinate (p1_27) at ( 0.02, 0.25 ) {};

\fill[blue, opacity=0.3, draw=none] (p1_1) -- (p1_2) -- (p1_3) -- (p1_4) -- (p1_5) -- (p1_6) -- (p1_7) -- (p1_8) -- (p1_9) -- (p1_10) -- (p1_11) -- (p1_12) -- (p1_13) -- (p1_14) -- (p1_15) -- (p1_16) -- (p1_17) -- (p1_18) -- (p1_19) -- (p1_20) -- (p1_21) -- (p1_22) -- (p1_23) -- (p1_24) -- (p1_25) -- (p1_26) -- (p1_27) -- cycle;

\draw[very thick, black!20!blue] (m7) -- (m8);
\draw[very thick, black!20!blue] (m8) -- (m9);
\draw[very thick, black!20!blue] (m9) -- (m10);
\draw[very thick, black!20!blue] (m10) -- (m11);
\draw[very thick, black!20!blue] (m11) -- (m12);
\coordinate (p2_1) at ( 0.6030750824860833, 0.08012968847319399 ) {};
\coordinate (p2_2) at ( 0.6355146222423826, 0.017013016167040797 ) {};
\coordinate (p2_3) at ( 0.75, 0.02 ) {};
\coordinate (p2_4) at ( 0.875, 0.02 ) {};
\coordinate (p2_5) at ( 0.9858578643762691, 0.014142135623730949 ) {};
\coordinate (p2_6) at ( 0.98, 0.125 ) {};
\coordinate (p2_7) at ( 0.98, 0.25 ) {};
\coordinate (p2_8) at ( 0.98, 0.375 ) {};
\coordinate (p2_9) at ( 0.98, 0.5 ) {};
\coordinate (p2_10) at ( 0.9829869838329592, 0.6144853777576174 ) {};
\coordinate (p2_11) at ( 0.919870311526806, 0.6469249175139167 ) {};
\coordinate (p2_12) at ( 0.8365369781934727, 0.5635915841805834 ) {};
\coordinate (p2_13) at ( 0.6525245310429357, 0.6525245310429357 ) {};
\coordinate (p2_14) at ( 0.5635915841805834, 0.8365369781934727 ) {};
\coordinate (p2_15) at ( 0.6469249175139167, 0.919870311526806 ) {};
\coordinate (p2_16) at ( 0.6144853777576174, 0.9829869838329592 ) {};
\coordinate (p2_17) at ( 0.5, 0.98 ) {};
\coordinate (p2_18) at ( 0.375, 0.98 ) {};
\coordinate (p2_19) at ( 0.25, 0.98 ) {};
\coordinate (p2_20) at ( 0.125, 0.98 ) {};
\coordinate (p2_21) at ( 0.014142135623730949, 0.9858578643762691 ) {};
\coordinate (p2_22) at ( 0.02, 0.875 ) {};
\coordinate (p2_23) at ( 0.02, 0.75 ) {};
\coordinate (p2_24) at ( 0.017013016167040797, 0.6355146222423826 ) {};
\coordinate (p2_25) at ( 0.08012968847319399, 0.6030750824860833 ) {};
\coordinate (p2_26) at ( 0.1634630218065273, 0.6864084158194166 ) {};
\coordinate (p2_27) at ( 0.33012968847319396, 0.6030750824860833 ) {};
\coordinate (p2_28) at ( 0.41346302180652733, 0.6864084158194166 ) {};
\coordinate (p2_29) at ( 0.5974754689570643, 0.5974754689570643 ) {};
\coordinate (p2_30) at ( 0.6864084158194166, 0.41346302180652733 ) {};
\coordinate (p2_31) at ( 0.6030750824860833, 0.33012968847319396 ) {};
\coordinate (p2_32) at ( 0.6864084158194166, 0.1634630218065273 ) {};

\fill[red, opacity=0.3, draw=none] (p2_1) -- (p2_2) -- (p2_3) -- (p2_4) -- (p2_5) -- (p2_6) -- (p2_7) -- (p2_8) -- (p2_9) -- (p2_10) -- (p2_11) -- (p2_12) -- (p2_13) -- (p2_14) -- (p2_15) -- (p2_16) -- (p2_17) -- (p2_18) -- (p2_19) -- (p2_20) -- (p2_21) -- (p2_22) -- (p2_23) -- (p2_24) -- (p2_25) -- (p2_26) -- (p2_27) -- (p2_28) -- (p2_29) -- (p2_30) -- (p2_31) -- (p2_32) -- cycle;

\draw[very thick, black!20!red] (m31) -- (m12);
\end{tikzpicture}
  \begin{tikzpicture}[scale=4]

\coordinate (m1) at ( 0.0, 0.125 ) {};
\coordinate (m2) at ( 0.0, 0.0 ) {};
\coordinate (m3) at ( 0.125, 0.0 ) {};
\coordinate (m4) at ( 0.25, 0.0 ) {};
\coordinate (m5) at ( 0.375, 0.0 ) {};
\coordinate (m6) at ( 0.5, 0.0 ) {};
\coordinate (m7) at ( 0.625, 0.0 ) {};
\coordinate (m8) at ( 0.5833333333333334, 0.08333333333333333 ) {};
\coordinate (m9) at ( 0.6666666666666666, 0.16666666666666666 ) {};
\coordinate (m10) at ( 0.5833333333333334, 0.3333333333333333 ) {};
\coordinate (m11) at ( 0.6666666666666666, 0.4166666666666667 ) {};
\coordinate (m12) at ( 0.5833333333333334, 0.5833333333333334 ) {};
\coordinate (m13) at ( 0.4166666666666667, 0.6666666666666666 ) {};
\coordinate (m14) at ( 0.3333333333333333, 0.5833333333333334 ) {};
\coordinate (m15) at ( 0.16666666666666666, 0.6666666666666666 ) {};
\coordinate (m16) at ( 0.08333333333333333, 0.5833333333333334 ) {};
\coordinate (m17) at ( 0.0, 0.625 ) {};
\coordinate (m18) at ( 0.0, 0.5 ) {};
\coordinate (m19) at ( 0.0, 0.375 ) {};
\coordinate (m20) at ( 0.0, 0.25 ) {};
\coordinate (m21) at ( 0.75, 0.0 ) {};
\coordinate (m22) at ( 0.875, 0.0 ) {};
\coordinate (m23) at ( 1.0, 0.0 ) {};
\coordinate (m24) at ( 1.0, 0.125 ) {};
\coordinate (m25) at ( 1.0, 0.25 ) {};
\coordinate (m26) at ( 1.0, 0.375 ) {};
\coordinate (m27) at ( 1.0, 0.5 ) {};
\coordinate (m28) at ( 1.0, 0.625 ) {};
\coordinate (m29) at ( 0.9166666666666666, 0.6666666666666666 ) {};
\coordinate (m30) at ( 0.8333333333333334, 0.5833333333333334 ) {};
\coordinate (m31) at ( 0.6666666666666666, 0.6666666666666666 ) {};
\coordinate (m32) at ( 0.5833333333333334, 0.8333333333333334 ) {};
\coordinate (m33) at ( 0.6666666666666666, 0.9166666666666666 ) {};
\coordinate (m34) at ( 0.625, 1.0 ) {};
\coordinate (m35) at ( 0.5, 1.0 ) {};
\coordinate (m36) at ( 0.375, 1.0 ) {};
\coordinate (m37) at ( 0.25, 1.0 ) {};
\coordinate (m38) at ( 0.125, 1.0 ) {};
\coordinate (m39) at ( 0.0, 1.0 ) {};
\coordinate (m40) at ( 0.0, 0.875 ) {};
\coordinate (m41) at ( 0.0, 0.75 ) {};
\coordinate (m42) at ( 1.0, 0.75 ) {};
\coordinate (m43) at ( 1.0, 0.875 ) {};
\coordinate (m44) at ( 1.0, 1.0 ) {};
\coordinate (m45) at ( 0.875, 1.0 ) {};
\coordinate (m46) at ( 0.75, 1.0 ) {};

\draw[thick] (m1) -- (m2);
\draw[thick] (m2) -- (m3);
\draw[thick] (m3) -- (m4);
\draw[thick] (m4) -- (m5);
\draw[thick] (m5) -- (m6);
\draw[thick] (m6) -- (m7);
\draw[thick] (m7) -- (m8);
\draw[thick] (m8) -- (m9);
\draw[thick] (m9) -- (m10);
\draw[thick] (m10) -- (m11);
\draw[thick] (m11) -- (m12);
\draw[thick] (m12) -- (m13);
\draw[thick] (m13) -- (m14);
\draw[thick] (m14) -- (m15);
\draw[thick] (m15) -- (m16);
\draw[thick] (m16) -- (m17);
\draw[thick] (m17) -- (m18);
\draw[thick] (m18) -- (m19);
\draw[thick] (m19) -- (m20);
\draw[thick] (m20) -- (m1);
\draw[thick] (m7) -- (m21);
\draw[thick] (m21) -- (m22);
\draw[thick] (m22) -- (m23);
\draw[thick] (m23) -- (m24);
\draw[thick] (m24) -- (m25);
\draw[thick] (m25) -- (m26);
\draw[thick] (m26) -- (m27);
\draw[thick] (m27) -- (m28);
\draw[thick] (m28) -- (m29);
\draw[thick] (m29) -- (m30);
\draw[thick] (m30) -- (m31);
\draw[thick] (m31) -- (m12);
\draw[thick] (m31) -- (m32);
\draw[thick] (m32) -- (m33);
\draw[thick] (m33) -- (m34);
\draw[thick] (m34) -- (m35);
\draw[thick] (m35) -- (m36);
\draw[thick] (m36) -- (m37);
\draw[thick] (m37) -- (m38);
\draw[thick] (m38) -- (m39);
\draw[thick] (m39) -- (m40);
\draw[thick] (m40) -- (m41);
\draw[thick] (m41) -- (m17);
\draw[thick] (m28) -- (m42);
\draw[thick] (m42) -- (m43);
\draw[thick] (m43) -- (m44);
\draw[thick] (m44) -- (m45);
\draw[thick] (m45) -- (m46);
\draw[thick] (m46) -- (m34);

\coordinate (p1_1) at ( 0.6030750824860833, 0.08012968847319399 ) {};
\coordinate (p1_2) at ( 0.6355146222423826, 0.017013016167040797 ) {};
\coordinate (p1_3) at ( 0.75, 0.02 ) {};
\coordinate (p1_4) at ( 0.875, 0.02 ) {};
\coordinate (p1_5) at ( 0.9858578643762691, 0.014142135623730949 ) {};
\coordinate (p1_6) at ( 0.98, 0.125 ) {};
\coordinate (p1_7) at ( 0.98, 0.25 ) {};
\coordinate (p1_8) at ( 0.98, 0.375 ) {};
\coordinate (p1_9) at ( 0.98, 0.5 ) {};
\coordinate (p1_10) at ( 0.98, 0.625 ) {};
\coordinate (p1_11) at ( 0.98, 0.75 ) {};
\coordinate (p1_12) at ( 0.98, 0.875 ) {};
\coordinate (p1_13) at ( 0.9858578643762691, 0.9858578643762691 ) {};
\coordinate (p1_14) at ( 0.875, 0.98 ) {};
\coordinate (p1_15) at ( 0.75, 0.98 ) {};
\coordinate (p1_16) at ( 0.625, 0.98 ) {};
\coordinate (p1_17) at ( 0.5, 0.98 ) {};
\coordinate (p1_18) at ( 0.375, 0.98 ) {};
\coordinate (p1_19) at ( 0.25, 0.98 ) {};
\coordinate (p1_20) at ( 0.125, 0.98 ) {};
\coordinate (p1_21) at ( 0.014142135623730949, 0.9858578643762691 ) {};
\coordinate (p1_22) at ( 0.02, 0.875 ) {};
\coordinate (p1_23) at ( 0.02, 0.75 ) {};
\coordinate (p1_24) at ( 0.017013016167040797, 0.6355146222423826 ) {};
\coordinate (p1_25) at ( 0.08012968847319399, 0.6030750824860833 ) {};
\coordinate (p1_26) at ( 0.1634630218065273, 0.6864084158194166 ) {};
\coordinate (p1_27) at ( 0.33012968847319396, 0.6030750824860833 ) {};
\coordinate (p1_28) at ( 0.41346302180652733, 0.6864084158194166 ) {};
\coordinate (p1_29) at ( 0.5974754689570643, 0.5974754689570643 ) {};
\coordinate (p1_30) at ( 0.6864084158194166, 0.41346302180652733 ) {};
\coordinate (p1_31) at ( 0.6030750824860833, 0.33012968847319396 ) {};
\coordinate (p1_32) at ( 0.6864084158194166, 0.1634630218065273 ) {};

\fill[green, opacity=0.3, draw=none] (p1_1) -- (p1_2) -- (p1_3) -- (p1_4) -- (p1_5) -- (p1_6) -- (p1_7) -- (p1_8) -- (p1_9) -- (p1_10) -- (p1_11) -- (p1_12) -- (p1_13) -- (p1_14) -- (p1_15) -- (p1_16) -- (p1_17) -- (p1_18) -- (p1_19) -- (p1_20) -- (p1_21) -- (p1_22) -- (p1_23) -- (p1_24) -- (p1_25) -- (p1_26) -- (p1_27) -- (p1_28) -- (p1_29) -- (p1_30) -- (p1_31) -- (p1_32) -- cycle;

\draw[very thick, black!30!green] (m28) -- (m29);
\draw[very thick, black!30!green] (m29) -- (m30);
\draw[very thick, black!30!green] (m30) -- (m31);
\draw[very thick, black!30!green] (m31) -- (m12);
\draw[very thick, black!30!green] (m31) -- (m32);
\draw[very thick, black!30!green] (m32) -- (m33);
\draw[very thick, black!30!green] (m33) -- (m34);
\node[fill=black!50!green, inner sep=0pt, minimum size=6pt] at (m31) {};

\end{tikzpicture}
  \caption{The two star-patch families: two face-star patches (red and blue) and a vertex-star patch (green). Generating entities are drawn thick; only the unknowns on the generating faces are active.}
  \label{fig:smoothers}
\end{figure}

\section{Convergence framework}\label{sec:framework}

We use the framework of \cite{Duan2007AGB}, which extends \cite{Bramble1991} to hierarchies whose coarse forms are not inherited from the fine one. It applies to the V-cycle \eqref{eq:vcycle} through three hypotheses on the ingredients of \sect{sec:method}, stated in \sect{sec:framework-assumptions} and reduced in \sect{sec:framework-roadmap} to properties proved in \sect{sec:intergrid} and \sect{sec:smoother}.

\subsection{The three hypotheses}\label{sec:framework-assumptions}

Let $\Ph : \Uhb \to \UHb$ be the coarse energy projection, defined by
\begin{equation} \label{eq:coarse-projection}
  \aH(\Ph \lambda, \mu) = \ah(\lambda, \IHh \mu) \qquad \forall \lambda \in \Uhb, \ \mu \in \UHb,
\end{equation}
and let
\begin{equation} \label{eq:symmetrised-smoother}
  \Kh \doteq \big( \identity - (\identity - \Sh' \Ah)(\identity - \Sh \Ah) \big) \Ah^{-1}
\end{equation}
be the operator of one complete pre- and post-smoothing step applied to a residual.

\begin{assumption}[Smoothing]\label{ass:smoothing}
  {\assnum
  \begin{equation} \label{eq:ass-smoothing}
    \h^2 \norm{\ell}{\lambda}^2 \leq C_S \braket{\Kh \lambda, \lambda}_\ell \qquad \forall \lambda \in \Uhb .
  \end{equation}}
\end{assumption}

\begin{assumption}[Bounded coarse correction]\label{ass:bounded}
  {\assnum
  \begin{equation} \label{eq:ass-bounded}
    \norm{\ah}{\lambda - \IHh \Ph \lambda} \leq C_B \norm{\ah}{\lambda} \qquad \forall \lambda \in \Uhb .
  \end{equation}}
\end{assumption}

\begin{assumption}[Regularity and approximation]\label{ass:regularity-approx}
  {\assnum
  \begin{equation} \label{eq:ass-regularity}
    \big| \ah(\lambda - \IHh \Ph \lambda, \lambda) \big| \leq C_R \, \h^2 \norm{\ell}{\Ah \lambda}^2 \qquad \forall \lambda \in \Uhb .
  \end{equation}}
\end{assumption}

\begin{theorem}[V-cycle, {\cite[Theorem 3.1]{Duan2007AGB}}]\label{th:vcycle}
  Let Assumptions~\ref{ass:smoothing}--\ref{ass:regularity-approx} hold and let $m_\ell = m > 2 C_R C_S$ on every level. Then
  \begin{equation} \label{eq:vcycle-contraction}
    \big| \ah((\identity - \Bh \Ah)\lambda, \lambda) \big| \leq \delta \, \ah(\lambda,\lambda) \qquad \forall \lambda \in \Uhb, \qquad
    \delta = \frac{C_R C_S}{m - C_R C_S} < 1,
  \end{equation}
  with $\delta$ independent of $\ell$.
\end{theorem}

\begin{theorem}[Variable V-cycle, {\cite[Theorem 6]{Bramble1991}}]\label{th:variable-vcycle}
  Let Assumptions~\ref{ass:smoothing} and \ref{ass:regularity-approx} hold and let $m_\ell$ satisfy \eqref{eq:smoothing-steps-condition}. Then
  \begin{equation} \label{eq:variable-vcycle}
    \eta_0 \, \ah(\lambda,\lambda) \leq \ah(\Bh \Ah \lambda, \lambda) \leq \eta_1 \, \ah(\lambda,\lambda) \qquad \forall \lambda \in \Uhb,
  \end{equation}
  with $\eta_0 \geq m_\ell / (M + m_\ell)$ and $\eta_1 \leq (M + m_\ell)/m_\ell$, and $M$ independent of $\ell$.
\end{theorem}

\begin{remark}[Normalisation]\label{rem:lambda-normalisation}
  The hypotheses of \cite{Duan2007AGB,Bramble1991} carry $1/\eigenmax$ where \eqref{eq:ass-smoothing} and \eqref{eq:ass-regularity} carry $\h^2$. The eigenvalue enters their proofs only through the quotient $\norm{\ell}{\Ah\lambda}^2 / \eigenmax$, which the regularity and approximation hypothesis produces and the smoothing hypothesis bounds immediately afterwards, at \cite[(3.9)--(3.10)]{Bramble1991} and at \cite[(7)--(8)]{Duan2007AGB}. Any positive sequence may therefore be substituted in the two hypotheses at once, and $\h^{-2}$ is the natural choice by \eqref{eq:max-eigenvalue}. Stated with $1/\eigenmax$, \eqref{eq:ass-smoothing} would instead require the lower bound $\eigenmax \gtrsim \h^{-2}$, which we do not prove.
\end{remark}

\begin{remark}[Nonsymmetric smoothers]\label{rem:symmetric-smoothers}
  The framework admits a nonsymmetric smoother, alternating $\Sh$ and $\Sh'$ so that $\Bh$ stays self-adjoint. For the smoothers of \sect{sec:smoother-def} one has $\Sh' = \Sh$ and $\Kh = 2\Sh - \Sh \Ah \Sh$. Keeping $\Sh'$ in \eqref{eq:symmetrised-smoother} costs nothing and covers a multiplicative patch sweep, a forward sweep before and a backward sweep after the coarse correction being such an adjoint pair.
\end{remark}

\subsection{Reduction to the ingredients}\label{sec:framework-roadmap}

Assumptions~\ref{ass:smoothing}--\ref{ass:regularity-approx} are consequences of the following properties of the prolongation and of the smoother.

\begin{assumption}[Prolongation]\label{ass:prolongation}
   For all $\mu \in \UHb$ and all $z \in H^1_0(\dom) \cap H^2(\dom)$,
  \begin{subequations} \asssubnum \label{eq:ass-prolongation}
  \begin{alignat}{2}
    \norm{\ah}{\IHh \mu} &\leq C_I \norm{\aH}{\mu},
      &\qquad \norm{\ell}{\IHh \mu} &\lesssim \norm{\ell-1}{\mu}, \label{eq:ass-prolongation-stab} \\
    \norm{\uah}{\ul{\mathcal{J}}_\ell z - \ul{I}_\ell \ul{\mathcal{J}}_{\ell-1} z} &\lesssim \H \seminorm{H^2(\dom)}{z},
      &\qquad \norm{L^2(\dom)}{\iscH \mu - \isch \IHh \mu} &\lesssim \H \norm{\aH}{\mu}, \label{eq:ass-prolongation-first} \\
    \norm{\ell}{\lproj{\Fh}{\k} z - \IHh \lproj{\FH}{\K} z} &\lesssim \H^2 \seminorm{H^2(\dom)}{z} . &&
      \label{eq:ass-prolongation-approx}
  \end{alignat}
  \end{subequations}
\end{assumption}
\begin{assumption}[Smoother]\label{ass:smoother}
  The patch operator $\Dh \doteq \omega^{-1} \Sh \Ah$ satisfies $\rho(\Dh) \leq c_0$, and every $\lambda \in \Uhb$ admits a decomposition $\lambda = \sum_{r \in \mathcal{D}} \lambda_r$ with $\lambda_r \in N_r$ and
  {\assnum
  \begin{equation} \label{eq:ass-smoother}
    \sum_{r \in \mathcal{D}} \norm{A_r}{\lambda_r}^2 \leq c_1 \h^{-2} \norm{\ell}{\lambda}^2 .
  \end{equation}}
\end{assumption}

The first relation of \eqref{eq:ass-prolongation-stab} is energy stability of the prolongation and the second its counterpart in the skeletal norms. The bounds \eqref{eq:ass-prolongation-first} are first-order consistency estimates, and \eqref{eq:ass-prolongation-approx} is their second-order counterpart.

Assumption~\ref{ass:prolongation} is established for both prolongations of \sect{sec:prolongation} in \sect{sec:intergrid}, and Assumption~\ref{ass:smoother} for the patch families of \sect{sec:smoother-def} in \sect{sec:smoother}, where it is also shown to imply Assumption~\ref{ass:smoothing}. Assumptions~\ref{ass:bounded} and \ref{ass:regularity-approx} are derived from Assumption~\ref{ass:prolongation} in \sect{sec:main-result}, which combines them with Theorems~\ref{th:vcycle} and \ref{th:variable-vcycle}. The route from  \eqref{eq:ass-prolongation-approx} to Assumption~\ref{ass:regularity-approx} follows \cite[Section 4]{Duan2007AGB}, except for their hypothesis H5), that every level norm be equivalent to a single $L^2(\dom)$ norm. It fails for \eqref{def:skeleton-inner-product}, whose face term lives on level-dependent skeletons, so in \sect{sec:main-result} the two terms of \eqref{def:skeleton-inner-product} are treated separately, the face one through \eqi{eq:face-bounds}{1} at the cost of one factor $\h$ against the energy.

\section{Analysis of the intergrid transfer operators}\label{sec:intergrid}

This section proves Assumption~\ref{ass:prolongation} for $\IUh$ and $\IRh$. Both are of the form $\Pavh \Theta_{\ell-1}$ with $\Theta_{\ell-1} : \UHb \to \Pk{\K+1}(\TH)$ a coarse reconstruction, so the bounds are first established for an arbitrary such $\Theta_{\ell-1}$ and then verified for the two choices.

\subsection{An averaging operator}\label{sec:averaging}

The prolongations project a broken coarse polynomial onto the fine faces. The following operator carries the same information into the whole hybrid space and is the workhorse of the section: for $z = (z_T)_{T \in \TH} \in H^1(\TH)$, let $\ul{\mathcal{W}}_\ell z \in \Uh$ have components
\begin{equation} \label{def:Wh}
  (\ul{\mathcal{W}}_\ell z)|_t \doteq \lproj{t}{\k+1} z_T \quad (t \in \Th(T)), \qquad
  \ul{\mathcal{W}}_\ell z = \big( \lproj{\Th}{\k+1} z, \Pavh z \big),
\end{equation}
its face component being $\Pavh z$ of \eqref{def:Pi-av}. It differs from $\ulisch \Pavh z$ only in the cell component, where the latter takes the discrete-harmonic extension instead of the local projection.

On each coarse cell, $\ul{\mathcal{W}}_\ell$ is a fine interpolation plus a skeletal correction that sees only the jumps of $z$:
\begin{equation}\label{def:WhT-decomp}
  (\ul{\mathcal{W}}_\ell z)|_T = \ul{\mathcal{J}}_\ell z_T + \ul{\xi}_T, \qquad \ul{\xi}_T = (0, \xi_T),
\end{equation}
where, by \eqref{def:Pi-av} and \eqref{eq:jump},
\begin{equation} \label{def:xiT}
  \xi_{T,F} =
  \begin{cases}
    0 & F \in \Fhi \setminus \FH, \\
    \alpha_{FT'} \lproj{F}{\k}(z_{T'} - z_T) & F \subset \partial T \cap \partial T', \ T' \in \TH, \\
    -\lproj{F}{\k} z_T & F \subset \partial T \cap \partial\dom,
  \end{cases}
  \qquad \text{so that} \qquad
  \norm{F}{\xi_{T,F}} \leq \norm{F}{\jump{z}} ,
\end{equation}
the last bound by the $L^2(F)$-contractivity of $\lproj{F}{\k}$ and $\alpha_{FT'} \leq 1$.

\begin{lemma}[Stability of the averaging operator]\label{lem:Wh-stability}
  For all $z \in H^1(\TH)$,
  \begin{equation} \label{eq:Wh-stability}
    \norm{\ul{1,\ell}}{\ul{\mathcal{W}}_\ell z}^2 \lesssim \seminorm{H^1(\TH)}{z}^2 + \h^{-1} \norm{\FH}{\jump{z}}^2 .
  \end{equation}
\end{lemma}
\begin{proof}
  By \eqref{def:WhT-decomp} the norm splits into a contribution of $\ul{\mathcal{J}}_\ell z_T$, bounded by $\seminorm{1,T}{z_T}$ through \eqref{eq:interpolation-bounded}, and one of $\ul{\xi}_T$. The latter has zero cell component and vanishes on the fine faces interior to $T$, so by \eqref{def:hho-norm}, quasi-uniformity and \eqref{def:xiT},
  \begin{equation} \label{eq:xiT-norm-bound}
    \norm{\ul{1,\ell},T}{\ul{\xi}_T}^2 = \sum_{t \in \Th(T)} h_t^{-1} \norm{\partial t}{\xi_T}^2
    \simeq \h^{-1} \sum_{F \subset \partial T} \norm{F}{\xi_{T,F}}^2
    \leq \h^{-1} \sum_{F \subset \partial T} \norm{F}{\jump{z}}^2 .
  \end{equation}
  Summing over $T \in \TH$, the fine faces on $\partial T$ tile the coarse faces of $T$ and each is counted at most twice.
\end{proof}

\subsection{Bubble bound for coarse polynomials}\label{sec:coarse-bubble}

\begin{lemma}[Coarse polynomial bubble bound]\label{lem:polynomial-bubble}
  For $z \in \Pk{\K+1}(\TH)$,
  \begin{subequations} \label{eq:polynomial-bubble}
  \begin{align}
    \norm{\uah}{\ul{\mathcal{W}}_\ell z - \ulisch \Pavh z} &\lesssim \seminorm{H^1(\TH)}{z} + \h^{-1/2} \norm{\FH}{\jump{z}}, \label{eq:polynomial-bubble-1} \\
    \norm{\uah}{\ul{\mathcal{W}}_\ell z - \ulisch \Pavh z} &\lesssim \h \seminorm{H^2(\TH)}{z} + \h^{-1/2} \norm{\FH}{\jump{z}} . \label{eq:polynomial-bubble-2}
  \end{align}
  \end{subequations}
\end{lemma}
\begin{proof}
  The difference $\ul{\delta}_\ell \doteq \ul{\mathcal{W}}_\ell z - \ulisch \Pavh z$ has zero face component by \eqref{def:Wh}, and $\ulisch \Pavh z$ is discrete harmonic, so $\uah(\ulisch \Pavh z, \ul{\delta}_\ell) = 0$ by \eqref{eq:inverse-static-condensation-a} and thus
  \[
    \norm{\uah}{\ul{\delta}_\ell}^2 = \uah(\ul{\mathcal{W}}_\ell z, \ul{\delta}_\ell)
    \just{=}{\eqref{def:WhT-decomp}} \sum_{t \in \Th} \ual{\ell,t}(\ul{\mathcal{J}}_{\ell,t} z, \ul{\delta}_t) + \sum_{T \in \TH} \sum_{t \in \Th(T)} \ual{\ell,t}(\ul{\xi}_T|_t, \ul{\delta}_t) .
  \]
  In the first sum $z|_t \in \Pk{\k+1}(t)$, so $\lproj{t}{\k+1} z = z$. The face residual \eqref{def:face-residual} of $\ul{\mathcal{J}}_{\ell,t} z$ is then $\lproj{\partial t}{\k} z - \lproj{\partial t}{\k} \lproj{t}{\k+1} z = 0$, so the stabilisation contributes nothing, and \eqref{eq:elliptic-identity} gives $\rec{\ell,t}{\k+1} \ul{\mathcal{J}}_{\ell,t} z = \lproj{t}{1,\k+1} z = z$. What remains is $\sum_{t \in \Th} (\nabla z, \nabla \rec{\ell,t}{\k+1} \ul{\delta}_t)_t$. Testing \eqref{eq:reconstruction-a} with $w = z|_t \in \Pk{\k+1}(t)$ and integrating the resulting bulk term by parts, the two boundary terms cancel, $\ul{\delta}_t$ having zero face component:
  \begin{equation} \label{eq:Wh-ibp}
    \sum_{t \in \Th} (\nabla z, \nabla \rec{\ell,t}{\k+1} \ul{\delta}_t)_t = - \sum_{t \in \Th} (\delta_t, \Delta z)_t .
  \end{equation}
  Cauchy--Schwarz bounds the left-hand side by $\seminorm{H^1(\TH)}{z} \norm{\uah}{\ul{\delta}_\ell}$, and the right-hand side, with \eqref{eq:bubble-l2}, by $\h \seminorm{H^2(\TH)}{z} \norm{\uah}{\ul{\delta}_\ell}$. The second sum is bounded by $\norm{\ul{1,\ell},T}{\ul{\xi}_T} \norm{\ual{\ell,t}}{\ul{\delta}_t}$ through \eqref{eq:local-continuity}. Dividing by $\norm{\uah}{\ul{\delta}_\ell}$ and invoking \eqref{eq:xiT-norm-bound} gives \eqref{eq:polynomial-bubble-1} with the first choice and \eqref{eq:polynomial-bubble-2} with the second.
\end{proof}

\subsection{An abstract framework}\label{sec:abstract}

\begin{lemma}[Abstract bounds]\label{lem:abstract}
  Let $\Theta_{\ell-1} : \UHb \to \Pk{\K+1}(\TH)$ and set $\IHh \doteq \Pavh \Theta_{\ell-1}$. If, for all $\mu \in \UHb$,
  \begin{equation} \label{eq:theta-hyp}
    \seminorm{H^1(\TH)}{\Theta_{\ell-1} \mu} \lesssim \norm{\aH}{\mu}, \qquad
    \norm{\FH}{\jump{\Theta_{\ell-1}\mu}}^2 \lesssim \H \norm{\aH}{\mu}^2,
  \end{equation}
  then $\IHh$ satisfies \eqi{eq:ass-prolongation-stab}{1} and
  \begin{equation} \label{eq:abstract-l2}
    \norm{L^2(\dom)}{\Theta_{\ell-1} \mu - \isch \IHh \mu} \lesssim \H \norm{\aH}{\mu} .
  \end{equation}
  If moreover, for all $z \in H^1_0(\dom) \cap H^2(\dom)$,
  \begin{equation} \label{eq:theta-hyp-2}
    \seminorm{H^1(\TH)}{\lproj{\TH}{1,\K+1} z - \Theta_{\ell-1} \lproj{\FH}{\K} z}
    + \H^{-1/2} \norm{\FH}{\jump{\lproj{\TH}{1,\K+1} z - \Theta_{\ell-1} \lproj{\FH}{\K} z}} \lesssim \H \seminorm{H^2(\dom)}{z},
  \end{equation}
  then $\IHh$ satisfies \eqi{eq:ass-prolongation-first}{1}.
\end{lemma}
\begin{proof}
  Write $z \doteq \Theta_{\ell-1}\mu$. By \eqref{eq:polynomial-bubble-1}, \eqref{eq:theta-hyp} and $\H \simeq \h$,
  \begin{equation} \label{eq:abstract-bubble}
    \norm{\uah}{\ul{\mathcal{W}}_\ell z - \ulisch \IHh \mu} \lesssim \norm{\aH}{\mu},
  \end{equation}
  while the triangle inequality with \eqref{eq:local-continuity}, \eqref{eq:Wh-stability} and \eqref{eq:polynomial-bubble-1} gives
  \begin{equation} \label{eq:Pav-stability}
    \norm{\ah}{\Pavh w} \lesssim \seminorm{H^1(\TH)}{w} + \h^{-1/2} \norm{\FH}{\jump{w}}
    \qquad \forall w \in \Pk{\K+1}(\TH) .
  \end{equation}
  Taking $w = \Theta_{\ell-1}\mu$ and using \eqref{eq:theta-hyp} with \eqi{eq:geometry-b}{3} gives \eqi{eq:ass-prolongation-stab}{1}. For \eqref{eq:abstract-l2}, $\IHh \mu = \Pavh z$ and $\lproj{\Th}{\k+1} z = z$, the latter since $z|_t \in \Pk{\k+1}(t)$, so \eqref{def:Wh} gives
  \[
    \ul{\mathcal{W}}_\ell z - \ulisch \IHh \mu = (z - \isch \IHh \mu, 0) ,
  \]
  and \eqref{eq:bubble-l2} with \eqref{eq:abstract-bubble} gives $\norm{L^2(\dom)}{z - \isch \IHh \mu} \lesssim \h \norm{\aH}{\mu}$.

  For the last claim set $z_{\ell-1} \doteq \lproj{\TH}{1,\K+1} z$ and split
  \[
    \norm{\uah}{\ul{\mathcal{J}}_\ell z - \ul{I}_\ell \ul{\mathcal{J}}_{\ell-1} z}
    \leq \norm{\uah}{\ul{\mathcal{J}}_\ell z - \ul{\mathcal{W}}_\ell z_{\ell-1}}
    + \norm{\uah}{\ul{\mathcal{W}}_\ell z_{\ell-1} - \ulisch \Pavh z_{\ell-1}}
    + \norm{\uah}{\ulisch \Pavh (z_{\ell-1} - \Theta_{\ell-1}\lproj{\FH}{\K} z)} .
  \]
  The first term is $\lesssim \H \seminorm{H^2(\dom)}{z}$ by \eqref{def:WhT-decomp}, \eqref{eq:local-continuity}, \eqref{eq:xiT-norm-bound} and \eqref{eq:interpolation-jumps}; the second by \eqref{eq:polynomial-bubble-2} with \eqi{eq:approximation-bulk}{1} and \eqref{eq:interpolation-jumps}; the third by \eqref{eq:Pav-stability} with $w = z_{\ell-1} - \Theta_{\ell-1}\lproj{\FH}{\K} z$, together with \eqref{eq:theta-hyp-2}.
\end{proof}

\subsection{Verification of the hypotheses}\label{sec:theta-verification}

It remains to check \eqref{eq:theta-hyp} and \eqref{eq:theta-hyp-2} for $\Theta_{\ell-1} = \iscH$ and $\Theta_{\ell-1} = \rec{\ell-1}{\K+1} \ulisc{\ell-1}$. \eqi{eq:theta-hyp}{1} is immediate from \eqref{eq:local-coercivity} for the former and from the nonnegativity of the stabilisation for the latter, summed over $T \in \TH$. The second follows from the next lemma applied to $\ulisc{\ell-1}\mu$, whose norm \eqref{def:hho-norm} is $\simeq \norm{\aH}{\mu}$ by \eqref{eq:local-continuity}: its first term covers $\Theta_{\ell-1} = \iscH$ and its second $\Theta_{\ell-1} = \rec{\ell-1}{\K+1} \ulisc{\ell-1}$.

\begin{lemma}[Convergence of the jumps]\label{lem:jump-convergence}
  For all $\uh = (u_{\Th}, u_{\Fh}) \in \Uh$,
  \begin{equation} \label{eq:jump-convergence}
    \Big( \sum_{F \in \Fh} \h^{-1} \norm{F}{\jump{u_{\Th}}}^2 \Big)^{1/2}
    + \Big( \sum_{F \in \Fh} \h^{-1} \norm{F}{\jump{\rec{\ell}{\k+1} \uh}}^2 \Big)^{1/2}
    \lesssim \norm{\ul{1,\ell}}{\uh} .
  \end{equation}
\end{lemma}
\begin{proof}
  For $w \in H^1(t)$, inserting $\lproj{\partial t}{\k} w$ and using that constants belong to $\IPk{\k}(F)$ by \eqi{eq:interface-space-design}{1}, the contractivity of $\lproj{F}{\k}$, \eqi{eq:polytopal-inequalities}{3} and Poincar\'e--Wirtinger,
  \begin{equation} \label{eq:jump-local}
    h_t^{-1} \norm{\partial t}{w - u_{\partial t}}^2 \lesssim \seminorm{1,t}{w}^2 + h_t^{-1} \norm{\partial t}{\lproj{\partial t}{\k} w - u_{\partial t}}^2 .
  \end{equation}
  Since $u_{\Fh}$ is single-valued and vanishes on $\Fhb$, \eqref{eq:jump} gives $\jump{w} = \jump{w - u_{\partial t}}$ on every face, and each face is met by at most two cells, so it suffices to bound \eqref{eq:jump-local} for $w = u_t$ and for $w = \rec{\ell,t}{\k+1} \ul{u}_t$. In the first case $\lproj{\partial t}{\k} u_t - u_{\partial t} = -r_t(\ul{u}_t)$ and $\norm{\partial t}{r_t(\ul{u}_t)} \leq \norm{\partial t}{u_{\partial t} - u_t}$, so \eqref{eq:jump-local} is bounded by $\norm{\ul{1,t}}{\ul{u}_t}^2$. In the second, $\lproj{\partial t}{\k} \rec{\ell,t}{\k+1} \ul{u}_t - u_{\partial t} = \lproj{\partial t}{\k}(\rec{\ell,t}{\k+1} \ul{u}_t - u_t) - r_t(\ul{u}_t)$, and $\rec{\ell,t}{\k+1} \ul{u}_t - u_t$ has zero mean by \eqref{eq:reconstruction-b}, so \eqi{eq:polytopal-inequalities}{3}, Poincar\'e--Wirtinger and \eqref{eq:local-continuity} again bound \eqref{eq:jump-local} by $\norm{\ul{1,t}}{\ul{u}_t}^2$.
\end{proof}

\begin{lemma}[Bulk approximation]\label{lem:bulk-approximation}
  Both choices of $\Theta_{\ell-1}$ satisfy \eqref{eq:theta-hyp-2}.
\end{lemma}
\begin{proof}
  Write $z_{\ell-1} \doteq \lproj{\TH}{1,\K+1} z$ and $\mu \doteq \lproj{\FH}{\K} z$. Set $\ul{\delta} \doteq \ul{\mathcal{J}}_{\ell-1} z - \ulisc{\ell-1} \mu$, a bubble with $\norm{\uaH}{\ul{\delta}} \lesssim \H \seminorm{H^2(\dom)}{z}$ by \eqref{eq:bubble-interp} at level $\ell-1$. For the two choices of $\Theta_{\ell-1}$, inserting $\lproj{\TH}{\K+1} z$ in the first and using \eqref{eq:elliptic-identity} in the second,
  \begin{align*}
    & \seminorm{H^1(\TH)}{z_{\ell-1} - \iscH \mu}
    \leq \seminorm{H^1(\TH)}{\delta_{\TH}} + \seminorm{H^1(\TH)}{z_{\ell-1} - \lproj{\TH}{\K+1} z}
    \just{\lesssim}{\eqref{eq:local-coercivity},\,\eqi{eq:approximation-bulk}{1}} \norm{\uaH}{\ul{\delta}} + \H \seminorm{H^2(\dom)}{z}, \\
    & \seminorm{H^1(\TH)}{z_{\ell-1} - \rec{\ell-1}{\K+1} \ulisc{\ell-1} \mu}
    = \seminorm{H^1(\TH)}{\rec{\ell-1}{\K+1} \ul{\delta}}
    \just{\leq}{\eqref{eq:local-bilinear-form}} \norm{\uaH}{\ul{\delta}},
  \end{align*}
  both $\lesssim \H \seminorm{H^2(\dom)}{z}$. For the jump term, the same two splittings and \eqref{eq:jump-convergence} at level $\ell-1$ give
  \[
    \H^{-1/2} \norm{\FH}{\jump{\delta_{\TH}}} + \H^{-1/2} \norm{\FH}{\jump{\rec{\ell-1}{\K+1} \ul{\delta}}}
    \just{\lesssim}{\eqref{eq:local-continuity}} \norm{\uaH}{\ul{\delta}},
    \qquad
    \H^{-1/2} \norm{\FH}{\jump{z_{\ell-1} - \lproj{\TH}{\K+1} z}}
    \just{\lesssim}{\eqref{eq:interpolation-jumps},\,\eqi{eq:approximation-bulk}{2}} \H \seminorm{H^2(\dom)}{z},
  \]
\end{proof}

\begin{theorem}[The two prolongations]\label{th:prolongation-verification}
  $\IUh$ and $\IRh$ satisfy \eqi{eq:ass-prolongation-stab}{1} and \eqref{eq:ass-prolongation-first}.
\end{theorem}
\begin{proof}
  Apply Lemma~\ref{lem:abstract} with $\Theta_{\ell-1} = \iscH$ and $\Theta_{\ell-1} = \rec{\ell-1}{\K+1}\ulisc{\ell-1}$, whose hypotheses hold by the paragraph above, Lemma~\ref{lem:jump-convergence} and Lemma~\ref{lem:bulk-approximation}. This gives \eqi{eq:ass-prolongation-stab}{1}, \eqi{eq:ass-prolongation-first}{1} and \eqref{eq:abstract-l2}.

  For $\Theta_{\ell-1} = \iscH$, \eqref{eq:abstract-l2} is \eqi{eq:ass-prolongation-first}{2}. For $\Theta_{\ell-1} = \rec{\ell-1}{\K+1}\ulisc{\ell-1}$ the two differ, and \eqref{eq:reconstruction-b} gives them the same mean on every coarse cell, so
  \[
    \norm{L^2(\dom)}{\Theta_{\ell-1}\mu - \iscH\mu}
    \leq \norm{L^2(\dom)}{(\identity - \lproj{\TH}{0})\Theta_{\ell-1}\mu} + \norm{L^2(\dom)}{(\identity - \lproj{\TH}{0})\iscH\mu}
    \just{\lesssim}{\eqi{eq:approximation-bulk}{1}} \H \norm{\aH}{\mu},
  \]
  with $l = 0$, $s = 1$, and then \eqi{eq:theta-hyp}{1} for both maps. The triangle inequality then turns \eqref{eq:abstract-l2} into \eqi{eq:ass-prolongation-first}{2}.
\end{proof}

\subsection{Stability in the skeletal norms}\label{sec:levelnorm}

\begin{lemma}[Level-norm stability]\label{lem:levelnorm-stability}
  $\IUh$ and $\IRh$ satisfy \eqi{eq:ass-prolongation-stab}{2}.
\end{lemma}
\begin{proof}
  By \eqi{eq:face-bounds}{2} at both levels it suffices to prove $\sum_{t \in \Th} h_t \norm{\partial t}{\IHh \mu}^2 \lesssim \sum_{T \in \TH} h_T \norm{\partial T}{\mu}^2$. Write $v_T \doteq \Theta_{\ell-1}\mu|_T$, let $t \in \Th(T)$ and let $F \in \Fh(t)$. If $F \in \Fh \setminus \FH$, then \eqref{def:Pi-av} and the $L^2(F)$-contractivity of $\lproj{F}{\k}$ give $\norm{F}{\IHh \mu} \leq \norm{F}{v_T}$; if $F \subset \partial T \cap \partial T'$, the convexity of $x \mapsto x^2$ and $\alpha_{FT} + \alpha_{FT'} = 1$ give $\norm{F}{\IHh \mu}^2 \leq \norm{F}{v_T}^2 + \norm{F}{v_{T'}}^2$. Summing over $F \in \Fh(t)$,
  \[
    h_t \norm{\partial t}{\IHh \mu}^2 \leq h_t \norm{\partial t}{v_T}^2 + \sum_{F \subset \partial t \cap \partial T'} h_t \norm{F}{v_{T'}}^2 .
  \]
  Each face of the second sum is also a face of the fine cell $t' \subset T'$ on its other side, and $h_t \simeq h_{t'}$, so on summing over $t \in \Th$ every neighbour term is absorbed into the first term of the cell carrying it. No bound on the number of fine faces per coarse cell is used: a fine face has at most one cell on each side, and the faces of $t'$ partition $\partial t'$, so the collected terms reassemble into $\norm{\partial t'}{v_{T'}}^2$. The fine-cell trace inequality \eqi{eq:polytopal-inequalities}{1} turns $h_t \norm{\partial t}{v_T}^2$ into $\norm{t}{v_T}^2$, and the fine cells of $T$ are disjoint, so summing gives $\norm{T}{v_T}^2$. Finally $\norm{T}{v_T} \lesssim h_T^{1/2} \norm{\partial T}{\mu}$: for $\Theta_{\ell-1} = \iscH$ this is \eqref{eq:local-energy-bound}, and for $\Theta_{\ell-1} = \rec{\ell-1}{\K+1}\ulisc{\ell-1}$ it follows from \eqref{eq:local-energy-bound}, \eqref{eq:reconstruction-b} and Poincar\'e--Wirtinger.
\end{proof}

\subsection{Second-order approximation of smooth functions}\label{sec:second-order}

The bounds of \eqref{eq:ass-prolongation-first} are first order. For the skeletal projection of a smooth function they improve by one order, which is what \eqref{eq:ass-prolongation-approx} requires. Throughout, $z \in H^1_0(\dom) \cap H^2(\dom)$ and $\theta \doteq \Theta_{\ell-1} \lproj{\FH}{\K} z \in \Pk{\K+1}(\TH)$.

\begin{lemma}[Coarse reconstruction of smooth functions]\label{lem:smooth-reconstruction}
  For both choices of $\Theta_{\ell-1}$,
  \begin{subequations} \label{eq:smooth-proximity}
  \begin{align}
    \norm{L^2(\dom)}{\theta - z} &\lesssim \H^2 \seminorm{H^2(\dom)}{z}, \qquad
    \seminorm{H^1(\TH)}{\theta - z} \lesssim \H \seminorm{H^2(\dom)}{z}, \label{eq:smooth-proximity-l2} \\
    \seminorm{H^2(\TH)}{\theta} &\lesssim \seminorm{H^2(\dom)}{z}, \qquad
    \norm{\FH}{\jump{\theta}}^2 \lesssim \H^3 \seminorm{H^2(\dom)}{z}^2 . \label{eq:smooth-proximity-jumps}
  \end{align}
  \end{subequations}
\end{lemma}
\begin{proof}
  Write $z_{\ell-1} \doteq \lproj{\TH}{1,\K+1} z$ and $\vartheta \doteq z_{\ell-1} - \theta$, so that Lemma~\ref{lem:bulk-approximation} gives
  \vspace{-0.5em}
  \begin{equation} \label{eq:vartheta-h1}
    \seminorm{H^1(\TH)}{\vartheta} + \H^{-1/2} \norm{\FH}{\jump{\vartheta}} \lesssim \H \seminorm{H^2(\dom)}{z} .
  \end{equation}
  Set $\ul{\delta} = (\delta_{\TH}, 0) \doteq \ul{\mathcal{J}}_{\ell-1} z - \ulisc{\ell-1} \lproj{\FH}{\K} z$, a bubble with bulk component $\delta_{\TH} = \lproj{\TH}{\K+1} z - \iscH \lproj{\FH}{\K} z$, so that $\norm{\TH}{\delta_{\TH}} \just{\lesssim}{\eqref{eq:bubble-l2}} \H \norm{\uaH}{\ul{\delta}} \just{\lesssim}{\eqref{eq:bubble-interp}} \H^2 \seminorm{H^2(\dom)}{z}$. For $\Theta_{\ell-1} = \iscH$ one has $\vartheta = (z_{\ell-1} - \lproj{\TH}{\K+1} z) + \delta_{\TH}$, and inserting $\pm z$ and using \eqi{eq:approximation-bulk}{1} twice gives $\norm{L^2(\dom)}{\vartheta} \lesssim \H^2 \seminorm{H^2(\dom)}{z}$. For $\Theta_{\ell-1} = \rec{\ell-1}{\K+1}\ulisc{\ell-1}$, \eqref{eq:elliptic-identity} gives $\vartheta = \rec{\ell-1}{\K+1} \ul{\delta}$, so $\lproj{\TH}{0} \vartheta = \lproj{\TH}{0} \delta_{\TH}$ by \eqref{eq:reconstruction-b} and
  \vspace{-0.5em}
  \[
    \norm{L^2(\dom)}{\vartheta} \leq \norm{L^2(\dom)}{(\identity - \lproj{\TH}{0})\vartheta} + \norm{\TH}{\lproj{\TH}{0}\delta_{\TH}}
    \lesssim \H \seminorm{H^1(\TH)}{\vartheta} + \norm{\TH}{\delta_{\TH}}
    \just{\lesssim}{\eqref{eq:vartheta-h1}} \H^2 \seminorm{H^2(\dom)}{z} .
  \]
  Then \eqi{eq:smooth-proximity-l2}{1} follows from $\theta - z = (z_{\ell-1} - z) - \vartheta$ with \eqi{eq:approximation-bulk}{1}, and \eqi{eq:smooth-proximity-l2}{2} from the same splitting with \eqref{eq:vartheta-h1}.
  For \eqref{eq:smooth-proximity-jumps}, $\seminorm{H^2(\TH)}{z_{\ell-1}} \lesssim \seminorm{H^2(\dom)}{z}$ on inserting $\pm z$ and using \eqi{eq:approximation-bulk}{1} with $m = s = 2$, and $\seminorm{H^2(\TH)}{\vartheta} \lesssim \H^{-1} \seminorm{H^1(\TH)}{\vartheta} \lesssim \seminorm{H^2(\dom)}{z}$ by \eqi{eq:polytopal-inequalities}{4}, $\vartheta$ being a polynomial. For the jumps,
  \[
    \norm{\FH}{\jump{\theta}}
    \leq \norm{\FH}{\jump{z_{\ell-1}}} + \norm{\FH}{\jump{\vartheta}}
    \just{\lesssim}{\eqref{eq:interpolation-jumps},\,\eqref{eq:vartheta-h1}} \H^{3/2} \seminorm{H^2(\dom)}{z} .
  \]
\end{proof}

\begin{lemma}[Skeletal approximation]\label{lem:approximation-verification}
  $\IUh$ and $\IRh$ satisfy \eqref{eq:ass-prolongation-approx}.
\end{lemma}
\begin{proof}
  Write $\varepsilon \doteq \lproj{\Fh}{\k} z - \IHh \lproj{\FH}{\K} z$. By \eqref{eq:polynomial-bubble-2} applied to $\theta$, with \eqi{eq:smooth-proximity-jumps}{1} and \eqi{eq:smooth-proximity-jumps}{2},
  \[
    \norm{\uah}{\ul{\mathcal{W}}_\ell \theta - \ulisch \IHh \lproj{\FH}{\K} z} \lesssim \h \seminorm{H^2(\dom)}{z},
  \]
  Since $\IHh \lproj{\FH}{\K} z = \Pavh \theta$ by \eqref{eq:prolongations}, the two pairs share their face component, so their difference is a bubble by \eqref{def:Wh}; its cell component is $\theta - \isch \IHh \lproj{\FH}{\K} z$, $\lproj{\Th}{\k+1}\theta$ being $\theta$ on each fine cell. Hence \eqref{eq:bubble-l2} gives
  \begin{equation} \label{eq:lem-skel-approx-intermediary}
    \norm{L^2(\dom)}{\theta - \isch \IHh \lproj{\FH}{\K} z} \lesssim \h^2 \seminorm{H^2(\dom)}{z} .
  \end{equation}
  By \eqi{eq:l2-representative}{2} it suffices to bound $\norm{\Th}{\isch \varepsilon}$ at order $\H^2$ and $\norm{\ah}{\varepsilon}$ at order $\H$; the latter is \eqi{eq:ass-prolongation-first}{1}, established for both prolongations in Theorem~\ref{th:prolongation-verification}, with energy minimality of static condensation. For the former, inserting $\pm z$ and $\pm\theta$,
  \[
    \norm{\Th}{\isch \varepsilon}
    \leq \norm{L^2(\dom)}{\isch \lproj{\Fh}{\k} z - z} + \norm{L^2(\dom)}{z - \theta} + \norm{L^2(\dom)}{\theta - \isch \IHh \lproj{\FH}{\K} z},
  \]
  the last two terms bounded by $\H^2 \seminorm{H^2(\dom)}{z}$ through \eqi{eq:smooth-proximity-l2}{1} and \eqref{eq:lem-skel-approx-intermediary}. For the first, $\ul{\mathcal{J}}_\ell z - \ulisch \lproj{\Fh}{\k} z$ is a bubble with cell component $\lproj{\Th}{\k+1} z - \isch \lproj{\Fh}{\k} z$, so \eqref{eq:bubbles} bounds its $L^2(\dom)$ norm by $\h^2 \seminorm{H^2(\dom)}{z}$, and \eqi{eq:approximation-bulk}{1} removes $\lproj{\Th}{\k+1} z$.
\end{proof}

For arbitrary coarse data the second-order bound is false: the jump term in \eqref{eq:polynomial-bubble-2} is then only $O(\H^{1/2})$ times the energy, which is the order at which \eqi{eq:ass-prolongation-first}{2} saturates.

\section{Analysis of the smoother}\label{sec:smoother}

This section proves Assumption~\ref{ass:smoother} for the patch families of \sect{sec:smoother-def} and deduces Assumption~\ref{ass:smoothing}. Both rest on one observation. If $\mu \in \Uhb$ vanishes on $\partial t$, the local problem \eqref{eq:inverse-static-condensation-a} for $\isc{\ell,t}\mu$ has zero right-hand side and is positive definite on $\Pk{\k+1}(t)$ by \eqref{eq:local-coercivity} and \eqi{eq:interface-space-design}{1}, so $\isc{\ell,t}\mu = 0$ and the cell contributes no energy. Applying \eqref{eq:local-energy-bound} cell by cell therefore gives
\begin{equation} \label{eq:zero-extension-bound}
  \norm{\ah}{\mu}^2 \leq C_E \sum_{t \in \Th} h_t^{-1} \norm{\partial t}{\mu}^2 \qquad \forall \mu \in \Uhb,
\end{equation}
with $C_E$ the constant of \eqref{eq:local-energy-bound}, and only the cells on which $\mu$ does not vanish contribute. Moreover,
\begin{equation} \label{eq:face-splitting}
  \norm{\partial t}{\mu}^2 = \sum_{F \in \Fh(t) \cap \Fhi} \norm{F}{\mu}^2 .
\end{equation}

\subsection{Stable decomposition}\label{sec:stable-decomposition}

\begin{lemma}[Stable decomposition]\label{lem:stable-decomposition}
  Let Assumption~\ref{ass:patches} hold and let $\alpha_{r,F} \geq 0$ satisfy $\sum_{r : F \in \mathcal{F}_r} \alpha_{r,F} = 1$ for every $F \in \Fhi$. For $\lambda \in \Uhb$ let $\lambda_r \in N_r$ equal $\alpha_{r,F} \lambda|_F$ on $F \in \mathcal{F}_r$ and vanish elsewhere. Then $\lambda = \sum_{r \in \mathcal{D}} \lambda_r$ and
  \begin{equation} \label{eq:stable-decomposition}
    \sum_{r \in \mathcal{D}} \norm{\ah}{\lambda_r}^2 \lesssim \h^{-2} \norm{\ell}{\lambda}^2 .
  \end{equation}
\end{lemma}
\begin{proof}
  That $\lambda = \sum_r \lambda_r$ is the partition of unity, every interior face being covered by Assumption~\ref{ass:patches}. Each $\lambda_r$ vanishes outside the cells meeting $\mathcal{F}_r$, so \eqref{eq:zero-extension-bound} and $\alpha_{r,F} \leq 1$ give
  \[
    \norm{\ah}{\lambda_r}^2 \leq C_E \sum_{t} h_t^{-1} \sum_{F \in \Fh(t) \cap \mathcal{F}_r} \alpha_{r,F}^2 \norm{F}{\lambda}^2 ,
  \]
  the outer sum running over the cells meeting $\mathcal{F}_r$. Summing over $r$ and exchanging the finite sums, $\sum_{r : F \in \mathcal{F}_r} \alpha_{r,F}^2 \leq \sum_{r : F \in \mathcal{F}_r} \alpha_{r,F} = 1$, so
  \[
    \sum_{r \in \mathcal{D}} \norm{\ah}{\lambda_r}^2
    \leq C_E \sum_{t \in \Th} h_t^{-1} \sum_{F \in \Fh(t) \cap \Fhi} \norm{F}{\lambda}^2
    \just{=}{\eqref{eq:face-splitting}} C_E \sum_{t \in \Th} h_t^{-1} \norm{\partial t}{\lambda}^2
    \just{\lesssim}{\eqi{eq:geometry-b}{2}} \h^{-2} \sum_{t \in \Th} h_t \norm{\partial t}{\lambda}^2
    \just{\lesssim}{\eqi{eq:face-bounds}{2}} \h^{-2} \norm{\ell}{\lambda}^2 .
  \]
  The exchange is an identity, a cell meets $\mathcal{F}_r$ exactly when one of its faces lies in $\mathcal{F}_r$, so no bound on the number of faces per cell is needed.
\end{proof}

For the face-star family each interior face generates its own patch, $\alpha_{r,F} = 1$, and \eqref{eq:stable-decomposition} holds with no further hypothesis.

It remains to bound $\rho(\Dh)$. Since the local solvers are exact, $\Dh = \sum_{r \in \mathcal{D}} P_r$ with $P_r \doteq I_r A_r^{-1} I_r' \Ah$ the $\ah$-orthogonal projection onto $N_r$: for $\eta \in N_r$,
$\ah(P_r \mu, \eta) = \braket{I_r' \Ah \mu, \eta}_\ell = \braket{\Ah \mu, I_r \eta}_\ell = \ah(\mu, \eta)$. Taking $\eta = P_r \mu$ in the previous identity,
\begin{equation} \label{eq:patch-projection}
  \norm{\ah}{P_r \mu}^2 = \ah(P_r \mu, \mu) \qquad \forall \mu \in \Uhb, \ r \in \mathcal{D} .
\end{equation}

\begin{lemma}[Interaction bound]\label{lem:interaction}
  Let \Aref{ass:patches} hold. Then $\rho(\Dh) \leq \NO$.
\end{lemma}
\begin{proof}
  Group the patches by colour, $\Dh \mu = \sum_{c=1}^{\NO} w_c$ with $w_c \doteq \sum_{r \in \mathcal{D}_c} P_r \mu$ and $\mathcal{D}_c$ the patches of colour $c$. Within a colour the supports are disjoint, so the corresponding local forms do not interact and
  \[
    \norm{\ah}{w_c}^2 = \sum_{r \in \mathcal{D}_c} \norm{\ah}{P_r \mu}^2 .
  \]
  By the triangle inequality and Cauchy--Schwarz in $\mathbb{R}^{\NO}$,
  \[
    \norm{\ah}{\Dh \mu}^2 \leq \Big( \sum_{c=1}^{\NO} \norm{\ah}{w_c} \Big)^2
    \leq \NO \sum_{c=1}^{\NO} \norm{\ah}{w_c}^2
    = \NO \sum_{r \in \mathcal{D}} \norm{\ah}{P_r \mu}^2 .
  \]
  By \eqref{eq:patch-projection} the last sum is $\ah(\Dh \mu, \mu) \leq \norm{\ah}{\Dh \mu} \norm{\ah}{\mu}$. Dividing gives $\norm{\ah}{\Dh \mu} \leq \NO \norm{\ah}{\mu}$, and $\Dh$ is self-adjoint for $\ah$.
\end{proof}

Together with Lemma~\ref{lem:stable-decomposition} this establishes \Aref{ass:smoother} with $c_0 = \NO$.

\subsection{The smoothing hypothesis}\label{sec:smoothing-proof}

\begin{theorem}[Smoothing]\label{th:smoother-cond-number}
  Let Assumption~\ref{ass:smoother} hold and let $\omega < 2/c_0$. Then Assumption~\ref{ass:smoothing} holds with $C_S = c_1 / \big(\omega (2 - \omega c_0)\big)$.
\end{theorem}
\begin{proof}
  By \eqref{eq:patch-projection}, $\ah(\Dh\mu,\mu) = \sum_{r \in \mathcal{D}} \norm{\ah}{P_r \mu}^2 \geq 0$, and it vanishes only if $\mu$ is $\ah$-orthogonal to every $N_r$, hence to $\sum_r N_r = \Uhb$ by the decomposition clause of Assumption~\ref{ass:smoother}, so $\mu = 0$. With $\rho(\Dh) \leq c_0$ the spectrum of $\Sh\Ah = \omega\Dh$ lies in $(0, \omega c_0] \subset (0,2)$, so $\Sh$ is invertible and $\norm{\ah}{(\identity - \Sh \Ah)\mu} \leq \norm{\ah}{\mu}$ for all $\mu \in \Uhb$, which is the role of $\omega < 2/c_0$.

  The subspace correction identity \cite[Lemma 2.4]{xu2002} and Assumption~\ref{ass:smoother} give
  \[
    \braket{\Sh^{-1} \lambda, \lambda}_\ell = \omega^{-1} \ah(\Dh^{-1}\lambda, \lambda)
    = \omega^{-1} \inf_{\substack{\lambda_r \in N_r \\ \sum_r \lambda_r = \lambda}} \sum_{r \in \mathcal{D}} \norm{A_r}{\lambda_r}^2
    \leq \omega^{-1} c_1 \h^{-2} \norm{\ell}{\lambda}^2 .
  \]
  Since $\Sh$ is self-adjoint and positive definite for $\braket{\cdot,\cdot}_\ell$, an upper bound on the largest eigenvalue of $\Sh^{-1}$ is a lower bound on the smallest of $\Sh$, thus $\braket{\Sh \lambda, \lambda}_\ell \geq \omega c_1^{-1} \h^{2} \norm{\ell}{\lambda}^2$.

  It remains to transfer the previous bounds to $\Kh$. As $\Sh' = \Sh$, \eqref{eq:symmetrised-smoother} reads $\Kh = 2\Sh - \Sh \Ah \Sh$, so $\braket{\Kh \lambda,\lambda}_\ell = 2\braket{\Sh\lambda,\lambda}_\ell - \braket{\Ah \Sh \lambda, \Sh \lambda}_\ell$. With $\Sh^{1/2}$ the positive square root, $\Sh^{1/2}\Ah\Sh^{1/2}$ is self-adjoint and similar to $\Sh\Ah = \omega \Dh$, hence has spectrum in $(0, \omega c_0]$, so $\braket{\Ah \Sh \lambda, \Sh \lambda}_\ell \leq \omega c_0 \braket{\Sh\lambda,\lambda}_\ell$. Since $\omega c_0 < 2$,
  \[
    \braket{\Kh \lambda, \lambda}_\ell \geq (2 - \omega c_0) \braket{\Sh \lambda,\lambda}_\ell
    \geq \frac{\omega (2 - \omega c_0)}{c_1} \, \h^{2} \norm{\ell}{\lambda}^2 . \qedhere
  \]
\end{proof}

With $c_0 = \NO$, the damping condition reads $\omega < 2/\NO$.

\section{Convergence}\label{sec:main-result}

It remains to derive Assumptions~\ref{ass:bounded} and \ref{ass:regularity-approx} from \Aref{ass:prolongation}. Throughout, $\Ph$ is the coarse projection \eqref{eq:coarse-projection} and $C_I$ the constant of \eqi{eq:ass-prolongation-stab}{1}.

\subsection{Bounded coarse correction}\label{sec:bounded-correction}

\begin{lemma}[Bounded coarse correction]\label{lem:bounded-correction}
  Let \eqi{eq:ass-prolongation-stab}{1} hold. Then, for all $\lambda \in \Uhb$,
  \begin{subequations} \label{eq:bounded-correction}
  \begin{align}
    \norm{\aH}{\Ph \lambda} &\leq C_I \norm{\ah}{\lambda}, \label{eq:P-stability} \\
    \norm{\ah}{\lambda - \IHh \Ph \lambda} &\leq (1 + C_I^4)^{1/2} \norm{\ah}{\lambda}, \label{eq:correction-stability}
  \end{align}
  \end{subequations}
  and \Aref{ass:bounded} holds with $C_B = (1 + C_I^4)^{1/2}$.
\end{lemma}
\begin{proof}
  Testing \eqref{eq:coarse-projection} with $\mu = \Ph \lambda$,
  \[
    \norm{\aH}{\Ph \lambda}^2 = \ah(\lambda, \IHh \Ph \lambda)
    \leq \norm{\ah}{\lambda} \norm{\ah}{\IHh \Ph \lambda}
    \just{\leq}{\eqi{eq:ass-prolongation-stab}{1}} C_I \norm{\ah}{\lambda} \norm{\aH}{\Ph \lambda},
  \]
  which is \eqref{eq:P-stability}. Expanding the square and using \eqref{eq:coarse-projection} once more,
  \[
    \norm{\ah}{\lambda - \IHh \Ph \lambda}^2
    = \norm{\ah}{\lambda}^2 - 2 \norm{\aH}{\Ph \lambda}^2 + \norm{\ah}{\IHh \Ph \lambda}^2
    \just{\leq}{\eqi{eq:ass-prolongation-stab}{1},\,\eqref{eq:P-stability}} (1 + C_I^4) \norm{\ah}{\lambda}^2 . \qedhere
  \]
\end{proof}

\subsection{A first-order bound in the skeletal norm}\label{sec:levelnorm-first-order}

\Aref{ass:regularity-approx} needs two ingredients: a first-order bound on the coarse-grid correction, valid for every $\lambda \in \Uhb$, and second-order two-level estimates, available for discrete solutions only. Here is the first.

\begin{lemma}[First-order skeletal bound]\label{lem:levelnorm-OH}
  Let \Aref{ass:regularity}, \eqi{eq:ass-prolongation-stab}{1} and \eqref{eq:ass-prolongation-first} hold. Then, for all $\lambda \in \Uhb$,
  \begin{equation} \label{eq:levelnorm-OH}
    \norm{\ell}{\lambda - \IHh \Ph \lambda} \lesssim \H \norm{\ah}{\lambda} .
  \end{equation}
\end{lemma}
\begin{proof}
  Write $\varepsilon \doteq \lambda - \IHh \Ph \lambda$.

  \emph{Step 1 (the bulk representative).} Assume $\isch \varepsilon \neq 0$ and let $z$ solve \eqref{eq:weak-poisson} with datum $\isch \varepsilon$. Inserting $\iscH \Ph \lambda$,
  \[
    \norm{L^2(\dom)}{\isch \varepsilon}^2
    = (-\Delta z, \isch \lambda - \iscH \Ph \lambda)_{\dom}
    + (-\Delta z, \iscH \Ph \lambda - \isch \IHh \Ph \lambda)_{\dom}
    \doteq \mathfrak{T}_1 + \mathfrak{T}_2 ,
  \]
  the two bulk representatives living on different meshes but both in $L^2(\dom)$. Bubbles are $\uah$-orthogonal to discrete-harmonic pairs by \eqref{eq:inverse-static-condensation-a}, so $\uah(\ul{\mathcal{J}}_\ell z, \ulisch \mu) = \ah(\lproj{\Fh}{\k} z, \mu)$ for all $\mu \in \Uhb$, and likewise on level $\ell-1$. Hence, by \eqref{eq:interpolator-consistency} at both levels and \eqref{eq:coarse-projection},
  \[
    \mathfrak{T}_1 = \Eh(z; \ulisch \lambda) - \EH(z; \ulisc{\ell-1} \Ph \lambda)
    + \ah\big( \lproj{\Fh}{\k} z - \IHh \lproj{\FH}{\k} z, \lambda \big) .
  \]
  The first two terms are bounded by $\H \seminorm{H^2(\dom)}{z} \norm{\ah}{\lambda}$ using \eqref{eq:consistency-bound} at the two levels, \eqref{eq:P-stability} and \eqi{eq:geometry-b}{3}. For the third, the skeletal component of $\ul{\mathcal{J}}_\ell z - \ul{I}_\ell \ul{\mathcal{J}}_{\ell-1} z$ is $\lproj{\Fh}{\k} z - \IHh \lproj{\FH}{\k} z$ by \eqref{eq:prolongations}, so energy minimality of static condensation gives
  \[
    \ah\big( \lproj{\Fh}{\k} z - \IHh \lproj{\FH}{\k} z, \lambda \big)
    \leq \norm{\uah}{\ul{\mathcal{J}}_\ell z - \ul{I}_\ell \ul{\mathcal{J}}_{\ell-1} z} \norm{\ah}{\lambda}
    \just{\lesssim}{\eqi{eq:ass-prolongation-first}{1}} \H \seminorm{H^2(\dom)}{z} \norm{\ah}{\lambda} .
  \]
  The same bound holds for $\mathfrak{T}_2$ by \eqi{eq:ass-prolongation-first}{2} and \eqref{eq:P-stability}. Collecting, using $\seminorm{H^2(\dom)}{z} \lesssim \norm{L^2(\dom)}{\isch \varepsilon}$ from \Aref{ass:regularity} and dividing,
  \begin{equation} \label{eq:duality-OH}
    \norm{L^2(\dom)}{\isch \varepsilon} \lesssim \H \norm{\ah}{\lambda} .
  \end{equation}

  \emph{Step 2 (conclusion).} The bound \eqref{eq:duality-OH} is trivially true when $\isch \varepsilon = 0$, so
  \[
    \norm{\ell}{\varepsilon} \just{\leq}{\eqi{eq:l2-representative}{2}} \norm{\Th}{\isch \varepsilon} + \h \norm{\ah}{\varepsilon}
    \just{\lesssim}{\eqref{eq:duality-OH},\,\eqref{eq:correction-stability}} \H \norm{\ah}{\lambda} + \h \norm{\ah}{\lambda}
    \just{\lesssim}{\eqi{eq:geometry-b}{3}} \H \norm{\ah}{\lambda} . \qedhere
  \]
\end{proof}

\subsection{Two-level estimates for discrete solutions}\label{sec:two-level}

\begin{lemma}[Two-level estimates]\label{lem:two-level}
  Let \Aref{ass:regularity}, \eqi{eq:ass-prolongation-stab}{2}, \eqi{eq:ass-prolongation-first}{2} and \eqref{eq:ass-prolongation-approx} hold. Then, for all $g \in L^2(\dom)$,
  \begin{subequations} \label{eq:two-level}
  \begin{align}
    \norm{\ell}{\Gh g - \IHh \GH g} &\lesssim \H^2 \norm{L^2(\dom)}{g}, \label{eq:two-level-l2} \\
    \norm{\ell-1}{\GH g - \Ph \Gh g} &\lesssim \H^2 \norm{L^2(\dom)}{g} . \label{eq:two-level-projection}
  \end{align}
  \end{subequations}
\end{lemma}
\begin{proof}
  Let $u_g$ solve \eqref{eq:weak-poisson} with datum $g$.

  \emph{Step 1: \eqref{eq:two-level-l2}.} We first show that, at either level,
  \begin{equation} \label{eq:skeletal-error}
    \norm{\ell}{\Gh g - \lproj{\Fh}{\k} u_g} \lesssim \h^2 \norm{L^2(\dom)}{g} .
  \end{equation}
  By \eqi{eq:l2-representative}{2} it suffices to bound the bulk part at order $\h^2$ and the energy at order $\h$; the latter is \eqref{eq:error}. For the former, inserting $\pm u_g$ and $\pm\lproj{\Th}{\k+1} u_g$,
  \[
    \norm{\Th}{\isch \Gh g - \isch \lproj{\Fh}{\k} u_g}
    \leq \norm{\Th}{\isch \Gh g - u_g} + \norm{\Th}{u_g - \lproj{\Th}{\k+1} u_g}
    + \norm{\Th}{\lproj{\Th}{\k+1} u_g - \isch \lproj{\Fh}{\k} u_g},
  \]
  bounded by \eqref{eq:error}, by \eqi{eq:approximation-bulk}{1}, and by \eqref{eq:bubbles}, the last two with \Aref{ass:regularity} and the last term being the bulk component of the bubble $\ul{\mathcal{J}}_\ell u_g - \ulisch \lproj{\Fh}{\k} u_g$. This gives \eqref{eq:skeletal-error}.

  Inserting $\pm\lproj{\Fh}{\k} u_g$ and $\pm\IHh \lproj{\FH}{\K} u_g$,
  \[
    \norm{\ell}{\Gh g - \IHh \GH g}
    \leq \norm{\ell}{\Gh g - \lproj{\Fh}{\k} u_g}
    + \norm{\ell}{\lproj{\Fh}{\k} u_g - \IHh \lproj{\FH}{\K} u_g}
    + \norm{\ell}{\IHh(\lproj{\FH}{\K} u_g - \GH g)},
  \]
  and the three terms are bounded by \eqref{eq:skeletal-error}, by \eqref{eq:ass-prolongation-approx} with \Aref{ass:regularity}, and by \eqi{eq:ass-prolongation-stab}{2} with \eqref{eq:skeletal-error} at level $\ell-1$. Conclude with \eqi{eq:geometry-b}{3}.

  \emph{Step 2: \eqref{eq:two-level-projection}.} Let $e \doteq \GH g - \Ph \Gh g \in \UHb$. Testing the coarse problem with $\mu$, then \eqref{eq:coarse-projection} and the fine problem with $\IHh \mu$,
  \begin{equation} \label{eq:e-identity}
    \aH(e, \mu) = (g, \iscH \mu)_{\TH} - (g, \isch \IHh \mu)_{\Th} = (g, \iscH \mu - \isch \IHh \mu)_{\dom}
    \qquad \forall \mu \in \UHb .
  \end{equation}
  Taking $\mu = e$,
  \[
    \norm{\aH}{e}^2 \leq \norm{L^2(\dom)}{g} \norm{L^2(\dom)}{\iscH e - \isch \IHh e}
    \just{\lesssim}{\eqi{eq:ass-prolongation-first}{2}} \H \norm{L^2(\dom)}{g} \norm{\aH}{e}.
  \]
  Dividing by $\norm{\aH}{e}$ gives $\norm{\aH}{e} \lesssim \H \norm{L^2(\dom)}{g}$. 
  The face term of $\norm{\ell-1}{e}$ is bounded by the energy as follows:
  \[
    \Big( \sum_{T \in \TH} h_T \norm{\partial T}{r_T(\ulisc{\ell-1} e)}^2 \Big)^{1/2}
    \just{\leq}{\eqi{eq:face-bounds}{1}} \H \norm{\aH}{e} \lesssim \H^2 \norm{L^2(\dom)}{g} .
  \]
  The cell term instead needs a duality argument, together with a second-order sharpening of \eqi{eq:ass-prolongation-first}{2} available when the coarse datum is a discrete solution: for $w \in L^2(\dom)$, inserting $\pm u_w$ and $\pm\isch \Gh w$,
  \begin{equation} \label{eq:two-level-bulk}
  \begin{aligned}
    \norm{L^2(\dom)}{\iscH \GH w - \isch \IHh \GH w}
    &\just{\leq}{\eqi{eq:l2-representative}{1}} \norm{L^2(\dom)}{\iscH \GH w - u_w} + \norm{L^2(\dom)}{u_w - \isch \Gh w} + \norm{\ell}{\Gh w - \IHh \GH w} \\
    &\just{\lesssim}{\eqref{eq:error},\,\eqref{eq:two-level-l2}} \H^2 \norm{L^2(\dom)}{w} ,
  \end{aligned}
  \end{equation}
  Assume $\iscH e \neq 0$ and take $\mu = \GH(\iscH e)$, so that $\aH(e,\mu) = \norm{L^2(\dom)}{\iscH e}^2$ by \eqref{eq:discrete-poisson-skeleton} at level $\ell-1$ tested with $e$, and the symmetry of $\aH$. Then \eqref{eq:e-identity} and \eqref{eq:two-level-bulk} with datum $\iscH e$ give
  \[
    \norm{L^2(\dom)}{\iscH e}^2
    \leq \norm{L^2(\dom)}{g} \, \norm{L^2(\dom)}{\iscH \GH (\iscH e) - \isch \IHh \GH (\iscH e)}
    \lesssim \H^2 \norm{L^2(\dom)}{g} \norm{L^2(\dom)}{\iscH e} .
  \]
  Dividing and adding the two terms gives \eqref{eq:two-level-projection}.
\end{proof}

\subsection{The approximation property}\label{sec:approximation-property}

\begin{theorem}[Approximation property]\label{th:approximation}
  Let Assumptions~\ref{ass:regularity} and \ref{ass:prolongation} hold. Then, for all $\lambda \in \Uhb$,
  \begin{equation} \label{eq:master-estimate}
    \norm{\ell}{\lambda - \IHh \Ph \lambda} \lesssim \H^2 \norm{\ell}{\Ah \lambda},
  \end{equation}
  and \Aref{ass:regularity-approx} holds.
\end{theorem}
\begin{proof}
  Let $f_\lambda \doteq \isch(\Ah \lambda)$ be the residual representative of Corollary~\ref{cor:residual-representative}. Since $\IHh$ and $\Ph$ are linear,
  \[
    \lambda - \IHh \Ph \lambda
    = \underbrace{(\lambda - \Gh f_\lambda) - \IHh \Ph (\lambda - \Gh f_\lambda)}_{\mathfrak{T}_1}
    + \underbrace{\Gh f_\lambda - \IHh \GH f_\lambda}_{\mathfrak{T}_2}
    + \underbrace{\IHh (\GH f_\lambda - \Ph \Gh f_\lambda)}_{\mathfrak{T}_3} .
  \]
  Lemma~\ref{lem:levelnorm-OH} applied to $\lambda - \Gh f_\lambda$ gives
  \[
    \norm{\ell}{\mathfrak{T}_1} \lesssim \H \norm{\ah}{\lambda - \Gh f_\lambda}
    \just{\leq}{\eqi{eq:residual-representative}{2}} \H \h \norm{\ell}{\Ah \lambda}
    \just{\lesssim}{} \H^2 \norm{\ell}{\Ah \lambda} .
  \]
  For $\mathfrak{T}_2$,
  \[
    \norm{\ell}{\mathfrak{T}_2} \just{\lesssim}{\eqref{eq:two-level-l2}} \H^2 \norm{L^2(\dom)}{f_\lambda}
    \just{\leq}{\eqi{eq:residual-representative}{1}} \H^2 \norm{\ell}{\Ah \lambda} .
  \]
  Finally $\norm{\ell}{\mathfrak{T}_3} \lesssim \norm{\ell-1}{\GH f_\lambda - \Ph \Gh f_\lambda} \lesssim \H^2 \norm{\ell}{\Ah \lambda}$ by \eqi{eq:ass-prolongation-stab}{2}, \eqref{eq:two-level-projection} and \eqi{eq:residual-representative}{1}, which proves \eqref{eq:master-estimate}. \Aref{ass:regularity-approx} follows from the symmetry of $\ah$ and \eqref{eq:level-operator},
  \[
    \big| \ah(\lambda - \IHh \Ph \lambda, \lambda) \big|
    = \big| \braket{\Ah \lambda, \lambda - \IHh \Ph \lambda}_\ell \big|
    \leq \norm{\ell}{\Ah \lambda} \norm{\ell}{\lambda - \IHh \Ph \lambda}
    \just{\lesssim}{\eqref{eq:master-estimate}} \h^2 \norm{\ell}{\Ah \lambda}^2 . \qedhere
  \]
\end{proof}

\subsection{Convergence of the V-cycle}\label{sec:convergence}

\begin{theorem}[Convergence]\label{th:convergence}
  Let Assumptions~\ref{ass:regularity}--\ref{ass:patches} hold, let $\IHh \in \{\IUh, \IRh\}$ and let $\Sh$ be the patch smoother \eqref{def:patch-smoother} with $\omega < 2/\NO$. Then there are constants $C_R$ and $C_S$, independent of $\ell$ and of $L$, such that
  \begin{enumerate}
    \item if $m_\ell = m > 2 C_R C_S$ on every level, the V-cycle \eqref{eq:vcycle} contracts as in \eqref{eq:vcycle-contraction}, with $\delta = C_R C_S / (m - C_R C_S) < 1$;
    \item if $m_\ell$ obeys \eqref{eq:smoothing-steps-condition}, the variable V-cycle satisfies \eqref{eq:variable-vcycle}, so that $\kappa(\Bh \Ah) \leq \eta_1 / \eta_0$ is bounded.
  \end{enumerate}
\end{theorem}
\begin{proof}
  \Aref{ass:prolongation} holds for both prolongations by Theorem~\ref{th:prolongation-verification}, Lemma~\ref{lem:levelnorm-stability} and  Lemma~\ref{lem:approximation-verification}, and \Aref{ass:smoother} for the patch families by Lemma~\ref{lem:stable-decomposition} and Lemma~\ref{lem:interaction}. \Aref{ass:smoothing} then follows from Theorem~\ref{th:smoother-cond-number}, \Aref{ass:bounded} from Lemma~\ref{lem:bounded-correction} and \Aref{ass:regularity-approx} from Theorem~\ref{th:approximation}. Theorems~\ref{th:vcycle} and \ref{th:variable-vcycle} give the two claims.
\end{proof}

For the spaces \eqref{eq:interface-space} of \sect{sec:spaces}, \Aref{ass:face-spaces} holds by Lemma~\ref{lem:design-property}. Both bounds are uniform in the mesh size and in the number of levels. \sect{sec:results} measures them.

\section{Numerical results}
\label{sec:results}

We now test the solver of \sect{sec:method}, verifying that the iteration counts are bounded independently of the hierarchy depth $L$ and of the fine mesh size, for both prolongations and all patch smoothers.
\begin{figure}[ht]
  \centering
  \includegraphics[width=0.3\textwidth]{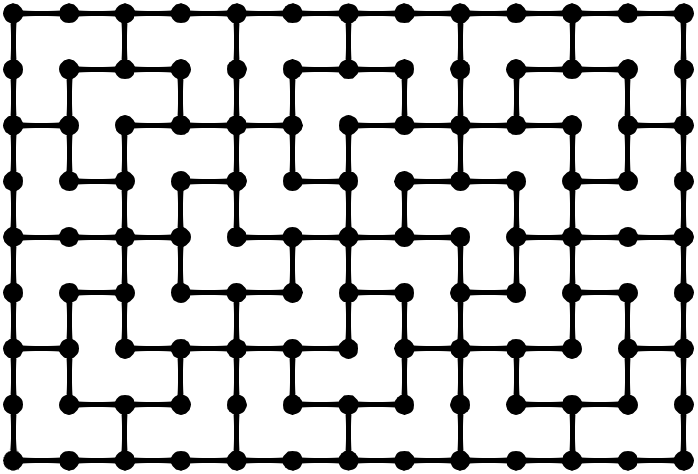} \hspace{2em}
  \includegraphics[width=0.3\textwidth]{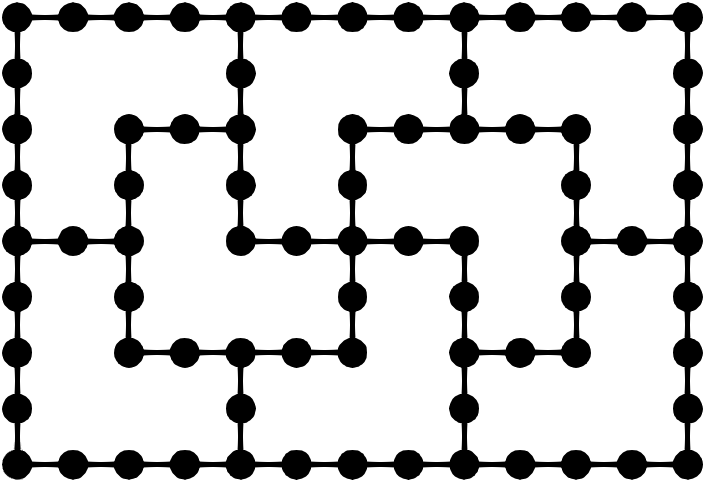}\\
  \vspace{2em}
  \includegraphics[width=0.25\textwidth]{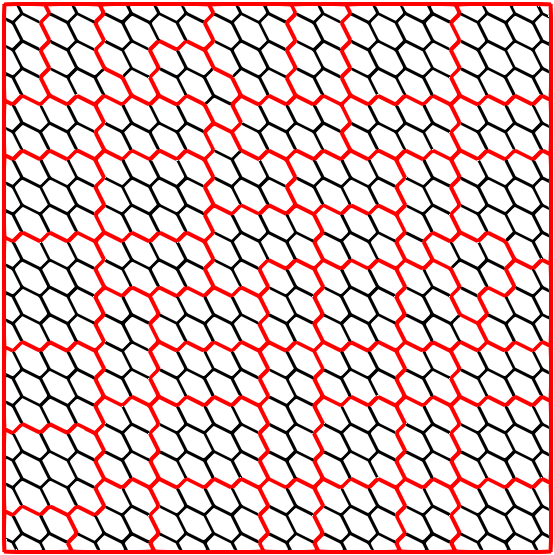}
  \includegraphics[width=0.25\textwidth]{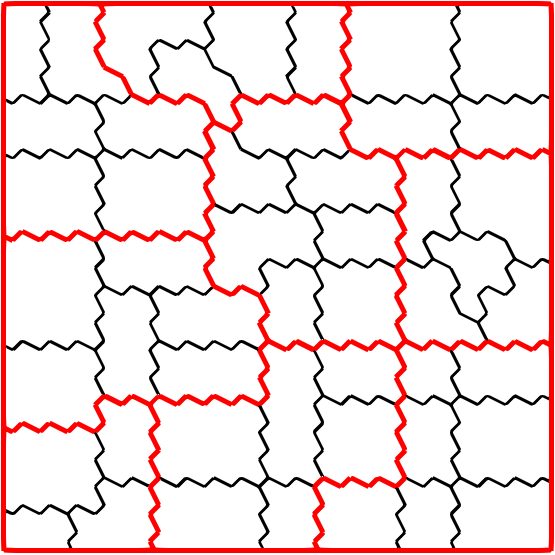}
  \includegraphics[width=0.25\textwidth]{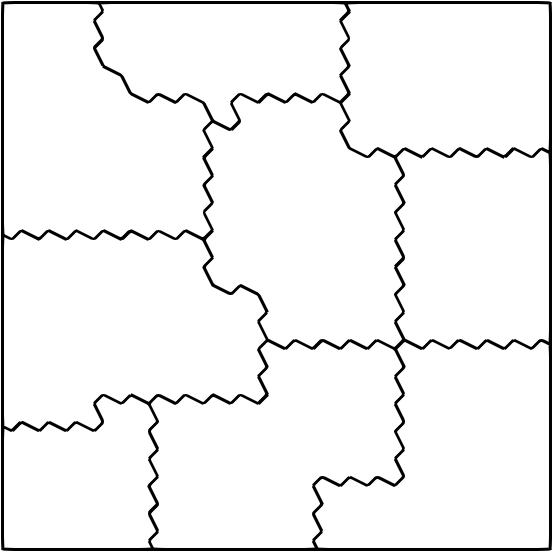}
  \caption{Illustration of the considered families of meshes in 2D: Rep-tile meshes (top row), and agglomerated Voronoi meshes (bottom row). Each row shows consecutive levels in the mesh hierarchy, with finer meshes on the left and coarser meshes on the right. Note that interfaces between coarse cells retain all vertices from the agglomeration process.}
  \label{fig:meshes}
\end{figure}

In two dimensions, we consider three different families of mesh hierarchies:

\textbf{Cartesian meshes}: Given the unit $d$-cube $\left[0,1\right]^d$ and an initial Cartesian mesh with $n$ cells in each direction, where $n$ is even, we generate a coarser mesh by agglomerating blocks of $2 \times 2$ quadrilateral elements into identical octagonal elements. The resulting mesh has $n/2$ octagonal elements in each direction, arranged in a Cartesian fashion. Note that this process differs from the one described in \cite[Section 3.1]{DiPietro2021} as we do not coarsen the octagon faces into quadrilaterals, but rather keep the extra vertices. This process can be repeated until the desired number of levels is reached, where in each level $\ell$ the elements can be seen as squares whose edges have been subdivided into $2^{(\ell-1)}$ segments. 

\textbf{Rep-tile meshes}: Next, we consider meshes composed of rep-tiles, i.e., polygons that can be subdivided into smaller copies of themselves. This allows us to test our algorithm with meshes having non-straight interfaces while keeping a regular coarsening strategy. We consider a rectangular domain $\left[0,3/2\right]\times\left[0,1\right]$ tiled procedurally with L-shaped polygons. The coarsening is performed by grouping four identical L-shaped elements into a larger L-shaped element, as illustrated in Figure \ref{fig:meshes}. This process can be repeated to generate multiple levels of coarsening. As with the above Cartesian meshes, straight interfaces are not coarsened.

\textbf{Agglomerated Voronoi meshes}: Finally, we consider mesh hierarchies obtained by arbitrary agglomeration of polygonal meshes. Starting from a Cartesian mesh of the unit $d$-cube $\left[0,1\right]^d$, we subdivide each quadrilateral into two identical triangles. The polygonal mesh is then generated by computing the Voronoi tessellation of the simplexified mesh, associating each Voronoi cell with a Cartesian node. The coarsening is performed by agglomerating connected clusters of cells into arbitrary polygonal elements. Note that we do not require that the resulting elements are convex or shape-regular; these hierarchies thus deliberately exceed the mesh assumptions of the theory (Assumption \ref{ass:geometry}), serving as a robustness test. Finding optimal agglomeration strategies for polytopal meshes remains an open research question \cite{FEDER2025,ANTONIETTI2026}. Within this work, the colouring of the mesh into clusters is performed using a hierarchical clustering algorithm with complete-linkage clustering and optimal branch ordering \cite{barjoseph2001}, provided by the Julia package \texttt{Clustering.jl} \cite{Clustering.jl}. Several other clustering strategies, including K-means, R-trees and graph-based methods (using Metis) were also considered but produced worse agglomerates when compared visually. The coarsening process is illustrated in Figure \ref{fig:meshes}.

In three dimensions, we limit our study to Cartesian meshes of the unit $d$-cube $\left[0,1\right]^d$, coarsened in a similar way as in two dimensions by agglomerating blocks of $2 \times 2 \times 2$ polyhedral elements. 

On each mesh family, we solve the Poisson problem \eqref{eq:poisson} with unit diffusion and source term
$$
f(x,y) = -\Delta(\sin(2 \pi x)  \sin(2 \pi y) x (x-L_x) y (y-L_y)),
$$
where $(L_x, L_y) = (1, 1)$ for unit square domains and $(L_x, L_y) = (3/2, 1)$ for rep-tile domains.

The resulting linear systems are solved using FGMRES \cite{saad1993} preconditioned with a single iteration of the Multigrid V-cycle from \sect{sec:method}. Each V-cycle uses $m_\ell = 5$ pre- and post-smoothing steps with a constant damping parameter $\omega = 0.2$, and a direct LU solve at the coarsest level. The damping complies with the condition $\omega < 2/\NO$ of Theorem \ref{th:smoother-cond-number} for at most five colours; the threshold on $m_\ell$ in Theorem \ref{th:convergence} involves constants that we do not track explicitly, and the level- and mesh-robustness observed below indicates it is met in practice. All experiments are implemented in Julia \cite{Julia} within the open source Gridap ecosystem \cite{Gridap,Verdugo2022}, making use of the recent library expansion for hybrid discretisations \cite{GridapHybrid} and the Julia-native solvers provided by GridapSolvers \cite{GridapSolvers}; the features of \cite{GridapHybrid} are available from Gridap \texttt{v0.19.8}.

We test two prolongation operators, $\IRh$ and $\IUh$ from Definition \ref{def:reconstruction-prolongation}, and two patch smoothers: face-star patch (FP) and vertex-star patch (VP). In 3D, we also test the edge-star patch (EP) smoother. For each family of meshes, we create hierarchies with $L = 2, \ldots, 6$ levels for different finest mesh sizes $n_c$ (number of cells in each direction for the finest mesh). We use the \ac{hho} spaces from \sect{sec:setting} with $k \in \{0, 1, 2\}$. Table \ref{tab:ndofs} reports the number of skeletal DOFs (interface unknowns $\Uhb$) after static condensation eliminates cell unknowns for different types of meshes in 2D and 3D.

For each combination of parameters, we report the number of FGMRES iterations required to reduce the relative residual by a factor of $10^{-8}$ for different depths $L$ and finest mesh sizes $n_c$. The results are presented in Table \ref{tab:niter}. We do not run the case $L=6$ for the smallest agglomeration-based mesh in 2D for lack of cells to coarsen. We also do not show the vertex-star patch case for $k = 2$ and the largest mesh in 3D, which does not converge for the current choice of damping parameter $\omega$ (increased overlapping of the vertex-patch and hidden $k$-dependence of the constants require more damping to achieve convergence).
For fixed polynomial order $k$ and finest mesh size $n_c$, the iteration counts remain stable for all $L = 2, \ldots, 6$. Similarly for fixed $k$ and $L$, as $n_c$ increases the iteration counts remain constant, validating $h$-independence of the preconditioner for both patches and prolongation operators. Although not covered in the theory, the preconditioner also appears to be $p$-robust. The vertex-star patch smoother consistently outperforms the face-star patch, while being more computationally expensive. Both prolongation operators $\IRh$ and $\IUh$ exhibit similar performance. These results confirm Theorem~\ref{th:convergence}.

\begin{table}[p]
  \centering
  \scriptsize
  \setlength{\tabcolsep}{3pt}
  \renewcommand{\arraystretch}{1.05}

  \caption{Number of skeletal DOFs (interface unknowns $\Uhb$) after static condensation.}
  \label{tab:ndofs}

  \begin{subtable}[t]{0.49\textwidth}
    \centering
    \begin{tabular}{ll!{\color{lightgray}\vrule width 0.4pt}cccccc}
      \toprule
      \multirow{2}{*}{$n_c$} & \multirow{2}{*}{$k$} & \multicolumn{6}{c}{$\ell$} \\
      \cmidrule{3-8}
      & & 1 & 2 & 3 & 4 & 5 & 6 \\
      \midrule
      \multirow{3}{*}{4096} & 0 & 8064 & 1984 & 480 & 112 & 24 & 4 \\
      & 1 & 16128 & 3968 & 960 & 224 & 48 & 8 \\
      & 2 & 24192 & 5952 & 1440 & 336 & 72 & 12 \\
      \midrule
      \multirow{3}{*}{16384} & 0 & 32512 & 8064 & 1984 & 480 & 112 & 24 \\
      & 1 & 65024 & 16128 & 3968 & 960 & 224 & 48 \\
      & 2 & 97536 & 24192 & 5952 & 1440 & 336 & 72 \\
      \midrule
      \multirow{3}{*}{65536} & 0 & 130560 & 32512 & 8064 & 1984 & 480 & 112 \\
      & 1 & 261120 & 65024 & 16128 & 3968 & 960 & 224 \\
      & 2 & 391680 & 97536 & 24192 & 5952 & 1440 & 336 \\
      \bottomrule
    \end{tabular}
    \subcaption{2D Cartesian meshes.}
    \label{tab:ndofs_d_2_ahho}
  \end{subtable}\hfill
  \begin{subtable}[t]{0.49\textwidth}
    \centering
    \begin{tabular}{ll!{\color{lightgray}\vrule width 0.4pt}cccccc}
      \toprule
      \multirow{2}{*}{$n_c$} & \multirow{2}{*}{$k$} & \multicolumn{6}{c}{$\ell$} \\
      \cmidrule{3-8}
      & & 1 & 2 & 3 & 4 & 5 & 6 \\
      \midrule
      \multirow{3}{*}{32768} & 0 & 130432 & 28512 & 7088 & 1752 & 428 & 102 \\
      & 1 & 260864 & 57024 & 14176 & 3504 & 856 & 204 \\
      & 2 & 391296 & 85536 & 21264 & 5256 & 1284 & 306 \\
      \midrule
      \multirow{3}{*}{131072} & 0 & 523008 & 114368 & 28512 & 7088 & 1752 & 428 \\
      & 1 & 1046016 & 228736 & 57024 & 14176 & 3504 & 856 \\
      & 2 & 1569024 & 343104 & 85536 & 21264 & 5256 & 1284 \\
      \midrule
      \multirow{3}{*}{524288} & 0 & 2094592 & 458112 & 114368 & 28512 & 7088 & 1752 \\
      & 1 & 4189184 & 916224 & 228736 & 57024 & 14176 & 3504 \\
      & 2 & 6283776 & 1374336 & 343104 & 85536 & 21264 & 5256 \\
      \bottomrule
    \end{tabular}
    \subcaption{2D Rep-tile meshes.}
    \label{tab:ndofs_d_2_polyahho}
  \end{subtable}

  \vspace{0.75em}

  \begin{subtable}[t]{0.49\textwidth}
    \centering
    \begin{tabular}{ll!{\color{lightgray}\vrule width 0.4pt}cccccc}
      \toprule
      \multirow{2}{*}{$n_c$} & \multirow{2}{*}{$k$} & \multicolumn{6}{c}{$\ell$} \\
      \cmidrule{3-8}
      & & 1 & 2 & 3 & 4 & 5 & 6 \\
      \midrule
      \multirow{3}{*}{10201} & 0 & 30200 & 6688 & 884 & 112 & 10 & - \\
      & 1 & 60400 & 17805 & 2601 & 335 & 30 & - \\
      & 2 & 90600 & 30625 & 5104 & 668 & 60 & - \\
      \midrule
      \multirow{3}{*}{19881} & 0 & 59080 & 12957 & 1717 & 243 & 32 & 10 \\
      & 1 & 118160 & 34623 & 5041 & 724 & 96 & 30 \\
      & 2 & 177240 & 59827 & 9899 & 1443 & 192 & 60 \\
      \midrule
      \multirow{3}{*}{40401} & 0 & 120400 & 26067 & 3461 & 419 & 66 & 10 \\
      & 1 & 240800 & 69607 & 10251 & 1256 & 198 & 30 \\
      & 2 & 361200 & 120679 & 20221 & 2509 & 396 & 60 \\
      \bottomrule
    \end{tabular}
    \subcaption{2D agglomerated Voronoi meshes.}
    \label{tab:ndofs_d_2_aggmesh}
  \end{subtable}\hfill
  \begin{subtable}[t]{0.49\textwidth}
    \centering
    \begin{tabular}{ll!{\color{lightgray}\vrule width 0.4pt}cccc}
      \toprule
      \multirow{2}{*}{$n_c$} & \multirow{2}{*}{$k$} & \multicolumn{4}{c}{$\ell$} \\
      \cmidrule{3-6}
      & & 1 & 2 & 3 & 4 \\
      \midrule
      \multirow{3}{*}{4096} & 0 & 11520 & 1344 & 144 & 12 \\
      & 1 & 34560 & 4032 & 432 & 36 \\
      & 2 & 69120 & 8064 & 864 & 72 \\
      \midrule
      \multirow{3}{*}{32768} & 0 & 95232 & 11520 & 1344 & 144 \\
      & 1 & 285696 & 34560 & 4032 & 432 \\
      & 2 & 571392 & 69120 & 8064 & 864 \\
      \midrule
      \multirow{3}{*}{262144} & 0 & 774144 & 95232 & 11520 & 1344 \\
      & 1 & 2322432 & 285696 & 34560 & 4032 \\
      & 2 & 4644864 & 571392 & 69120 & 8064 \\
      \bottomrule
    \end{tabular}
    \subcaption{3D Cartesian meshes.}
    \label{tab:ndofs_d_3_ahho}
  \end{subtable}

\end{table}

\newlength{\nitermax}
\newsavebox{\niterA}\newsavebox{\niterB}\newsavebox{\niterC}\newsavebox{\niterD}

\newcommand{\nitersetw}[4]{%
  \setlength{\tabcolsep}{\dimexpr 2pt + (#3 - \wd#1)/#2\relax}%
  \sbox#1{\input{#4}}%
}

\begin{table}[htbp]
  \centering
  \scriptsize
  \setlength{\tabcolsep}{2pt}
  \sbox\niterA{\begin{tabular}{cc!{\color{lightgray}\vrule width 0.4pt}ccccc!{\color{lightgray}\vrule width 0.4pt}ccccc!{\color{lightgray}\vrule width 0.4pt}ccccc!{\color{lightgray}\vrule width 0.4pt}ccccc}
  \toprule
  \multirow{3}{*}{$n_c$} & \multirow{3}{*}{$k$} &
  \multicolumn{10}{c}{FP} &
  \multicolumn{10}{c}{VP} \\
  \cmidrule(lr){3-12}\cmidrule(lr){13-22}
  & &
  \multicolumn{5}{c}{$I^{R}_{\ell-1}$} & \multicolumn{5}{c}{$I^{U}_{\ell-1}$} &
  \multicolumn{5}{c}{$I^{R}_{\ell-1}$} & \multicolumn{5}{c}{$I^{U}_{\ell-1}$} \\
  \cmidrule(lr){3-7}\cmidrule(lr){8-12}\cmidrule(lr){13-17}\cmidrule(lr){18-22}
  & &
  2 & 3 & 4 & 5 & 6 & 2 & 3 & 4 & 5 & 6 &
  2 & 3 & 4 & 5 & 6 & 2 & 3 & 4 & 5 & 6 \\
  \midrule
  4096  & 0 & 13 & 14 & 14 & 13 & 12 & 13 & 14 & 14 & 13 & 12 & 9 & 9 & 9 & 9 & 9 & 9 & 9 & 9 & 9 & 9 \\
        & 1 & 9 & 10 & 11 & 12 & 12 & 11 & 12 & 13 & 13 & 12 & 6 & 6 & 6 & 6 & 6 & 6 & 6 & 6 & 6 & 6 \\
        & 2 & 12 & 13 & 14 & 14 & 14 & 14 & 15 & 15 & 16 & 16 & 7 & 7 & 7 & 7 & 7 & 7 & 7 & 7 & 7 & 7 \\
  \cmidrule{1-22}
  16384 & 0 & 13 & 13 & 14 & 14 & 13 & 13 & 13 & 14 & 14 & 13 & 9 & 9 & 9 & 9 & 9 & 9 & 9 & 9 & 9 & 9 \\
        & 1 & 9 & 10 & 11 & 12 & 12 & 11 & 12 & 13 & 13 & 13 & 5 & 6 & 6 & 6 & 6 & 6 & 6 & 6 & 6 & 6 \\
        & 2 & 12 & 13 & 13 & 14 & 14 & 14 & 15 & 15 & 16 & 16 & 7 & 7 & 7 & 7 & 7 & 7 & 7 & 7 & 7 & 7 \\
  \cmidrule{1-22}
  65536 & 0 & 12 & 13 & 15 & 15 & 14 & 12 & 13 & 15 & 15 & 14 & 8 & 9 & 9 & 9 & 9 & 8 & 9 & 9 & 9 & 9 \\
        & 1 & 9 & 10 & 11 & 11 & 12 & 11 & 12 & 12 & 13 & 13 & 5 & 6 & 6 & 6 & 6 & 6 & 6 & 6 & 6 & 6 \\
        & 2 & 12 & 12 & 13 & 14 & 14 & 13 & 14 & 15 & 16 & 16 & 6 & 7 & 7 & 7 & 7 & 6 & 7 & 7 & 7 & 7 \\
  \bottomrule
\end{tabular}
}
  \sbox\niterB{\begin{tabular}{cc!{\color{lightgray}\vrule width 0.4pt}ccccc!{\color{lightgray}\vrule width 0.4pt}ccccc!{\color{lightgray}\vrule width 0.4pt}ccccc!{\color{lightgray}\vrule width 0.4pt}ccccc}
  \toprule
  \multirow{3}{*}{$n_c$} & \multirow{3}{*}{$k$} &
  \multicolumn{10}{c}{FP} &
  \multicolumn{10}{c}{VP} \\
  \cmidrule(lr){3-12}\cmidrule(lr){13-22}
  & &
  \multicolumn{5}{c}{$I^{R}_{\ell-1}$} & \multicolumn{5}{c}{$I^{U}_{\ell-1}$} &
  \multicolumn{5}{c}{$I^{R}_{\ell-1}$} & \multicolumn{5}{c}{$I^{U}_{\ell-1}$} \\
  \cmidrule(lr){3-7}\cmidrule(lr){8-12}\cmidrule(lr){13-17}\cmidrule(lr){18-22}
  & &
  2 & 3 & 4 & 5 & 6 & 2 & 3 & 4 & 5 & 6 &
  2 & 3 & 4 & 5 & 6 & 2 & 3 & 4 & 5 & 6 \\
  \midrule
  32768  & 0 & 15 & 16 & 16 & 16 & 15 & 15 & 16 & 16 & 16 & 15 & 9 & 10 & 10 & 10 & 10 & 9 & 10 & 10 & 10 & 10 \\
         & 1 & 15 & 15 & 15 & 15 & 15 & 16 & 16 & 16 & 16 & 16 & 7 & 8 & 8 & 8 & 8 & 8 & 8 & 8 & 8 & 8 \\
         & 2 & 12 & 12 & 12 & 13 & 13 & 14 & 14 & 15 & 15 & 15 & 6 & 6 & 6 & 6 & 6 & 6 & 6 & 6 & 6 & 6 \\
  \cmidrule{1-22}
  131072 & 0 & 15 & 16 & 16 & 16 & 16 & 15 & 16 & 16 & 16 & 16 & 9 & 10 & 10 & 10 & 10 & 9 & 10 & 10 & 10 & 10 \\
         & 1 & 15 & 15 & 15 & 15 & 15 & 16 & 16 & 16 & 16 & 16 & 7 & 7 & 8 & 8 & 8 & 8 & 8 & 8 & 8 & 8 \\
         & 2 & 12 & 12 & 12 & 13 & 13 & 14 & 14 & 15 & 15 & 15 & 6 & 6 & 6 & 6 & 6 & 6 & 6 & 6 & 6 & 6 \\
  \cmidrule{1-22}
  524288 & 0 & 15 & 16 & 16 & 16 & 16 & 15 & 16 & 16 & 16 & 16 & 9 & 9 & 10 & 10 & 10 & 9 & 10 & 10 & 10 & 10 \\
         & 1 & 15 & 15 & 15 & 15 & 15 & 16 & 16 & 16 & 17 & 17 & 7 & 7 & 8 & 8 & 8 & 8 & 8 & 8 & 8 & 8 \\
         & 2 & 12 & 12 & 12 & 13 & 13 & 14 & 14 & 15 & 15 & 15 & 6 & 6 & 6 & 6 & 6 & 6 & 6 & 6 & 6 & 6 \\
  \bottomrule
\end{tabular}
}
  \sbox\niterC{\begin{tabular}{cc!{\color{lightgray}\vrule width 0.4pt}ccccc!{\color{lightgray}\vrule width 0.4pt}ccccc!{\color{lightgray}\vrule width 0.4pt}ccccc!{\color{lightgray}\vrule width 0.4pt}ccccc}
  \toprule
  \multirow{3}{*}{$n_c$} & \multirow{3}{*}{$k$} &
  \multicolumn{10}{c}{FP} &
  \multicolumn{10}{c}{VP} \\
  \cmidrule(lr){3-12}\cmidrule(lr){13-22}
  & &
  \multicolumn{5}{c}{$I^{R}_{\ell-1}$} & \multicolumn{5}{c}{$I^{U}_{\ell-1}$} &
  \multicolumn{5}{c}{$I^{R}_{\ell-1}$} & \multicolumn{5}{c}{$I^{U}_{\ell-1}$} \\
  \cmidrule(lr){3-7}\cmidrule(lr){8-12}\cmidrule(lr){13-17}\cmidrule(lr){18-22}
  & &
  2 & 3 & 4 & 5 & 6 & 2 & 3 & 4 & 5 & 6 &
  2 & 3 & 4 & 5 & 6 & 2 & 3 & 4 & 5 & 6 \\
  \midrule
  10201  & 0 & 27 & 28 & 26 & 24 & - & 29 & 31 & 30 & 26 & - & 19 & 19 & 18 & 17 & - & 19 & 20 & 19 & 19 & - \\
         & 1 & 28 & 26 & 26 & 26 & - & 26 & 27 & 26 & 24 & - & 17 & 17 & 17 & 17 & - & 17 & 17 & 17 & 17 & - \\
         & 2 & 29 & 31 & 30 & 30 & - & 28 & 30 & 30 & 28 & - & 15 & 15 & 16 & 16 & - & 15 & 15 & 15 & 15 & - \\
  \cmidrule{1-22}
  19881  & 0 & 28 & 28 & 27 & 25 & 25 & 31 & 32 & 31 & 27 & 26 & 19 & 19 & 18 & 18 & 18 & 21 & 21 & 20 & 19 & 19 \\
         & 1 & 28 & 27 & 27 & 27 & 27 & 27 & 28 & 28 & 25 & 24 & 17 & 17 & 17 & 18 & 18 & 18 & 17 & 17 & 17 & 17 \\
         & 2 & 30 & 32 & 31 & 31 & 31 & 29 & 30 & 29 & 28 & 27 & 15 & 16 & 16 & 16 & 16 & 15 & 15 & 15 & 16 & 16 \\
  \cmidrule{1-22}
  40401  & 0 & 19 & 19 & 20 & 19 & 18 & 20 & 21 & 22 & 20 & 19 & 14 & 13 & 13 & 13 & 13 & 14 & 14 & 15 & 14 & 14 \\
         & 1 & 19 & 18 & 19 & 19 & 19 & 19 & 19 & 20 & 18 & 18 & 11 & 11 & 11 & 11 & 11 & 12 & 12 & 12 & 12 & 12 \\
         & 2 & 20 & 21 & 22 & 22 & 22 & 19 & 20 & 21 & 21 & 20 & 10 & 10 & 10 & 10 & 10 & 10 & 10 & 10 & 10 & 10 \\
  \bottomrule
\end{tabular}
}
  \sbox\niterD{\begin{tabular}{cc!{\color{lightgray}\vrule width 0.4pt}ccc!{\color{lightgray}\vrule width 0.4pt}ccc!{\color{lightgray}\vrule width 0.4pt}ccc!{\color{lightgray}\vrule width 0.4pt}ccc!{\color{lightgray}\vrule width 0.4pt}ccc!{\color{lightgray}\vrule width 0.4pt}ccc}
  \toprule
  \multirow{3}{*}{$n_c$} & \multirow{3}{*}{$k$} &
  \multicolumn{6}{c}{FP} &
  \multicolumn{6}{c}{EP} &
  \multicolumn{6}{c}{VP} \\
  \cmidrule(lr){3-8}\cmidrule(lr){9-14}\cmidrule(lr){15-20}
  & &
  \multicolumn{3}{c}{$I^{R}_{\ell-1}$} & \multicolumn{3}{c}{$I^{U}_{\ell-1}$} &
  \multicolumn{3}{c}{$I^{R}_{\ell-1}$} & \multicolumn{3}{c}{$I^{U}_{\ell-1}$} &
  \multicolumn{3}{c}{$I^{R}_{\ell-1}$} & \multicolumn{3}{c}{$I^{U}_{\ell-1}$} \\
  \cmidrule(lr){3-5}\cmidrule(lr){6-8}\cmidrule(lr){9-11}\cmidrule(lr){12-14}\cmidrule(lr){15-17}\cmidrule(lr){18-20}
  & &
  2 & 3 & 4 & 2 & 3 & 4 &
  2 & 3 & 4 & 2 & 3 & 4 &
  2 & 3 & 4 & 2 & 3 & 4 \\
  \midrule
  4096   & 0 & 18 & 18 & 17 & 18 & 18 & 17 & 8 & 8 & 8 & 8 & 8 & 8 & 12 & 11 & 11 & 12 & 11 & 11 \\
         & 1 & 13 & 13 & 13 & 15 & 14 & 13 & 5 & 5 & 5 & 6 & 6 & 6 & 6 & 6 & 6 & 7 & 7 & 7 \\
         & 2 & 14 & 15 & 14 & 19 & 18 & 18 & 5 & 6 & 6 & 6 & 6 & 6 & 7 & 7 & 7 & 7 & 7 & 7 \\
  \cmidrule{1-20}
  32768  & 0 & 20 & 22 & 19 & 20 & 22 & 19 & 9 & 9 & 9 & 9 & 9 & 9 & 12 & 13 & 12 & 12 & 13 & 12 \\
         & 1 & 12 & 13 & 13 & 16 & 15 & 14 & 5 & 5 & 5 & 6 & 6 & 6 & 6 & 6 & 6 & 8 & 8 & 8 \\
         & 2 & 14 & 15 & 15 & 20 & 19 & 19 & 6 & 6 & 6 & 6 & 6 & 6 & 7 & 7 & 7 & 8 & 8 & 8 \\
  \cmidrule{1-20}
  262144 & 0 & 21 & 25 & 23 & 21 & 25 & 23 & 9 & 9 & 9 & 9 & 9 & 9 & 13 & 14 & 13 & 13 & 14 & 13 \\
         & 1 & 12 & 12 & 13 & 16 & 16 & 15 & 5 & 5 & 5 & 6 & 6 & 6 & 6 & 6 & 6 & 8 & 8 & 8 \\
         & 2 & 14 & 14  & 15  & 20 & 19 & 19  & 6 & 6  & 5  & 5 & 6 & 6  & - & -  & -  & - & -  & - \\
  \bottomrule
\end{tabular}
}
  \setlength{\nitermax}{\wd\niterA}
  \ifdim\wd\niterB>\nitermax \setlength{\nitermax}{\wd\niterB}\fi
  \ifdim\wd\niterC>\nitermax \setlength{\nitermax}{\wd\niterC}\fi
  \ifdim\wd\niterD>\nitermax \setlength{\nitermax}{\wd\niterD}\fi
  \setlength{\tabcolsep}{\dimexpr 2pt + (\nitermax - \wd\niterA)/44\relax}%
  \sbox\niterA{\begin{tabular}{cc!{\color{lightgray}\vrule width 0.4pt}ccccc!{\color{lightgray}\vrule width 0.4pt}ccccc!{\color{lightgray}\vrule width 0.4pt}ccccc!{\color{lightgray}\vrule width 0.4pt}ccccc}
  \toprule
  \multirow{3}{*}{$n_c$} & \multirow{3}{*}{$k$} &
  \multicolumn{10}{c}{FP} &
  \multicolumn{10}{c}{VP} \\
  \cmidrule(lr){3-12}\cmidrule(lr){13-22}
  & &
  \multicolumn{5}{c}{$I^{R}_{\ell-1}$} & \multicolumn{5}{c}{$I^{U}_{\ell-1}$} &
  \multicolumn{5}{c}{$I^{R}_{\ell-1}$} & \multicolumn{5}{c}{$I^{U}_{\ell-1}$} \\
  \cmidrule(lr){3-7}\cmidrule(lr){8-12}\cmidrule(lr){13-17}\cmidrule(lr){18-22}
  & &
  2 & 3 & 4 & 5 & 6 & 2 & 3 & 4 & 5 & 6 &
  2 & 3 & 4 & 5 & 6 & 2 & 3 & 4 & 5 & 6 \\
  \midrule
  4096  & 0 & 13 & 14 & 14 & 13 & 12 & 13 & 14 & 14 & 13 & 12 & 9 & 9 & 9 & 9 & 9 & 9 & 9 & 9 & 9 & 9 \\
        & 1 & 9 & 10 & 11 & 12 & 12 & 11 & 12 & 13 & 13 & 12 & 6 & 6 & 6 & 6 & 6 & 6 & 6 & 6 & 6 & 6 \\
        & 2 & 12 & 13 & 14 & 14 & 14 & 14 & 15 & 15 & 16 & 16 & 7 & 7 & 7 & 7 & 7 & 7 & 7 & 7 & 7 & 7 \\
  \cmidrule{1-22}
  16384 & 0 & 13 & 13 & 14 & 14 & 13 & 13 & 13 & 14 & 14 & 13 & 9 & 9 & 9 & 9 & 9 & 9 & 9 & 9 & 9 & 9 \\
        & 1 & 9 & 10 & 11 & 12 & 12 & 11 & 12 & 13 & 13 & 13 & 5 & 6 & 6 & 6 & 6 & 6 & 6 & 6 & 6 & 6 \\
        & 2 & 12 & 13 & 13 & 14 & 14 & 14 & 15 & 15 & 16 & 16 & 7 & 7 & 7 & 7 & 7 & 7 & 7 & 7 & 7 & 7 \\
  \cmidrule{1-22}
  65536 & 0 & 12 & 13 & 15 & 15 & 14 & 12 & 13 & 15 & 15 & 14 & 8 & 9 & 9 & 9 & 9 & 8 & 9 & 9 & 9 & 9 \\
        & 1 & 9 & 10 & 11 & 11 & 12 & 11 & 12 & 12 & 13 & 13 & 5 & 6 & 6 & 6 & 6 & 6 & 6 & 6 & 6 & 6 \\
        & 2 & 12 & 12 & 13 & 14 & 14 & 13 & 14 & 15 & 16 & 16 & 6 & 7 & 7 & 7 & 7 & 6 & 7 & 7 & 7 & 7 \\
  \bottomrule
\end{tabular}
}%

  \setlength{\tabcolsep}{\dimexpr 2pt + (\nitermax - \wd\niterB)/44\relax}%
  \sbox\niterB{\begin{tabular}{cc!{\color{lightgray}\vrule width 0.4pt}ccccc!{\color{lightgray}\vrule width 0.4pt}ccccc!{\color{lightgray}\vrule width 0.4pt}ccccc!{\color{lightgray}\vrule width 0.4pt}ccccc}
  \toprule
  \multirow{3}{*}{$n_c$} & \multirow{3}{*}{$k$} &
  \multicolumn{10}{c}{FP} &
  \multicolumn{10}{c}{VP} \\
  \cmidrule(lr){3-12}\cmidrule(lr){13-22}
  & &
  \multicolumn{5}{c}{$I^{R}_{\ell-1}$} & \multicolumn{5}{c}{$I^{U}_{\ell-1}$} &
  \multicolumn{5}{c}{$I^{R}_{\ell-1}$} & \multicolumn{5}{c}{$I^{U}_{\ell-1}$} \\
  \cmidrule(lr){3-7}\cmidrule(lr){8-12}\cmidrule(lr){13-17}\cmidrule(lr){18-22}
  & &
  2 & 3 & 4 & 5 & 6 & 2 & 3 & 4 & 5 & 6 &
  2 & 3 & 4 & 5 & 6 & 2 & 3 & 4 & 5 & 6 \\
  \midrule
  32768  & 0 & 15 & 16 & 16 & 16 & 15 & 15 & 16 & 16 & 16 & 15 & 9 & 10 & 10 & 10 & 10 & 9 & 10 & 10 & 10 & 10 \\
         & 1 & 15 & 15 & 15 & 15 & 15 & 16 & 16 & 16 & 16 & 16 & 7 & 8 & 8 & 8 & 8 & 8 & 8 & 8 & 8 & 8 \\
         & 2 & 12 & 12 & 12 & 13 & 13 & 14 & 14 & 15 & 15 & 15 & 6 & 6 & 6 & 6 & 6 & 6 & 6 & 6 & 6 & 6 \\
  \cmidrule{1-22}
  131072 & 0 & 15 & 16 & 16 & 16 & 16 & 15 & 16 & 16 & 16 & 16 & 9 & 10 & 10 & 10 & 10 & 9 & 10 & 10 & 10 & 10 \\
         & 1 & 15 & 15 & 15 & 15 & 15 & 16 & 16 & 16 & 16 & 16 & 7 & 7 & 8 & 8 & 8 & 8 & 8 & 8 & 8 & 8 \\
         & 2 & 12 & 12 & 12 & 13 & 13 & 14 & 14 & 15 & 15 & 15 & 6 & 6 & 6 & 6 & 6 & 6 & 6 & 6 & 6 & 6 \\
  \cmidrule{1-22}
  524288 & 0 & 15 & 16 & 16 & 16 & 16 & 15 & 16 & 16 & 16 & 16 & 9 & 9 & 10 & 10 & 10 & 9 & 10 & 10 & 10 & 10 \\
         & 1 & 15 & 15 & 15 & 15 & 15 & 16 & 16 & 16 & 17 & 17 & 7 & 7 & 8 & 8 & 8 & 8 & 8 & 8 & 8 & 8 \\
         & 2 & 12 & 12 & 12 & 13 & 13 & 14 & 14 & 15 & 15 & 15 & 6 & 6 & 6 & 6 & 6 & 6 & 6 & 6 & 6 & 6 \\
  \bottomrule
\end{tabular}
}%

  \setlength{\tabcolsep}{\dimexpr 2pt + (\nitermax - \wd\niterC)/44\relax}%
  \sbox\niterC{\begin{tabular}{cc!{\color{lightgray}\vrule width 0.4pt}ccccc!{\color{lightgray}\vrule width 0.4pt}ccccc!{\color{lightgray}\vrule width 0.4pt}ccccc!{\color{lightgray}\vrule width 0.4pt}ccccc}
  \toprule
  \multirow{3}{*}{$n_c$} & \multirow{3}{*}{$k$} &
  \multicolumn{10}{c}{FP} &
  \multicolumn{10}{c}{VP} \\
  \cmidrule(lr){3-12}\cmidrule(lr){13-22}
  & &
  \multicolumn{5}{c}{$I^{R}_{\ell-1}$} & \multicolumn{5}{c}{$I^{U}_{\ell-1}$} &
  \multicolumn{5}{c}{$I^{R}_{\ell-1}$} & \multicolumn{5}{c}{$I^{U}_{\ell-1}$} \\
  \cmidrule(lr){3-7}\cmidrule(lr){8-12}\cmidrule(lr){13-17}\cmidrule(lr){18-22}
  & &
  2 & 3 & 4 & 5 & 6 & 2 & 3 & 4 & 5 & 6 &
  2 & 3 & 4 & 5 & 6 & 2 & 3 & 4 & 5 & 6 \\
  \midrule
  10201  & 0 & 27 & 28 & 26 & 24 & - & 29 & 31 & 30 & 26 & - & 19 & 19 & 18 & 17 & - & 19 & 20 & 19 & 19 & - \\
         & 1 & 28 & 26 & 26 & 26 & - & 26 & 27 & 26 & 24 & - & 17 & 17 & 17 & 17 & - & 17 & 17 & 17 & 17 & - \\
         & 2 & 29 & 31 & 30 & 30 & - & 28 & 30 & 30 & 28 & - & 15 & 15 & 16 & 16 & - & 15 & 15 & 15 & 15 & - \\
  \cmidrule{1-22}
  19881  & 0 & 28 & 28 & 27 & 25 & 25 & 31 & 32 & 31 & 27 & 26 & 19 & 19 & 18 & 18 & 18 & 21 & 21 & 20 & 19 & 19 \\
         & 1 & 28 & 27 & 27 & 27 & 27 & 27 & 28 & 28 & 25 & 24 & 17 & 17 & 17 & 18 & 18 & 18 & 17 & 17 & 17 & 17 \\
         & 2 & 30 & 32 & 31 & 31 & 31 & 29 & 30 & 29 & 28 & 27 & 15 & 16 & 16 & 16 & 16 & 15 & 15 & 15 & 16 & 16 \\
  \cmidrule{1-22}
  40401  & 0 & 19 & 19 & 20 & 19 & 18 & 20 & 21 & 22 & 20 & 19 & 14 & 13 & 13 & 13 & 13 & 14 & 14 & 15 & 14 & 14 \\
         & 1 & 19 & 18 & 19 & 19 & 19 & 19 & 19 & 20 & 18 & 18 & 11 & 11 & 11 & 11 & 11 & 12 & 12 & 12 & 12 & 12 \\
         & 2 & 20 & 21 & 22 & 22 & 22 & 19 & 20 & 21 & 21 & 20 & 10 & 10 & 10 & 10 & 10 & 10 & 10 & 10 & 10 & 10 \\
  \bottomrule
\end{tabular}
}%

  \setlength{\tabcolsep}{\dimexpr 2pt + (\nitermax - \wd\niterD)/40\relax}%
  \sbox\niterD{\begin{tabular}{cc!{\color{lightgray}\vrule width 0.4pt}ccc!{\color{lightgray}\vrule width 0.4pt}ccc!{\color{lightgray}\vrule width 0.4pt}ccc!{\color{lightgray}\vrule width 0.4pt}ccc!{\color{lightgray}\vrule width 0.4pt}ccc!{\color{lightgray}\vrule width 0.4pt}ccc}
  \toprule
  \multirow{3}{*}{$n_c$} & \multirow{3}{*}{$k$} &
  \multicolumn{6}{c}{FP} &
  \multicolumn{6}{c}{EP} &
  \multicolumn{6}{c}{VP} \\
  \cmidrule(lr){3-8}\cmidrule(lr){9-14}\cmidrule(lr){15-20}
  & &
  \multicolumn{3}{c}{$I^{R}_{\ell-1}$} & \multicolumn{3}{c}{$I^{U}_{\ell-1}$} &
  \multicolumn{3}{c}{$I^{R}_{\ell-1}$} & \multicolumn{3}{c}{$I^{U}_{\ell-1}$} &
  \multicolumn{3}{c}{$I^{R}_{\ell-1}$} & \multicolumn{3}{c}{$I^{U}_{\ell-1}$} \\
  \cmidrule(lr){3-5}\cmidrule(lr){6-8}\cmidrule(lr){9-11}\cmidrule(lr){12-14}\cmidrule(lr){15-17}\cmidrule(lr){18-20}
  & &
  2 & 3 & 4 & 2 & 3 & 4 &
  2 & 3 & 4 & 2 & 3 & 4 &
  2 & 3 & 4 & 2 & 3 & 4 \\
  \midrule
  4096   & 0 & 18 & 18 & 17 & 18 & 18 & 17 & 8 & 8 & 8 & 8 & 8 & 8 & 12 & 11 & 11 & 12 & 11 & 11 \\
         & 1 & 13 & 13 & 13 & 15 & 14 & 13 & 5 & 5 & 5 & 6 & 6 & 6 & 6 & 6 & 6 & 7 & 7 & 7 \\
         & 2 & 14 & 15 & 14 & 19 & 18 & 18 & 5 & 6 & 6 & 6 & 6 & 6 & 7 & 7 & 7 & 7 & 7 & 7 \\
  \cmidrule{1-20}
  32768  & 0 & 20 & 22 & 19 & 20 & 22 & 19 & 9 & 9 & 9 & 9 & 9 & 9 & 12 & 13 & 12 & 12 & 13 & 12 \\
         & 1 & 12 & 13 & 13 & 16 & 15 & 14 & 5 & 5 & 5 & 6 & 6 & 6 & 6 & 6 & 6 & 8 & 8 & 8 \\
         & 2 & 14 & 15 & 15 & 20 & 19 & 19 & 6 & 6 & 6 & 6 & 6 & 6 & 7 & 7 & 7 & 8 & 8 & 8 \\
  \cmidrule{1-20}
  262144 & 0 & 21 & 25 & 23 & 21 & 25 & 23 & 9 & 9 & 9 & 9 & 9 & 9 & 13 & 14 & 13 & 13 & 14 & 13 \\
         & 1 & 12 & 12 & 13 & 16 & 16 & 15 & 5 & 5 & 5 & 6 & 6 & 6 & 6 & 6 & 6 & 8 & 8 & 8 \\
         & 2 & 14 & 14  & 15  & 20 & 19 & 19  & 6 & 6  & 5  & 5 & 6 & 6  & - & -  & -  & - & -  & - \\
  \bottomrule
\end{tabular}
}%

  \newcommand{\niterbox}[1]{\resizebox{\linewidth}{!}{#1}}

  \caption{Number of preconditioned FGMRES iterations.}
  \label{tab:niter}

  \begin{subtable}[t]{0.5\linewidth}
    \centering
    \vspace{0pt}
    \niterbox{\usebox\niterA}
    \subcaption{2D Cartesian meshes.}
    \label{tab:niter_d_2_ahho}
  \end{subtable}%
  \begin{subtable}[t]{0.5\linewidth}
    \centering
    \vspace{0pt}
    \niterbox{\usebox\niterB}
    \subcaption{2D Rep-tile meshes.}
    \label{tab:niter_d_2_polyahho}
  \end{subtable}

  \vspace{0.4em}

  \begin{subtable}[t]{0.5\linewidth}
    \centering
    \vspace{0pt}
    \niterbox{\usebox\niterC}
    \subcaption{2D agglomerated Voronoi meshes.}
    \label{tab:niter_d_2_aggmesh}
  \end{subtable}%
  \begin{subtable}[t]{0.5\linewidth}
    \centering
    \vspace{0pt}
    \niterbox{\usebox\niterD}
    \subcaption{3D Cartesian meshes.}
    \label{tab:niter_d_3_ahho}
  \end{subtable}
\end{table}

\section{Conclusion}

We have constructed and analysed a fully hybrid \ac{gmg} solver for \ac{hho} discretisations on arbitrary polytopal agglomeration hierarchies in 2D and 3D. Theorem~\ref{th:convergence} gives uniform contraction of the V-cycle and a uniform condition number for the variable V-cycle, for two prolongation operators and for the star-patch smoothers, with bounds independent of the mesh size and of the number of levels. The two ingredients that make this possible are an enriched \ac{hho} face space, which handles the curved interfaces produced by agglomeration while reducing the number of \acp{dof}, and a skeletal inner product that does not tie the levels to a common bulk space. This removes the restriction present in earlier \ac{gmg} approaches for hybrid methods, which were mostly confined to nested hierarchies with planar faces at all levels.

The numerical experiments confirm the theory and, on the arbitrary agglomerations of \sect{sec:results}, remain robust well outside the mesh assumptions under which it is proved. Extending the framework to \ac{hdg} and Weak Galerkin discretisations, and to other problems such as incompressible flows, is future work.

\section*{Acknowledgments}

The authors would like to thank Dr. Jai Tushar and Prof. Silvia Bertoluzza for insightful discussions and valuable feedback. 
This research was partially funded by the Australian Government through the Australian Research Council (project numbers DP210103092 and DP220103160). This work was also supported by computational resources provided by the Australian Government through NCI and Pawsey under the NCMAS Merit Allocation Schemes. 

\printbibliography

\end{document}